\documentclass[10pt, letterpaper, twoside]{article} 

\usepackage{amsmath, amsthm, amssymb,bbold,mathrsfs} 
\usepackage{graphicx}
\usepackage[titletoc,title]{appendix}
\usepackage{enumitem}
\setlist[itemize]{itemsep=0pt,parsep=2pt,topsep=2pt}
\setlist[enumerate]{itemsep=0pt,parsep=2pt,topsep=2pt}
\usepackage[numbers]{natbib}
\usepackage[hidelinks]{hyperref}
\usepackage{cleveref}

\newtheorem{theorem}{Theorem}[section]
\newtheorem{lemma}{Lemma}[section]
\newtheorem{cor}{Corollary}[section]
\newtheorem{prop}{Proposition}[section]
\newtheorem{assumption}{Assumption}[section]

\theoremstyle{remark}
\newtheorem{rem}{Remark}[section]

\theoremstyle{definition}
\newtheorem{example}{Example}[section]
\newtheorem{definition}{Definition}[section]

\numberwithin{equation}{section}

\def\argmin{\mathop{\arg\min}}
\def\A{\mathbb{A}}
\def\X{\mathbb{X}}
\def\E{\mathbb{E}}
\def\R{\mathbb{R}}
\def\Q{\mathbb{Q}}
\def\Pr{\mathbb{P}}
\def\C{\mathcal{C}}
\def\Cs{{\mathcal{C}_s}}
\def\D{\mathcal{D}}
\def\P{\mathcal{P}}
\def\B{\mathcal{B}}
\def\U{\mathcal{U}}
\def\Z{\mathbb{Z}}
\def\S{\mathcal{S}}

\def\imS{\S_I}
\def\timS{\tilde \S_I}
\def\imSn{\S'_I}
\def\timSn{{\tilde \S}'_I}
\def\mmS{\S_M}
\def\tmmS{{\tilde \S}_M}
\def\hmmS{{\widehat \S}_M}
\def\imPi{\Pi_I}
\def\imPin{\Pi'_I}

\def\nmid{\,|\,}
\def\grh{\text{grh} \,}
\def\ind{\mathbb{1}}
 
\def\mdp{$\text{AC}^+$}
\def\mdn{$\text{AC}^-$}
\def\ptmdp{$\widetilde{\text{AC}}^+$}
\def\ptmdn{$\widetilde{\text{AC}}^-$}
\def\tJ{{\tilde J}}

\begin{document} 
\markboth{Strategic Measures in Stochastic Control on Borel Spaces}{Strategic Measures in Stochastic Control on Borel Spaces}

\title{On Strategic Measures and Optimality Properties in Discrete-Time Stochastic Control with\\ Universally Measurable Policies\thanks{A shorter version of this paper is to appear in the journal \emph{Mathematics of Operations Research}. This research was supported by grants from DeepMind, Alberta Machine Intelligence Institute (AMII), and Alberta Innovates---Technology Futures (AITF).}}

\author{Huizhen Yu\thanks{RLAI Lab, Department of Computing Science, University of Alberta, Canada (\texttt{janey.hzyu@gmail.com})}}
\date{} 

\maketitle

\begin{center} 
\emph{Dedicated to the memory of Professor Sanjoy K.\ Mitter, an inspiring mentor.}
\end{center}
\bigskip

\begin{abstract}
This paper concerns discrete-time infinite-horizon stochastic control systems with Borel state and action spaces and universally measurable policies. We study optimization problems on strategic measures induced by the policies in these systems. The results are then applied to risk-neutral and risk-sensitive Markov decision processes, as well as their partially observable counterparts, to establish the measurability of the optimal value functions and the existence of universally measurable, randomized or nonrandomized, $\epsilon$-optimal policies, for a variety of average cost criteria and risk criteria. We also extend our analysis to a class of minimax control problems and establish similar optimality results under the axiom of analytic determinacy.
\end{abstract}

Borel spaces; universally measurable policies; strategic measures 

\clearpage 
\tableofcontents

\clearpage
\section{Introduction}

We consider discrete-time infinite-horizon stochastic control in the universal measurability framework as formulated by Shreve and Bertsekas \cite{ShrB78, ShrB79} (see also \cite[Part II]{bs}), where the state and action spaces are Borel, the control constraints have analytic graphs, and the policies are universally measurable. Our particular focus will be on Markov decision processes (MDPs), their partially observable counterparts in which the controller has imperfect information about the states, and minimax control problems in which the controller plays against an opponent.
The main results of this paper concern optimality properties of these systems, such as the measurability of the optimal value functions and the existence of measurable, randomized or nonrandomized, optimal or nearly optimal policies, under general conditions on the system dynamics and with respect to (w.r.t.) a broad range of performance criteria. 

For MDPs with \emph{Borel measurable policies}, strategic measures were analyzed and optimality results similar to ours were established a long time ago in Strauch \cite{Str-negative}, Blackwell~\cite{Blk76}, Dynkin and Yushkevich~\cite[Chaps.\ 3, 5]{DyY79}, and Feinberg~\cite{Fei82, Fei82a, Fei91, Fei96}. (Although this is outside our scope, we mention that topological properties of strategic measures have also been studied by Sch{\"a}l \cite{Sch75b} and Balder \cite{Bal89}, and by Y{\"u}ksel and Saldi recently for Borel team-decision models \cite{YuS17}.) 
For MDPs with universally measurable policies, however, an analysis of strategic measures seems to be missing in the literature. One of the purposes of this paper is to fill this gap. While the two MDP models are similar, their differences are not superficial, so our results are not mere repetitions of the established results for Borel measurable policies. In our proofs, besides arguments from the prior works, we use also the properties of analytic sets and semianalytic or universally measurable functions. Another important difference is in the notion of optimality. The prior results in \cite[Chaps.\ 3, 5]{DyY79} and \cite{Fei82, Fei82a, Str-negative} concern the existence of Borel measurable, randomized or nonrandomized ($\epsilon$-)optimal policies for \emph{almost all} states, w.r.t.\ any given finite measure on the state space.
By contrast, with universally measurable policies, we can show that there exist policies that are ($\epsilon$-)optimal \emph{everywhere}.

Our work also builds on the early work of Shreve and Bertsekas \cite{ShrB78, ShrB79} (see also \cite[Chap.~9]{bs}) on MDPs with universally measurable policies. They focused on the discounted and total cost criteria. Instead of strategic measures, they considered the marginal distributions of the state and action pairs at each stage and related the optimal control problems to optimization problems on those sequences of marginal distributions that can be induced by the policies through time. 
This method has also been used by the author \cite{Yu-tc15, Yu20} for certain total and average cost problems. But it cannot be applied when the performance criteria are not functions of those marginal state-action distributions. Such criteria include, for example, many popular risk criteria and certain long-run average cost criteria that involve costs along sample paths. To address problems with such performance criteria in the universal measurability framework is the main purpose of this paper.

For minimax control, we consider a class of problems in which the controller does not observe the opponent's actions, and we introduce an absolute continuity condition on the finite-dimensional distributions of the processes observable to the controller while the controller applies the same policy (see Assumption~\ref{cond-minimax-abscont}). This allows us to reformulate the minimax control problems as minimax optimization problems on strategic measures. We then prove an optimality result regarding the optimal value function and the ($\epsilon$-)optimal policies for the controller (see Theorem~\ref{thm-minmax-opt}) by using the properties of the set of strategic measures and a uniformization theorem of Kond{\^o} from descriptive set theory. In this result, besides the absolute continuity condition mentioned earlier and a measurability condition on the performance criterion, we assume the axiom of analytic determinacy. (The standard ZFC axioms, i.e., Zermelo--Fraenkel set theory with the axiom of choice, are assumed throughout the paper without explicit reference.) The purpose of this additional axiomatic assumption is to resolve the measurability questions that cannot be settled in the ZFC system. Previously, such an approach and a closely related axiomatic assumption (the universal measurability of $\Sigma_2^1$ sets) had been discussed and used in Maitra, Purves, and Sudderth \cite{MPS90} for Borel leavable gambling problems, and in Prikry and Sudderth \cite{PrS16} for one-stage Borel parametrized games. Further discussions about analytic determinacy (and $\Sigma_2^1$ sets) will be given in Remarks~\ref{rmk-zfc-ad}, \ref{rmk-ad} and Section~\ref{sec-sel-proof}.

We are not aware of any prior result similar to ours in the minimax control/stochastic game literature.
The majority of published results (for discounted, total, or average cost criteria) concern the validity of dynamic programming equations, the existence of stationary optimal policies, and/or minimax equalities, under various continuity/compactness conditions on the control models (see e.g., \cite{GHH02,JaN14,JaN18,Rie91} and the references therein).
Nowak \cite{Now85b} and Maitra and Sudderth \cite{MS93} placed such conditions on only one of the players. Maitra and Sudderth \cite{MS98} did not require such conditions; to avoid measurability issues, they took the approach that is based on the theory of finitely additive (instead of countably additive) probability measures. 

We mention that this paper does not aim to establish dynamic programming equations for the optimal value functions, because the strategic-measure approach alone is generally insufficient for that purpose (especially for average-cost and risk-sensitive problems). However, the optimality results it produces can serve as starting point for deriving such equations. We will demonstrate this for minimax control with risk-neutral criteria (see Section~\ref{sec-minmax-oe}). Some other related examples can be found in our earlier analyses of discounted, total-cost, and average-cost MDPs with universally measurable policies (cf.\ \cite[Thm.\ 2.3]{Yu20}, \cite[App.\ A]{Yu-tc15}, and \cite[Sec.\ 2]{Yu22}, respectively), and in Y{\"u}ksel \cite{Yuk20} for finite-horizon decentralized team-decision problems with Borel measurable policies.

This paper is organized as follows. In Section~\ref{sec-2} we give background materials about Borel-space MDPs. In Section~\ref{sec-3} we present general results regarding the measurability of various sets of strategic measures and the optimality properties of associated optimization problems. We then apply these results to risk-neutral and risk-sensitive MDPs in Section~\ref{sec-4}, where we also extend these results to partially observable problems and minimax control. We collect some of our proofs in Section~\ref{sec-5}.

\section{Borel-Space MDPs} \label{sec-2}

In this section we describe the universal measurability framework for Borel-space MDPs. We start with the definitions and important properties of certain sets/functions that underly this framework.

\subsection{Preliminaries} \label{sec-2.1}

A Borel subset of a Polish space is called a \emph{Borel space} (or \emph{standard Borel space}) \cite[Def.\ 7.7]{bs}. 
For a Borel space $X$, $\B(X)$ denotes the Borel $\sigma$-algebra on $X$, and $\P(X)$ denotes the set of all probability measures on $\B(X)$. We endow the space $\P(X)$ with the topology of weak convergence; then $\P(X)$ is also a Borel space~\cite[Sec.\ 7.4.2]{bs}. 
The \emph{universal $\sigma$-algebra} on $X$ is given by $\U(X) : = \cap_{p \in \P(X)} \B_p(X)$. Here, for each $p \in \P(X)$, $\B_p(X)$ is the $\sigma$-algebra generated by $\B(X)$ and all the subsets of $X$ with $p$-outer measure $0$, and on $\B_p(X)$, $p$ admits a unique extension, called the \emph{completion of $p$} (cf.\ \cite[Sec.\ 3.3]{Dud02}). 
The sets in $\U(X)$ and the $\U(X)$-measurable mappings on $X$, being measurable w.r.t.\ the completion of any $p \in \P(X)$, are called \emph{universally measurable}.

Let $X$ and $Y$ be Borel spaces. A function $q(\cdot \nmid \cdot): \B(Y) \times X \to [0,1]$ is called a \emph{universally measurable stochastic kernel} (resp.\ \emph{Borel measurable stochastic kernel}) on $Y$ given $X$, if for each $x \in X$, $q(\cdot \nmid x)$ is a probability measure on $\B(Y)$ and for each $B \in \B(Y)$, $q(B \nmid \cdot)$ is universally measurable (resp.\ Borel measurable). Equivalently, such a kernel can also be defined as a measurable mapping $x \mapsto q(\cdot \nmid x)$ from the space $(X, \U(X))$ (resp.\ $(X, \B(X))$) into the space $\big(\P(Y), \B(\P(Y))\big)$ of probability measures; cf.\ \cite[Def.~7.12, Prop.~7.26, Lem.~7.28]{bs}. 
To refer to the stochastic kernel, we will often use the notation $q$ or $q(dy \,|\, x)$. 

For $p \in \P(X)$ and $D \in \U(X)$, $p(D)$ is the probability of $D$ w.r.t.\ the completion of $p$. (Notation-wise, we do not distinguish between $p$ and its completion.) Let $\R : = (-\infty, + \infty)$ and $\bar \R : = [-\infty, + \infty]$. If $f: X \to \bar \R$ is universally measurable, then $\int f dp : = \int f^+ dp - \int f^- dp$, where $f^+$ ($f^-$) is the positive (negative) part of $f$ and the integration is w.r.t.\ the completion of $p$. (For summations involving extended real numbers, the convention $\infty - \infty = - \infty + \infty = \infty$ is adopted unless otherwise stated.) 
The same interpretations apply to $q(D \nmid y)$ and $ \int f(x) \, q(dx \nmid y)$, if $D$ and $f$ are as defined above and $q(dx \nmid y)$ is a Borel or universally measurable stochastic kernel on $X$ given $Y$. 

A subset of $X$ is called \emph{analytic}, if it is either empty or the image of a Borel subset of some Polish space under a Borel measurable mapping (cf.\ \cite[Prop.\ 7.41]{bs}, \cite[Sec.~13.2]{Dud02}). The complement of an analytic set is called \emph{coanalytic}. A \emph{lower semianalytic} function is a function $f: D \to \bar \R$ such that the domain $D$ is an analytic set and for every $r \in \R$, the level set $\{ x \in D \!\mid f(x) \leq r\}$ is analytic \cite[Def.\ 7.21, Lem.\ 7.30(1)]{bs}. (This definition is equivalent to that the epigraph of $f$, $\{ (x, r) \in D \times \R \!\mid   f(x) \leq r \}$, is analytic; cf.~\cite[p.~186]{bs}.) Replacing ``$\leq r$'' with ``$\geq r$'' in the preceding statement gives the definition of an \emph{upper semianalytic} function.
Borel subsets of $X$ are analytic (and coanalytic);
Borel measurable functions from $X$ into $\bar \R$ are lower (and upper) semianalytic.
The $\sigma$-algebra $\mathcal{A}(X)$ generated by analytic subsets of $X$ contains $\B(X)$ and is contained in $\U(X)$. 
Thus analytically measurable (i.e., $\mathcal{A}(X)$-measurable) sets and functions, in particular, analytic subsets of $X$ and lower (or upper) semianalytic functions on $X$, are universally measurable.

In addition, analytic sets and lower semianalytic functions have a number of important properties that play fundamental roles in the mathematical framework for Borel-space MDPs. Let us recount a few of these properties here, which will be used frequently in this paper (see the article \cite{BFO74} for a more comprehensive review, and the monograph \cite[Chap.\ 7]{bs} for an in-depth study). 
\begin{itemize}[leftmargin=0.65cm,labelwidth=!]
\item[(a)] For any analytic set $D \subset X$ and real number $r \geq 0$, $\{ p \in \P(X) \mid p (D) > r\}$ and $\{ p \in \P(X) \mid p (D) \geq r\}$ are analytic subsets of $\P(X)$ (\cite[Prop.~7.43, Cor.~7.43.1]{bs}; see also \cite[Lem.~(25)]{BFO74}).
\item[(b)] If $f: X \times Y \to \bar \R$ is lower semianalytic and $q(dy \,|\, x)$ is a Borel measurable stochastic kernel on $Y$ given $X$, then $\phi(x) : = \int f(x,y) \, q(dy \nmid x)$ is lower semianalytic on $X$ \cite[Prop.\ 7.48]{bs}.
\end{itemize}
For $D \subset X \times Y$, $\text{proj}_X(D) : = \{ x \in X \!\mid (x,y) \in D \ \text{for some} \ y \in Y \}$ is the projection of $D$ on $X$, and we call a function $\psi: \text{proj}_X(D) \to Y$ \emph{a selection of $D$} if the graph of $\psi$ is contained in $D$, i.e., $(x, \psi(x)) \in D$ for all $x \in \text{proj}_X(D)$.
\begin{itemize}[leftmargin=0.65cm,labelwidth=!]
\item[(c)] If $D \subset X \times Y$ is analytic, then $\text{proj}_X(D)$ is analytic \cite[Prop.~7.39]{bs} and $D$ admits an analytically measurable selection (the Jankov-von Neumann selection theorem \cite[Prop.\ 7.49]{bs}).
\end{itemize}
For partial minimization of a function $f$ on $X \times Y$ (that is, minimizing $f$ over $Y$ for each $x \in X$), by applying (c) to the level sets or epigraph of a lower semianalytic $f$ and using the properties of universally measurable functions, one obtains the following:
\begin{enumerate}[leftmargin=0.65cm,labelwidth=!] 
\item[(d)] If $D \subset X \times Y$ is analytic and $f: D \to \bar \R$ is lower semianalytic, then the function $f^*: \text{proj}_X(D) \to \bar \R$ defined by
\begin{equation} \label{eq-minf}
   f^*(x) = \inf_{y \in D_x} f(x, y), \quad \text{where} \ D_x = \{ y \in Y \mid (x, y) \in D \}, 
\end{equation}   
is also lower semianalytic  \cite[Prop.~7.47]{bs}. Furthermore, for each $\epsilon > 0$, there exists a universally measurable function $\psi: \text{proj}_X(D) \to Y$
such that $\psi(x) \in D_x$ for all $x \in \text{proj}_X(D)$ and 
\begin{equation} 
     f(x, \psi(x)) = f^*(x), \qquad \forall \,  x \in E^*,
\end{equation}
\begin{equation}  
 f(x, \psi(x)) \leq  \begin{cases}
        f^*(x) + \epsilon & \text{if} \ f^*(x) > - \infty, \\
        -1/\epsilon & \text{if} \ f^*(x) = - \infty,
        \end{cases} \qquad \ \ \  \forall \, x \in \text{proj}_X(D) \setminus E^*,
\end{equation}
where $E^* : = \{ x \in \text{proj}_X(D) \mid  \argmin_{y \in D_x} f(x,y) \not= \varnothing \} \in \U(X)$ (cf.\ \cite[Prop.~7.50]{bs}).
\end{enumerate}

\begin{rem} \rm \label{rmk-zfc-ad}
We will introduce later a counterpart of (d) for upper semianalytic functions as well as for a more general class of functions under the axioms of ZFC and analytic determinacy (see Prop.~\ref{prp-ext-sel} and Remark~\ref{rmk-sigma21}). Based on G{\"o}del's work \cite{God38}, it is known that this is not possible under ZFC alone (cf.\ the discussions in \cite[p.~302]{bs} and Remark~\ref{rmk-ad}(b)). 
\qed 
\end{rem}

\noindent {\bf Miscellaneous notation:} Throughout the paper, for $x \in X$, $\delta_x$ denotes the Dirac measure concentrating at $x$, and $\D(X)$ denotes the set of all Dirac measures on $X$. For $B \subset X$, $B^c: = X \setminus B$ and $\ind_B$ is the indicator function for $B$. For an event $E$ in a probability space with a probability measure $p$, we write $\ind(E)$ for the indicator of $E$ and $p\{ E\}$ for the probability of $E$.

\subsection{MDPs with Universally Measurable Policies} \label{sec-2.2}

We consider a Borel-space MDP in the universal measurability framework (cf.\ \cite[Sec.\ 8.1]{bs}). Specifically, the following will be assumed throughout the paper unless otherwise stated:
\begin{itemize}[leftmargin=0.6cm,labelwidth=!]
\item The state space $\X$ and the action space $\A$ are \emph{Borel spaces}.
\item State transitions are governed by $q(dy \nmid x, a)$, a \emph{Borel measurable} stochastic kernel on $\X$ given $\X \times \A$.
\item The control constraint is specified by a set-valued mapping $A: x \mapsto A(x)$ on $\X$, where $A(x) \subset \A$ is a nonempty set of admissible actions at the state $x$. The graph of $A(\cdot)$,
$\Gamma := \{(x, a) \mid x \in \X, a \in A(x)\},$
is \emph{analytic}.
\item The one-stage cost function $c: \Gamma \to \bar \R$ is \emph{lower semianalytic}.
\end{itemize}

Our primary interest is in infinite-horizon control problems (several performance criteria will be introduced in Section~\ref{sec-4}). 
Let $\omega : = (x_0, a_0, x_1, a_1, \ldots)$, where $x_n$ and $a_n$ denote the state and action, respectively, at the $n$th stage. For $n \geq 0$, let $h_n : = (x_0, a_0, x_1, a_1,\ldots, x_n)$ and $h'_n : = (x_0, a_0, x_1, a_1,\ldots, x_n, a_n)$. We denote the space of $\omega$ by $\Omega: = (\X \times \A)^\infty$, the space of $h_n$ by $H_n : = (\X \times \A)^n \times \X$, and the space of $h'_n$ by $H'_n : = (\X \times \A)^{n+1}$. All these spaces are endowed with the product topology to make them Borel spaces \cite[Prop.\ 7.13]{bs}. 

A \emph{universally measurable policy} (or a \emph{policy} for short) is a sequence of \emph{universally measurable} stochastic kernels, $\pi : =(\mu_0, \mu_1, \ldots)$, where for each $n \geq 0$,
$\mu_n\big(da_n \nmid h_n \big)$ is a universally measurable stochastic kernel on $\A$ given $H_n$ that satisfies the control constraint
\begin{equation}  \label{eq-control-constraint}
   \mu_n\big(A(x_n) \mid h_n \big) = 1, \qquad \forall \, h_n = (x_0, a_0, \ldots, a_{n-1}, x_n) \in H_n.
\end{equation}   
Note that the set $A(y), y \in \X$, is analytic since it is a section of the analytic set $\Gamma$ (cf.\ \cite[Prop.~7.40]{bs}), so the probability of $A(x_n)$ in the above expression is well-defined. A \emph{Borel measurable policy} is a policy that consists of Borel measurable stochastic kernels $\{\mu_n\}_{n \geq 0}$.
If for every $n \geq 0$ and every $h_n \in H_n$, $\mu_n(da_n \nmid h_n)$ is a Dirac measure, $\pi$ is called a \emph{nonrandomized} policy. 
A \emph{Markov} (resp.\ \emph{semi-Markov}) policy is a policy such that for every $n \geq 0$, as a function of $h_n$, the probability measure $\mu_n(d a_n \nmid h_n)$ depends only on $x_n$ (resp.\ $(x_0, x_n)$). If a Markov policy $\pi$ has identical stochastic kernels, i.e., $\pi = (\mu, \mu, \ldots)$, we call it a \emph{stationary} policy and write it simply as $\mu$. Likewise, if a semi-Markov policy $\pi$ satisfies that for some stochastic kernel $\tilde \mu$ on $\A$ given $\X^2$, 
$\mu_0(da_0 \nmid x_0) = \tilde \mu(da_0 \nmid x_0, x_0)$ and $\mu_n(d a_n \nmid x_0, x_n)  = \tilde \mu( da_n \nmid x_0, x_n)$ for all $x_0, x_n \in \X$ and $n \geq 1$, we call $\pi$ a \emph{semi-stationary} policy and write it simply as $\tilde \mu$.

Let $\Pi$ denote the set of all policies, and let $\Pi_{sm}, \Pi_m$, $\Pi_{ss}$, and $\Pi_s$ denote the subsets of semi-Markov, Markov, semi-stationary, and stationary policies, respectively. If $\Pi_\star$ is any one of these sets, we write $\Pi'_\star$ for the subset of nonrandomized policies in $\Pi_\star$.  All these sets are nonempty, because by the Jankov-von Neumann selection theorem \cite[Prop.~7.49]{bs}, the analytic set $\Gamma$ admits an analytically measurable selection, which corresponds to a universally measurable nonrandomized stationary policy. By contrast, a Borel measurable policy may not exist even if $\Gamma$ is Borel \cite{Blk-borel}.

For $\pi := \{\mu_n\}_{n\geq0} \in \Pi$, an initial state distribution $p_0 \in \P(\X)$ together with the collection of stochastic kernels $\mu_0(d a_0 \nmid x_0)$, $q(dx_1 \nmid x_0, a_0)$, $\mu_1(da_1 \nmid h_1)$, $q(dx_2 \nmid x_1, a_1), \ldots$ determines uniquely a probability measure $\Pr^\pi_{p_0}$ on $\U(\Omega)$ \cite[Prop.\ 7.45]{bs}. The expectation w.r.t.\ $\Pr^\pi_{p_0}$ is denoted by $\E^\pi_{p_0}$.
If $p_0 = \delta_x$, we will also use the notation $\Pr^\pi_{x}$ and $\E^\pi_{x}$.
Throughout the paper, we will simply write a stochastic process on $\Omega$ as $\{(x_n, a_n)\}_{n \geq 0}$, using $(x_n, a_n)$ to denote the random variables $\big(x_n(\omega), a_n(\omega)\big)$ instead, which are the $(x_n, a_n)$-components of $\omega$. 

\section{Properties of Strategic Measures} \label{sec-3}

The probability measures on $\Omega$ induced by the policies are also known as strategic measures in the literature \cite[Sec.~3.5]{DyY79}.
In the universal measurability framework, these are probability measures on $\U(\Omega)$; however, we will work with their restrictions to $\B(\Omega)$, 
in order to make use of well-studied properties of Borel probability measures. Since a probability measure on $\U(\Omega)$ is uniquely determined by its restriction to $\B(\Omega)$ (cf.\ \cite[Sec.~3.3]{Dud02} on completion of measures), 
the two are effectively the same. 
Specifically, for a policy $\pi \in \Pi$ and an initial distribution $p_0 \in \P(\X)$, let $\rho_{p_0}[\pi]$ denote the restriction of $\Pr_{p_0}^\pi$ to $\B(\Omega)$; if $p_0 = \delta_x$, we also write it as $\rho_x [\pi ]$. 
Let 
$$ \S  : = \big\{ p \in \P(\Omega) \mid p = \rho_{p_0}[\pi], \, \pi \in \Pi, \, p_0 \in \P(\X)\big\}$$
be the set of all strategic measures, and let $\S'$ be the subset of strategic measures corresponding to nonrandomized policies in $\Pi'$.
Similarly, with $\Pi_m$ (resp.\ $\Pi_s$) in place of $\Pi$, we define $\S_m$ (resp.\ $\S_s$) to be the set of all strategic measures induced by Markov (resp.\ stationary) policies, and let $\S'_m$ (resp.\ $S'_s$) denote its subset corresponding to those nonrandomized policies. 

\begin{theorem} \label{thm-strat-m}
For $\S_\star \in \{\S, \S_m, \S_s\}$, $\S_\star$ and $\S'_\star$ are analytic; they are Borel if $\Gamma$ is Borel.
\end{theorem}

We prove this theorem in Section~\ref{sec-proof-1}. There we also consider the strategic measures induced by policies with certain structural constraints and give a more general result (cf.\ Prop.~\ref{prp-Sc}) that entails the conclusions of the above theorem for $\S$ and $\S_m$.

\begin{rem} \rm \label{rmk-strm}
We make a few comments on the preceding theorem and related earlier results.\\
(a) For the MDP model with a Borel $\Gamma$ and Borel measurable policies, the Borel measurability of $\S$ was shown in Strauch \cite[Lem.\ 7.2]{Str-negative} and in Dynkin and Yushkevich \cite[Secs.~3.5, 3.6, 5.5]{DyY79}; the Borel measurability of $\S'$ was implied by the results of Blackwell \cite{Blk76}; and the Borel measurability of the other sets involved in Theorem~\ref{thm-strat-m} were established by Feinberg \cite{Fei96}. All these results were summarized in Feinberg \cite[Thm.\ 3.2]{Fei96}. A small difference between \cite[Thm.\ 3.2]{Fei96} and the part of Theorem~\ref{thm-strat-m} for the case of a Borel $\Gamma$ is worth noting: in our case $\Gamma$ need not admit a Borel measurable selection and hence a Borel measurable policy need not exist. (By the Blackwell--Ryll-Nardzewski selection theorem \cite[Thm.~2]{BlR63}, if $\Gamma$ is Borel, then $\Gamma$ admits a Borel measurable selection if and only if there exists a Borel measurable stationary policy.)

\smallskip
\noindent (b) Suppose $\Gamma$ admits a Borel measurable selection, so that there is at least one Borel measurable policy $\mu^o \in \Pi_s'$. Then the sets $\S_\star, \S'_\star$ in Theorem~\ref{thm-strat-m} all coincide with their respective subsets of strategic measures induced by Borel measurable policies. This follows from the fact that w.r.t.\ any given finite measure on its domain space, a universally measurable function coincides with some Borel measurable function almost everywhere \cite[Lems.\ 7.27, 7.28(c)]{bs}. More specifically, using this fact and the policy $\mu^o$, for any initial distribution $p_0$ and policy $\pi : = \{\mu_n\}_{n \geq 0} \in \Pi$, one can modify the stochastic kernels $\mu_n$ one by one to construct a Borel measurable policy $\tilde \pi$ such that $\rho_{p_0}[\tilde \pi] = \rho_{p_0}[\pi]$ and $\tilde \pi$ and $\pi$ belong to the same class of policies (any class $\Pi_\star$ or $\Pi_\star'$ defined in Section~\ref{sec-2.2}).
However, for optimal control, universally measurable policies are still essential in this case; cf.\ Remark~\ref{rmk-opt}(b). 

\smallskip
\noindent (c) Let $\Delta \subset \P(\X) \times \big(\P(\X \times \A)\big)^\infty$ consist of all those sequences $(p_0, \gamma_0, \gamma_1, \ldots)$ such that $p_0 \in \P(\X)$ and for some $\pi \in \Pi_m$, $\gamma_n \in \P(\X \times \A)$ is the marginal distribution of $(x_n, a_n)$ w.r.t.\ $\Pr^\pi_{p_0}$ for all $n \geq 0$. It has been shown by Shreve and Bertsekas \cite[Lem.\ 1]{ShrB79} (see also \cite[Lem.\ 9.1]{bs}) that the set $\Delta$ is analytic; their proof also shows that $\Delta$ is Borel if $\Gamma$ is Borel. As to the relation between $\Delta$ and strategic measures, we can show that $\S_m$ is Borel-isomorphic to $\Delta$, whereas $\S_m'$, $\S_s$, and $\S_s'$ are Borel-isomorphic to certain subsets of $\Delta$. 
\qed
\end{rem}

Many optimal control problems can be equivalently formulated as optimization problems on strategic measures. 
We now give some results that will be useful for constructing (nearly) optimal policies from (nearly) optimal solutions of those optimization problems.

Let $\S_\star \in \{\S, \S_m, \S_s\}$. Define sets ${\tilde \S}_\star, {\tilde \S}^{'}_\star$ by
\begin{align*}
 {\tilde \S}_\star & : = \big\{ (x, p) \in \X \times \P(\Omega) \,\big|\,  p \in \S_\star, \, p_0(p) = \delta_x \big\},  & 
{\tilde \S}^{'}_\star & : = \big\{ (x, p) \in \X \times \P(\Omega) \,\big|\,  p \in \S'_\star, \, p_0(p) = \delta_x \big\},
\end{align*} 
where $p_0(p)$ denotes the marginal distribution of $x_0$ w.r.t.\ $p$.
For $x \in \X$, denote by $\tilde \S_\star(x)$ the $x$-section of $\tilde \S_\star$; i.e., 
$\tilde \S_\star(x) : = \{ p \in \P(\Omega) \mid (x, p) \in \tilde \S_\star \} = \{ p \in \S_\star \mid p_0(p) = \delta_x \}.$
Similarly, the $x$-section of $\tilde \S'_\star$ is denoted by $\tilde \S'_\star(x)$.

\begin{lemma} \label{lem-ext-set}
The sets $\tilde \S_\star$ and $\tilde \S'_\star$ are analytic; they are Borel if $\Gamma$ is Borel.
\end{lemma}

\begin{proof}
The mapping $\psi_1 : p \mapsto (p_0(p), p)$ from $\P(\Omega)$ into $\P(\X) \times \P(\Omega)$ and the mapping $\psi_2 : (x, p) \mapsto (\delta_x, p)$ from $\X \times \P(\Omega)$ into $\P(\X) \times \P(\Omega)$ are homeomorphisms.  
We can express the sets $\tilde \S_\star$ and $\tilde \S'_\star$ as $\tilde \S_\star = \psi_2^{-1}\big( \psi_1(\S_\star)\big)$ and $\tilde \S'_\star = \psi_2^{-1}\big( \psi_1(\S'_\star)\big)$. 
By \cite[Prop.~7.40]{bs}, images and preimages of analytic sets are analytic under Borel measurable mappings;  
by \cite[Sec.\ I.3, Cor.\ 3.3]{Par67}, under one-to-one Borel measurable mappings, images of Borel sets are Borel. 
Therefore, $\tilde \S_\star$ and $\tilde \S'_\star$ are analytic (Borel) if  $\S_\star$ and $\S'_\star$ are analytic (Borel). We then obtain the desired conclusion from Theorem~\ref{thm-strat-m}.
\end{proof} 

\begin{prop} \label{prp-ummap-pol}
Let $\S_\star \in \{\S, \S_m, \S_s\}$. Let $\zeta : \X \to \P(\Omega)$ be a universally measurable mapping such that  $\zeta(x) \in \tilde \S_\star(x)$ for all $x \in \X$. 
Then there exists a policy $\pi \in \Pi$ such that: 
\begin{enumerate}[leftmargin=0.7cm,labelwidth=!]
\item[\rm (i)] $\rho_x[\pi] =  \zeta(x)$ for all $x \in \X$; 
\item[\rm (ii)] for $\S_\star = \S_m$ or $\S_s$, $\pi$ is semi-Markov or semi-stationary, respectively.
\end{enumerate}
If, in addition, $\zeta$ is such that $\zeta(x) \in \tilde \S'_\star(x)$ for all $x \in \X$, then $\pi$ can be taken to be nonrandomized.
\end{prop}

The proof of this result is given in Section~\ref{sec-proof-2}.
A similar result for partially observable problems will be given later in Section~\ref{sec-pomdp} (see Prop.~\ref{prp-Si-pol}). 

The next theorem is similar to Feinberg \cite[Thm.\ 1]{Fei82a} for MDPs with Borel measurable policies. It represents the strategic measures of randomized policies as ``mixtures'' of the strategic measures of nonrandomized policies, and will be important in studying the optimality of nonrandomized polices. We prove this theorem and a more general version of it in Section~\ref{sec-proof-3}.
The proof is similar to that of \cite[Thm.\ 1]{Fei82a} and makes use of a result in Gikhman and Skorokhod~\cite[Lem.\ 1.2]{GiS79}, in addition to the properties of universally measurable sets and functions.

Let $\bar \lambda$ be the countable product of Lebesgue measures on $[0,1]$. 

\begin{theorem} \label{thm-rep-nonrand}
For each $\pi \in \Pi$, there exists a family of nonrandomized policies $\bar f(\bar \theta) \in \Pi'$, parametrized by $\bar \theta \in [0,1]^\infty$, such that
\begin{enumerate}[leftmargin=0.7cm,labelwidth=!]
\item[\rm (i)] the mapping $(p_0, \bar \theta) \mapsto \rho_{p_0}\big[\bar f(\bar \theta) \big]$ from $\P(\X) \times [0,1]^\infty$ into $\P(\Omega)$ is universally measurable; 
\item[\rm (ii)] for all $p_0 \in \P(\X)$,
\begin{equation} \label{eq-rep-nonrand}
   \Pr_{p_0}^\pi( E) = \int_{[0,1]^\infty} \! \Pr_{p_0}^{\bar f(\bar \theta)}( E) \, \bar \lambda (d \bar \theta), \qquad \forall \, E \in \U(\Omega).
\end{equation}
\end{enumerate}
Moreover, if $\pi$ is Markov (semi-Markov), the policies $\bar f(\bar \theta)$ are Markov (semi-Markov) as well.
\end{theorem}
 
\begin{rem} \rm
The representation (\ref{eq-rep-nonrand}) for $\Pr_{p_0}^\pi$ involves the probability measure $\bar \lambda$ on $[0,1]^\infty$. Alternatively, similarly to Feinberg \cite[Thm.\ 5.2]{Fei96}, we can introduce a probability measure $\eta$ on $\U\big(\P(\Omega)\big)$ to express $\Pr_{p_0}^\pi$ as
\begin{equation}
     \Pr_{p_0}^\pi(E) = \int_{\P(\Omega)} \! p(E) \, \eta(d p), \qquad \forall \, E \in \U(\Omega),
\end{equation}
where $\eta$ depends on both $p_0$ and $\pi$. In particular, let $\gamma_{p_0} : [0,1]^\infty \to \P(\Omega)$ with $\gamma_{p_0}(\bar \theta) : = \rho_{p_0}[ \bar f(\bar \theta)]$. By Theorem~\ref{thm-rep-nonrand}, $\gamma_{p_0}$ is universally measurable and we can let $\eta$ be the image measure of $\bar \lambda$ under $\gamma_{p_0}$: $\eta = \bar \lambda \circ \gamma^{-1}_{p_0}$. 
Moreover, by Theorem~\ref{thm-rep-nonrand}, $\eta(\S')=1$, and if $\pi$ is Markov, then $\eta(\S'_m) = 1$.\qed
\end{rem}

By combining the preceding results with \cite[Props.\ 7.47, 7.50]{bs}, we obtain the following theorem about optimization problems on strategic measures.

\begin{theorem} \label{thm-opt}
Let $(\S_\star, \Pi_\star) = (\S, \Pi), (\S_m, \Pi_{sm})$, or  $(\S_s, \Pi_{ss})$. 
Let $\phi : \tilde \S_\star \to \bar \R$ be lower semianalytic. 
Then the following hold:
\begin{enumerate}[leftmargin=0.7cm,labelwidth=!]
\item[\rm (i)] The function $\phi^*(x) : =  \inf_{p \in \tilde\S_\star(x)} \phi(x, p)$, $x \in \X$, is lower semianalytic.
\item[\rm (ii)] For each $\epsilon > 0$, there exists a policy $\pi^* \in \Pi_\star$ such that 
\begin{align}
     \phi(x, \rho_x [ \pi^* ])  & = \phi^*(x), \qquad \forall \,  x \in E^* : = \Big\{ x \in \X \,\big|\, \argmin_{p \in \tilde\S_\star(x)} \phi(x, p) \not= \varnothing \Big\};  \label{eq-um-minimizer-1} \\
 \phi(x, \rho_x [ \pi^* ]) & \leq  \begin{cases}
        \phi^*(x) + \epsilon & \text{if} \ \phi^*(x) > - \infty, \\
        -1/\epsilon & \text{if} \ \phi^*(x) = - \infty,
        \end{cases} \qquad \ \ \  \forall \, x \in \X \setminus E^*.   \label{eq-um-minimizer-2}
\end{align}
\item[\rm (iii)]  The policy $\pi^*$ in (ii) can be taken to be nonrandomized if the function $\phi$ satisfies the condition $\phi^*(x) = \inf_{p \in \tilde\S'_\star(x)} \phi(x, p)$ for all $x \in \X$. In particular, for $\S_\star = \S$ or $\S_m$, this condition holds if $\phi$ is such that for each $\pi \in \Pi$ or $\Pi_m$, respectively, 
\begin{equation} \label{suff-cond-nonrand}
  \phi(x, \rho_x[\pi] ) \geq \int_{[0,1]^\infty} \! \phi \big(x, \rho_x[\bar f(\bar \theta)] \big) \, \bar \lambda (d \bar \theta), \qquad \forall \, x \in \X,
\end{equation}  
where the function $\bar f(\bar \theta)$ associated with $\pi$ is as given in Theorem~\ref{thm-rep-nonrand}.
\end{enumerate}
\end{theorem}

\begin{proof}
Since the set $\tilde \S_\star$ is analytic by Lemma~\ref{lem-ext-set} and the function $\phi$ is lower semianalytic by assumption, part (i) holds by \cite[Prop.\ 7.47]{bs}. Moreover, by \cite[Prop.\ 7.50(b)]{bs}, there exists a universally measurable function $\zeta^* : \X \to \P(\Omega)$ such that its graph lies in $\tilde \S_\star$ and the relations (\ref{eq-um-minimizer-1})-(\ref{eq-um-minimizer-2}) hold with $\zeta^*(x)$ in place of $\rho_x[\pi^*]$.
Letting $\pi^* \in \Pi_\star$ be the policy given by Prop.~\ref{prp-ummap-pol} for the function $\zeta^*$, we obtain part (ii). 

If $\phi^*(x) = \inf_{p \in \tilde\S'_\star(x)} \phi(x, p)$ for all $x \in \X$, then, since the set $\tilde\S'_\star$ is analytic by Lemma~\ref{lem-ext-set}, we can replace $\tilde \S_\star$ with $\tilde \S'_\star$ in the preceding proof so that the graph of $\zeta^*$ lies in $\tilde \S'_\star$. From the last statement of Prop.~\ref{prp-ummap-pol}, we then obtain a nonrandomized policy $\pi^* \in \Pi_\star$ that satisfies (\ref{eq-um-minimizer-1})-(\ref{eq-um-minimizer-2}). 

Finally, for the last statement in part (iii), if $\S_\star = \S$ or $\S_m$, then by Theorem~\ref{thm-rep-nonrand}, $\bar f(\bar \theta) \in \Pi'$ or $\Pi'_m$, respectively, for all $\bar \theta \in [0,1]^\infty$. Hence by assumption (\ref{suff-cond-nonrand}), $\inf_{p \in \tilde\S_\star(x)} \phi(x, p) \geq \inf_{p \in \tilde\S'_\star(x)} \phi(x, p)$ for all $x \in \X$, which implies $\phi^*(x) = \inf_{p \in \tilde\S'_\star(x)} \phi(x, p)$ for all $x \in \X$.
\end{proof} 

\begin{rem} \rm \label{rmk-opt}
We make two comments regarding the preceding theorem and related results for MDPs with Borel measurable policies.\\
(a) For MDPs with Borel $\Gamma$ and Borel measurable policies, optimality results similar to Theorem~\ref{thm-opt}(i)-(ii) were given in
Strauch \cite[Thms.\ 7.1, 8.1]{Str-negative} and Dynkin and Yushkevich \cite[Sec.\ 5.5]{DyY79} for the discounted and total cost criteria, and in Feinberg \cite[Thm.\ 3.1]{Fei82} and \cite[Lem.~4.1, Thm.\ 4.2]{Fei96} for arbitrary Borel measurable numerical criteria;
results related to Theorem~\ref{thm-opt}(iii) were given in Feinberg \cite[Thm.\ 2]{Fei82a} and \cite[Thm.\ 3.2, Cors.\ 3.2, 3.3, 3.5]{Fei82}.
As noted in our introduction, these prior results showed that w.r.t.\ any given finite measure on $\X$, there exist (Borel measurable) randomized or nonrandomized, almost-everywhere $\epsilon$-optimal policies. This is different from Theorem~\ref{thm-opt}(ii) or (iii), in which the policy $\pi^*$ is $\epsilon$-optimal everywhere.

\smallskip
\noindent (b) In the prior works just mentioned, $\Gamma$ is Borel and admits a Borel measurable selection. As discussed earlier in Remark~\ref{rmk-strm}(b), in this case every $\S_\star$ in the preceding theorem coincides with its subset of strategic measures induced by Borel measurable policies. Moreover, the set $\tilde \S_\star$ admits a Borel measurable selection (since the mapping $x \mapsto \rho_x[\mu]$, where $\mu$ is any Borel measurable stationary policy, is a Borel measurable selection of $\tilde \S_\star$). 
However, for the partial minimization problem in the preceding theorem, this ensures only the \emph{feasibility} of at least one Borel measurable policy $\pi$, in the sense that $\rho_x[\pi]$ lies in the feasible set $\tilde \S_\star(x)$ for every state $x$. The level sets of the objective function, $\{ (x, p) \in \tilde \S_\star \mid \phi(x, p) < r \}$, $r \in \R$, need not admit Borel measurable selections even if $\phi$ is Borel measurable. This is the reason that if we only allow Borel measurable policies, in general we can only obtain an almost-everywhere $\epsilon$-optimal policy.
\qed
\end{rem}

As mentioned earlier, the results of this section can be generalized to partially observable problems; see Sections~\ref{sec-pomdp} and \ref{sec-pomdp-proof}. 
In addition, our proof arguments given in Section~\ref{sec-5}, which deal with structured policies, can be adapted and applied to decentralized control problems like those studied in \cite{Yuk20,YuS17} to derive similar results in the universal measurability framework.

\section{Applications and Extensions}  \label{sec-4}

In this section we first specialize the results of Section~\ref{sec-3} to risk-neutral and risk-sensitive MDPs, focusing on performance criteria of the average type. We then extend our analysis to partially observable problems (Section~\ref{sec-pomdp}) and minimax control problems (Section~\ref{sec-minimax}), with the latter being the major extension.

\subsection{MDPs with Risk-Neutral and Risk-Sensitive Average Criteria} \label{sec-4.1}

We will consider various average cost and risk criteria in this subsection and apply Theorem~\ref{thm-opt} to derive optimality results for MDPs under these criteria.
If $J^\star : \Pi \times \X \to \bar \R$ is a cost or risk criterion, w.r.t.\ $J^\star$ and a subset $\Pi_\star$ of policies in $\Pi$,
we define the optimal average cost or risk function by
\begin{equation} \label{def-g-opt}
  g^\star(x) : = \inf_{\pi \in \Pi_\star} J^\star(\pi, x), \qquad x \in \X,
\end{equation}  
and we say a policy $\pi \in \Pi_\star$ is \emph{optimal for state $x$} if $J^\star(\pi, x) = g^\star(x)$,
and \emph{$\epsilon$-optimal for state $x$} (where $\epsilon > 0$) if  
\begin{equation} \label{def-pol-opt}
     J^\star(\pi, x) \leq \begin{cases} 
                g^\star(x) + \epsilon & \text{if} \ g^\star(x) > - \infty, \\
                - \epsilon^{-1} & \text{if} \ g^\star(x) = - \infty.
                \end{cases}   
\end{equation}                
If these relations hold for \emph{all} states $x \in \X$, $\pi$ is called \emph{($\epsilon$-)optimal} w.r.t.\ $J^\star$ and $\Pi_\star$.

\subsubsection{The Risk-Neutral Case} \label{sec-4.1.1}
Let $c^+$ ($c^-$) be the positive (negative) part of the one-stage cost function $c$. 

\begin{definition} \label{def-ac-models}
We say an MDP is in the model class \mdp (\mdn), if for $c^\diamond = c^-$ ($c^\diamond = c^+$),
\begin{equation} \label{eq-ac-mdpn}
   \textstyle{ \E^\pi_x \big[ \sum_{k=0}^n c^\diamond(x_k, a_k) \big] < + \infty, \qquad \forall \, x \in \X, \, \pi \in \Pi, \, n \geq 0.}
\end{equation}
We say an MDP is in the model class \ptmdp (\ptmdn), if $c(\cdot)$ is real-valued and bounded below (above) on $\Gamma$.
\end{definition}

For MDPs in the \mdp or \mdn class, the long-run expected average cost criteria $J^{(i)}, 1 \leq i \leq 4$, given below are all well-defined (i.e., they do not involve $\infty - \infty$). For $\pi \in \Pi, x \in \X$, 
\begin{equation} \label{eq-crt1}
 J^{(1)}(\pi, x) : =   \limsup_{n \to \infty} n^{-1} J_n(\pi, x), \qquad J^{(2)}(\pi, x) : =   \liminf_{n \to \infty} n^{-1}  J_n(\pi, x),
\end{equation}
\begin{equation} \label{eq-crt3}
 J^{(3)}(\pi, x) : =   \limsup_{n \to \infty} \sup_{j \geq 0} n^{-1} J_{n,j}(\pi, x), \quad J^{(4)}(\pi, x) : =   \liminf_{n \to \infty} \inf_{j \geq 0} n^{-1} J_{n,j}(\pi, x),
\end{equation} 
where for $n \geq 1$ and $j \geq 0$, 
$$ J_n(\pi, x) : = \E^\pi_x \big[ \, \textstyle{\sum_{k=0}^{n-1} c(x_k, a_k)} \, \big], \qquad J_{n,j}(\pi, x) : = \E^\pi_x \big[ \, \textstyle{\sum_{k=0}^{n-1} c(x_{k+j}, a_{k+j})} \, \big].$$
($J_n$ is the standard $n$-stage cost function; $J_{n,j}$ instead measures the expected costs incurred from time $j$ to time $j+n-1$.) For the sub-models \ptmdp and \ptmdn, we consider four more well-defined average cost criteria $\tJ^{(i)}, 1 \leq i \leq 4$, which involve average costs along sample paths: for $\pi \in \Pi, x \in \X$,
\begin{align}
   \tJ^{(1)}(\pi, x) & : = \E^\pi_x \Big[ \limsup_{n \to \infty} \tilde c_n
   \Big],  & \tJ^{(2)}(\pi, x) & : = \E^\pi_x \Big[ \liminf_{n \to \infty} \tilde c_n  \Big],    \label{eq-pt-ac1} \\
     \tJ^{(3)}(\pi, x)  & : = \E^\pi_x \Big[ \limsup_{n \to \infty} \sup_{j \geq 0} \tilde c_{n,j} \Big],  &
       \tJ^{(4)}(\pi, x)  & : = \E^\pi_x \Big[ \liminf_{n \to \infty} \inf_{j \geq 0} \tilde c_{n,j} \Big], 
      \label{eq-pt-ac2a}  
\end{align}
where $\tilde c_n : =  n^{-1} \sum_{k=0}^{n-1} c(x_k, a_k)$ and $\tilde c_{n,j} : = n^{-1} \sum_{k=0}^{n-1} c(x_{k+j}, a_{k+j})$.

\begin{rem} \rm
The standard average cost criterion is $J^{(1)}$. Our definitions of the criteria $J^{(i)}, \tJ^{(i)}$, $i = 3, 4$, were motivated by a theorem about Banach limits~\cite[Thm.\ 3.4.1]{Kre85}; see \cite[Remark 2.1]{Yu22} for more details. Regarding the criteria $J^{(i)}, \tJ^{(i)}$, $i = 1, 2$, it has been shown in \cite{Bie87,Fei80} that interestingly, if $\X$ is finite and $c(\cdot)$ is bounded below (the action space can be arbitrary), then the optimal average cost functions w.r.t.\ these criteria and $\Pi$ are identical. Moreover, in this case, there exists a nonrandomized Markov policy that is $\epsilon$-optimal simultaneously w.r.t.\ all these criteria, and under additional compactness/continuity conditions, this policy can be taken to be nonrandomized and stationary. The finiteness of the state space is critical for these results, however. For an infinite state space $\X$, the optimal average cost functions w.r.t.\ these criteria are in general not the same.
\qed
\end{rem}

Consider the case $\Pi_\star = \Pi, \Pi_{sm}$, or $\Pi_{ss}$, and correspondingly let $\S_\star = \S, \S_m$ or $\S_s$, respectively. 
Note that if $\pi$ is a semi-Markov or semi-stationary policy, then for each initial state $x \in \X$, $\rho_x[\pi] \in \S_m$ or $\S_s$, respectively. Let $J^\star$ be any criterion given above. 
Since $J^\star(\pi,x)$ can be expressed as a function of $\Pr^\pi_x$, we can write the optimal average cost function $g^\star$ w.r.t.\ $J^\star$ and $\Pi_\star$ as
\begin{equation} \label{eq-g-strmeas}
  g^\star(x) : = \inf_{\pi \in \Pi_\star} J^\star(\pi, x) =  \inf_{p \in \tilde \S_\star(x)} \phi(p), \qquad x \in \X.
\end{equation}  
Here we define the function $\phi: \P(\Omega) \to \bar \R$ by using the same defining equation for the criterion $J^\star$, except that we replace the expectation w.r.t.\ $\Pr^\pi_x$ by the expectation w.r.t\ $p$. 
For technical convenience, in this definition, we extend $c(\cdot)$ to the entire state-action space by letting $c(\cdot) \equiv +\infty$ outside $\Gamma$, which makes $c(\cdot)$ a lower semianalytic function on $\X \times \A$. (Recall also that by convention $\infty - \infty = - \infty + \infty = \infty$.)

\begin{lemma} \label{lem-ac-lsa}
The function $\phi: \P(\Omega) \to \bar \R$ corresponding to any average cost criterion defined in (\ref{eq-crt1})-(\ref{eq-pt-ac2a}) is lower semianalytic.
\end{lemma}

\begin{proof}
Consider the case $J^{(3)}$, for example. The corresponding function $\phi$ is given by
$\phi(p) : = \limsup_{n \to \infty} \sup_{j \geq 0} n^{-1} \int_\Omega \textstyle{\sum_{k=0}^{n-1} c(x_{k+j}, a_{k+j})} \, p(d \omega)$. Since $c(\cdot)$ is lower semianalytic by our model assumption, $\phi$ is lower semianalytic by \cite[Lem.\ 7.30, Cor.\ 7.48.1]{bs}.
The same argument applies to the other criteria under consideration.
\end{proof} 

\begin{theorem} \label{thm-ac-basic}
Consider these two cases of an average-cost MDP:
\begin{enumerate}[leftmargin=0.7cm,labelwidth=!]
\item[\rm 1.] The MDP is in the class \mdp or \mdn and the criterion $J^\star \in \{J^{(i)}\}_{1 \leq i \leq 4}$.
\item[\rm 2.] The MDP is in the class \ptmdp or \ptmdn and the criterion $J^\star \in \{\tJ^{(i)}\}_{1 \leq i \leq 4}$.
\end{enumerate}
In each case, w.r.t.\ $J^\star$ and $\Pi_\star \in  \{\Pi, \Pi_{sm}, \Pi_{ss} \}$, the following hold:
\begin{enumerate}[leftmargin=0.7cm,labelwidth=!]
\item[\rm (a)] The optimal average cost function $g^\star$ is lower semianalytic.
\item[\rm (b)] For each $\epsilon > 0$, there exists an $\epsilon$-optimal policy $\pi^* \in \Pi_\star$ that is, moreover, optimal for every state that admits an optimal policy.
\item[\rm (c)] In Case 1, the optimal average cost function w.r.t.\ $\Pi$ coincides with that w.r.t.\ $\Pi_{sm}$, and thus the $\epsilon$-optimal policy $\pi^*$ w.r.t.\ $\Pi$ given in (b) can be taken to be semi-Markov.
\item[\rm (d)] If $\Pi_\star \in \{ \Pi, \Pi_{sm}\}$, then for Case 1 with $J^\star \in \{J^{(2)}, J^{(4)}\}$ and $c(\cdot)$ being bounded from below, as well as for Case 2,  the $\epsilon$-optimal policy $\pi^*$ in (b) can be taken to be nonrandomized.
\end{enumerate}
\end{theorem}

\begin{proof}
Parts (a)-(b) follow from (\ref{eq-g-strmeas}), Lemma~\ref{lem-ac-lsa}, and Theorem~\ref{thm-opt}(i)-(ii). 
Part (c) follows from the following two facts. First, for each initial distribution $p_0$ and policy $\pi \in \Pi$, there is a Markov policy $\pi_m \in \Pi_m$ under which the marginal distributions of $(x_n, a_n)$, $n \geq 0$, coincide with those under $\pi$ (cf.\ the proof of \cite[Prop.~1]{ShrB79}). Second, the criteria $J^{(i)}$, $1 \leq i \leq 4,$ can be expressed as functions of those marginal distributions of $(x_n, a_n)$, $n \geq 0$.
Under the assumptions of part (d), condition (\ref{suff-cond-nonrand}) in Theorem~\ref{thm-opt}(iii) is satisfied by the function $\phi$. In particular, under the stated assumptions for Case 1, inequality (\ref{suff-cond-nonrand}) holds by Theorem~\ref{thm-rep-nonrand}(ii) and Fatou's lemma; for Case 2, clearly (\ref{suff-cond-nonrand}) holds with equality by Theorem~\ref{thm-rep-nonrand}(ii). Thus part (d) follows from Theorem~\ref{thm-opt}(iii).
\end{proof} 

\begin{rem} \rm \label{rmk-basicthm} 
We make a few comments on the preceding theorem and related results for average-cost MDPs.\\
(a) For the standard average cost criterion $J^{(1)}$, Theorem~\ref{thm-ac-basic}(c) asserts the existence of a \emph{randomized} semi-Markov $\epsilon$-optimal policy that is, moreover, optimal for every state that admits an optimal policy.
This conclusion is the strongest possible, without extra conditions on the MDP. Indeed it is known that even in an MDP with a countable state space, a finite action space, and bounded one-stage costs, there need not exist a nonrandomized semi-Markov $\epsilon$-optimal policy \cite[Chap.~7, Example 3]{DyY79} nor a randomized Markov $\epsilon$-optimal policy \cite[Sec.~5]{Fei80}. In both of these counterexamples, there exists an optimal policy for each state.

\smallskip
\noindent (b)  For the criteria $J^{(1)}, J^{(3)}, \tJ^{(i)}$, $i \leq 4$, using Theorem~\ref{thm-ac-basic}, we can further show that under certain reachability and boundedness conditions, the optimal average cost functions w.r.t.\ $\Pi$ are constant almost everywhere w.r.t.\ certain $\sigma$-finite measures (see \cite[Lem.\ 2.4, Cor.\ 3.2]{Yu22} for more details).

\smallskip
\noindent (c) For MDPs with the average cost criterion $J^{(1)}$ and with Borel measurable policies, there are extensive studies on the existence of nonrandomized stationary optimal policies and the validity of optimality equations or optimality inequalities, under various continuity/compactness or ergodicity conditions, and with a variety of methods. Some of the results have also been extended to MDPs with universally measurable policies. The literature is too vast to list here; we refer the interested reader to the surveys and book chapters \cite{ArB93,Bor02,HL02}, \cite[Chaps.\ 5-6]{HL96}, \cite[Chaps.\ 10-12]{HL99}, the recent articles \cite{FK21,FKL20,VAm18,Yu20}, and the references therein.
\qed
\end{rem}

\subsubsection{The Risk-Sensitive Case} \label{sec-risk-mdp}

The results of Section~\ref{sec-3} on strategic measures are applicable to a large class of risk-sensitive MDPs. 
As concrete examples, we discuss below several criteria related to average risk.
For simplicity, in this discussion, we assume that the one-stage cost function $c$ is real-valued on $\Gamma$.

Following a general framework in \cite{BaR14}, we define the first two average risk criteria through a strictly increasing, nonnegative continuous function $\psi: \bar \R \to [0, + \infty]$ as follows.
For all $\pi \in \Pi$, let
\begin{equation} \label{eq-risk-1}
    J_\psi(\pi, x) : = \limsup_{n \to \infty} \frac{1}{n} \, \psi^{-1} \! \left( \E^\pi_x \left[ \psi\big( \textstyle{\sum_{k=0}^{n-1}} c(x_k, a_k) \big) \right] \right),  \qquad x \in \X,
\end{equation}  
\begin{equation} \label{eq-risk-2}
   \hat J_\psi(\pi, x) : = \limsup_{n \to \infty}  \sup_{j \geq 0}  \frac{1}{n} \,\psi^{-1} \! \left( \E^\pi_x \left[ \psi\big( \textstyle{\sum_{k=0}^{n-1}} c(x_{j+k}, a_{j+k}) \big) \right] \right), \qquad x \in \X.
\end{equation}  
In the special case where $\psi(z) : = e^{\beta z}$ for some $\beta > 0$, (\ref{eq-risk-1}) becomes
\begin{equation}
  J_\psi(\pi, x) = \limsup_{n \to \infty} \frac{1}{\beta n} \log \E^\pi_x \left[ e^{\beta \sum_{k=0}^{n-1} c(x_k, a_k)} \right], \quad x \in \X, \notag
\end{equation}
which is related to the standard exponential utility and has been considered in many works on risk-sensitive MDPs under certain continuity/compactness model assumptions (see, e.g., \cite{CaS10,Jas07}). 

We mention that the non-negativity of $\psi$ is assumed for simplicity. As in \cite{BaR14}, we can also treat a larger class of criteria without this restriction, using the same arguments given below, if we introduce additional conditions to ensure that the expectations in (\ref{eq-risk-1})-(\ref{eq-risk-2}) are well-defined.

The next two criteria are based on Conditional Value-at-Risk (CVaR) and Value-at-Risk (VaR), two popular criteria for risk-averse control. In particular, for $\omega \in \Omega$, define the average cost along a sample path by
$\bar c(\omega) : = \limsup_{n \to \infty} \tfrac{1}{n} \sum_{k=0}^{n-1} c(x_k, a_k)$ (one can also define $\bar c(\omega)$ differently; the idea is the same).
For $r \in \bar \R$, let $(r)_+ : = \max \{ r, 0\}$.
Define the risk criteria $R_1, R_2$ to be the $\text{CVaR}_\alpha$ and $\text{VaR}_\alpha$ of the random variable $\bar c(\omega)$ under a policy $\pi \in \Pi$, respectively, for some parameter $\alpha \in (0,1]$ that represents the level of risk-aversion:
\begin{equation} \label{eq-cvar}
    R_1(\pi, x) : = \inf_{z \in \R} \left\{ z + \tfrac{1}{\alpha} \E^\pi_x \left[ ( \bar c(\omega) - z)_+ \right] \right\}, \qquad x \in \X,
\end{equation} 
\begin{equation} \label{eq-var}
    R_2(\pi, x) : = \inf \left\{ z \in \R \,\big|\,  \Pr^\pi_x \{ \bar c(\omega) \leq z \} \geq 1 - \alpha \right\}, \qquad x \in \X.
\end{equation} 
(As a side note, if for some constant $\ell$, $\bar c(\omega) = \ell$, $\Pr^\pi_x$-almost surely, then $R_1(\pi, x) = R_2(\pi, x) = \ell$. This ``degenerate'' case can happen if $\pi$ is stationary and satisfies certain ergodicity conditions.) 
For the probabilistic interpretations of $\text{CVaR}_\alpha$ and $\text{VaR}_\alpha$ and their use in stochastic programming/control, we refer the interested reader to \cite[Chap.\ 6]{SDR21}, \cite{WFC22}, and the references therein.

Every risk criterion given above can be expressed as a function of $\Pr^\pi_x$. So, as in the risk-neutral case, we can express the optimal risk function $g^\star$ w.r.t.\ the criterion and a certain subset of policies as a partial minimization problem of the form (\ref{eq-g-strmeas}) on strategic measures. The function $\phi(p)$ in (\ref{eq-g-strmeas}) is defined in the same way as before, by extending $c(\cdot)$ to $\X \times \A$ and by substituting $p$ for $\Pr^\pi_x$ in the defining equation for the criterion.
To apply Theorem~\ref{thm-opt}, let us show that $\phi$ is lower semianalytic.

\begin{lemma} \label{lem-anal-fn}
Let $X$ be a Borel space, $f: X \to \bar \R$, and $s:  \bar \R \to \bar \R$. If $f$ is lower semianalytic and $s$ is non-decreasing, then $s \circ f$ is lower semianalytic.
\end{lemma}
\begin{proof}
By the definition of lower seminalytic functions, we need to show that for any $r \in \R$, the set $B: = \{ x \in X \mid s ( f(x)) < r \}$ is analytic. Since $s$ is non-decreasing, $E: = \{ a \in \bar \R \mid s(a) < r \}$ is either the empty set or a nonempty interval of the form $[-\infty, r')$ or $[-\infty, r']$ for some $r' \in \bar \R$. In the first case, $B = \varnothing$ and is thus analytic. In the second case, we have $B = \{ x \in X \mid f(x) < r' \}$ or $\{ x \in X \mid f(x) \leq r' \}$, which is analytic by \cite[Lem.\ 7.30(1)]{bs} since $f$ is lower semianalytic.
\end{proof} 

\begin{lemma} \label{lem-risk-lsa}
Let $c(\cdot)$ be real-valued on $\Gamma$. The function $\phi: \P(\Omega) \to \bar \R$ corresponding to any risk criterion defined in (\ref{eq-risk-1})-(\ref{eq-var}) is lower semianalytic.
\end{lemma}

\begin{proof}
Consider first the criteria $J_\psi, \hat J_\psi$ in (\ref{eq-risk-1})-(\ref{eq-risk-2}). 
Since $\psi$ is strictly increasing, we can extend $\psi^{-1}$ to $\bar \R$ in such a way that the extension is a nondecreasing function on $\bar \R$. To simplify notation, we use the same symbol $\psi^{-1}$ to refer to this extension in the proof. 
Since $c(\cdot)$ is lower semianalytic, for all $n \geq 1, j \geq 0$,
the function $v_{n,j} (\omega) : = \sum_{k=0}^{n-1} c(x_{j+k}, a_{j+k})$ is lower semianalytic \cite[Lem.\ 7.30(4)]{bs}, and hence the function $\psi(v_{n,j}(\omega))$ is lower semianalytic by the monotonicity of $\psi$ and Lemma~\ref{lem-anal-fn}.
Then by \cite[Prop.\ 7.48]{bs}, the function $f_{n,j}(p) : =  \int \psi(v_{n,j}(\omega)) \, p(dw)$ is lower semianalytic on $\P(\Omega)$, so by Lemma~\ref{lem-anal-fn} $\psi^{-1}(f_{n,j}(p))$ is lower semianalytic on $\P(\Omega)$ as well. Finally, it follows from \cite[Lem.\ 7.30(2) and (4)]{bs} that the function $\phi(p) = \limsup_{n \to \infty} \tfrac{1}{n} \psi^{-1} (f_{n,0}(p))$ for the criterion $J_\psi$ and  the function $\phi(p) = \limsup_{n \to \infty} \sup_{j  \geq 0} \tfrac{1}{n} \psi^{-1} (f_{n,j}(p))$ for the criterion $\hat J_\psi$ are both lower semianalytic. 

Consider now the criterion $R_1$ given in (\ref{eq-cvar}). Since $c(\cdot)$ is lower semianalytic, by \cite[Lem.\ 7.30(2) and (4)]{bs}, the function $\bar c(\omega)$ is lower seminalytic on $\Omega$, and then the function 
$f(\omega, z) : = \tfrac{1}{\alpha} ( \bar c(\omega) - z)_+$ is lower semianalytic on $\Omega \times \R$ by \cite[Lem.\ 7.30(2) and (4)]{bs}.
Consequently, the function $(p, z) \mapsto \int f(\omega, z) \, p(d \omega)$ is lower semianalytic on $\P(\Omega) \times \R$ by \cite[Prop.\ 7.48]{bs}. It then follows from \cite[Lem.\ 7.30(4), Prop.\ 7.47]{bs} that the result of the partial minimization, $\phi(p) = \inf_{z \in \R} \{ z +  \int f(\omega, z) \, p(d \omega)\}$, is a lower semianalytic function on $\P(\Omega)$. 

Finally, consider the criterion $R_2$ given in (\ref{eq-var}). Since $\bar c(\cdot)$ is lower semianalytic, $(\omega, z) \mapsto \bar c(\omega) - z$ is a lower semianalytic function on $\Omega \times \R$ by \cite[Lem.\ 7.30(4)]{bs}, so the set $E : = \{ (\omega, z) \in \Omega \times \R \mid \bar c (\omega) - z \leq 0 \}$ is analytic. Then by \cite[Prop.\ 7.48]{bs},  the function $(p, z) \mapsto \int \ind_E(\omega, z) \, p(d\omega) = p \{ \bar c(\omega) \leq z \}$ is upper semianalytic on $\P(\Omega) \times \R$ and hence
the set $D : = \big\{ (p, z) \in \P(\Omega) \times \R \mid  p \{ \bar c(\omega) \leq z  \} \geq 1 - \alpha \big\}$ is analytic.
By the definition of $\phi$ in this case, we have $\phi(p) = \inf_{z \in D_p} z$, $p \in \P(\Omega)$, where $D_p$ is the $p$-section of $D$. Therefore $\phi$ is lower semianalytic by \cite[Prop.\ 7.47]{bs}.
\end{proof} 

\begin{rem} \rm
If $\Gamma$ and $c(\cdot)$ are Borel measurable, the function $\phi$ for any criterion defined in (\ref{eq-risk-1})-(\ref{eq-var}) is Borel measurable. This is obvious for (\ref{eq-risk-1})-(\ref{eq-risk-2}). For (\ref{eq-cvar})-(\ref{eq-var}), this can be seen from an alternative expression of $\phi$ given in (\ref{eq-cvar-exp}) and (\ref{eq-var-exp}), respectively, in Section~\ref{sec-minimax-model}. 
\qed
\end{rem}

From Lemma~\ref{lem-risk-lsa} and the proof of Theorem~\ref{thm-opt}, we obtain the following:

\begin{theorem} \label{thm-risk-mdp-opt} 
Let $c(\cdot)$ be real-valued on $\Gamma$. Let $J^\star$ be any risk criterion defined in (\ref{eq-risk-1})-(\ref{eq-var}), and let $\Pi_\star \in \{\Pi, \Pi', \Pi_{sm}, \Pi'_{sm}, \Pi_{ss}, \Pi'_{ss}\}$.
Then w.r.t.\ $J^\star$ and $\Pi_\star$, the optimal risk function $g^\star$ is lower semianalytc, and for each $\epsilon > 0$, there exists an $\epsilon$-optimal policy $\pi^* \in \Pi_\star$ that is, moreover, optimal for every state that admits an optimal policy.
\end{theorem} 

\subsubsection{Partially Observable Problems} \label{sec-pomdp}

We consider a general non-stationary model for stochastic control with imperfect state information. It is similar to the one defined in \cite[Def.\ 10.3]{bs} and can be described as follows:
\begin{itemize}[leftmargin=0.65cm,labelwidth=!]
\item The state space $\X$, the action space $\A$, and the state transition stochastic kernel $q(dy \,|\, x, a)$ are the same as defined in Section~\ref{sec-2.2}. 
\item The \emph{observation space} $\Z$ is a Borel space. There is a Borel measurable mapping $f: \X \to \Z$, which we call the \emph{observation function}. 
For $n \geq 0$, observable to the controller at the $n$th stage is not $x_n$ but the value of $z_n = f(x_n)$. 
We call $i_n : = (z_0, a_0, z_1, a_1, \ldots, z_n)$ an \emph{$n$th information vector} and denote its space $(\Z \times \A)^n \times \Z$ by $I_n$.
\item The control constraint at the $n$th stage depends on the $n$th information vector realized in the system. Specifically, there is a set-valued mapping $U_n: i_n \mapsto U_n(i_n) $ on $I_n$, where $U_n(i_n) \subset \A$ is a nonempty set of admissible actions when $i_n$ is the realized $n$th information vector. We assume that the graph of $U_n$, $\Gamma_n : = \{ (i_n, a_n) \mid i_n \in I_n, \, a_n \in U_n(i_n) \},$ is analytic.
\item The distribution $p_0$ of the initial state $x_0$ is known to the controller.
\end{itemize}

In this model we let $\Omega : = (\X \times \Z \times \A)^\infty$ and define a policy to be a sequence of stochastic kernels $\mu_n, n \geq 0$, where $\mu_n$ is a universally measurable stochastic kernel on $\A$ given $I_n \times \P(\X)$ such that
$$ \mu_n(U_n(i_n) \mid i_n; p_0) = 1, \qquad \forall \, i_n \in I_n, \, p_0 \in \P(\X).$$
The space of all policies is denoted by $\imPi$, and the subset of all nonrandomized policies is denoted by $\imPin$.
Since $\Gamma_n, n \geq 0$ are analytic, by the Jankov-von Neumann selection theorem \cite[Prop.~7.49]{bs}, there exists a sequence of analytically measurable functions $f_n: I_n \to \A$ with the graph of $f_n$ contained in $\Gamma_n$. Such a sequence corresponds to a nonrandomized policy in $\imPin$. The sets $\imPin$ and $\imPi$ are therefore nonempty. 

Let $J^\star(\pi, p_0)$ be an extended real-valued function on $\imPi \times \P(\X)$ that measures the performance of $\pi$ for an initial distribution $p_0$. For the criterion $J^\star$ and $\Pi_\star \in \{\imPi, \imPin\}$, the optimal value function $g^\star$ and the ($\epsilon$-)optimal policies are defined as in the case of MDPs with initial distributions $p_0 \in \P(\Omega)$ in place of initial states $x \in \X$ (cf.\ (\ref{def-g-opt})-(\ref{def-pol-opt})).

Let $\imS$ denote the set of strategic measures $\rho_{p_0}[\pi]$ on $\B(\Omega)$ induced by all policies $\pi \in \imPi$ and initial distributions $p_0 \in \P(\X)$. Let $\imSn$ denote the subset of strategic measures corresponding to those nonrandomzied policies in $\imPin$. Similarly to the case of MDPs, let 
$$ \timS : = \{ (p_0, p) \in \P(\X) \times \P(\Omega) \mid p \in \imS, \, p_0(p) = p_0 \}, \ \   \timSn : = \{ (p_0, p) \in \P(\X) \times \P(\Omega) \mid p \in \imSn, \, p_0(p) = p_0 \},$$
where $p_0(p)$ is the initial distribution of $x_0$ w.r.t.\ $p$.
For each $p_0 \in \P(\X)$, $\timS(p_0)$ and $\timSn(p_0)$ denote the $p_0$-section of $\imS$ and of $\imSn$, respectively.

\begin{prop} \label{prp-Si-pol}
For the partially observable control problem defined above, the following hold:
\begin{enumerate}[leftmargin=0.7cm,labelwidth=!]
\item[\rm (i)] The sets $\imS, \imSn, \timS$, and $\timSn$ are analytic; they are Borel if $\Gamma_n, n \geq 0,$ are Borel.
\item[\rm (ii)] Let $\zeta : \P(\X) \to \P(\Omega)$ be a universally measurable mapping such that  $\zeta(p_0) \in \tilde \imS(p_0)$ for all $p_0 \in \P(\X)$. 
Then there exists a universally measurable policy $\pi \in \imPi$ such that  $\rho_{p_0}[\pi] =  \zeta(p_0)$ for all $p_0 \in \P(\X)$. If, in addition, $\zeta(p_0) \in \timSn(p_0)$ for all $p_0 \in \P(\X)$, then $\pi$ can be taken to be nonrandomized.
\end{enumerate}
\end{prop}

\begin{theorem} \label{thm-pomdp-opt}
Consider the partially observable control problem defined above. Suppose that the one-stage cost function $c(\cdot)$ is real-valued and bounded above or below on $\X \times \A$. Define $J^\star$ according to any average cost or risk criterion given in (\ref{eq-crt1})-(\ref{eq-pt-ac2a}) or (\ref{eq-risk-1})-(\ref{eq-var}) by extending those definitions to the space of initial distributions. Then w.r.t.\ $J^\star$ and $\Pi_\star \in \{\imPi, \imPin \}$, the following hold:
\begin{enumerate}[leftmargin=0.7cm,labelwidth=!]
\item[\rm (i)] the optimal average cost/risk function $g^\star$ is lower semianalytic; and 
\item[\rm (ii)] for each $\epsilon > 0$, there exists an $\epsilon$-optimal policy $\pi^* \in \Pi_\star$ that is, moreover, optimal for every initial distribution that admits an optimal policy. 
\end{enumerate}
\end{theorem}

The proofs of Prop.~\ref{prp-Si-pol} and Theorem~\ref{thm-pomdp-opt}, being similar to their counterparts in the case of MDPs, will be outlined in Section~\ref{sec-pomdp-proof}.

\subsection{Minimax Control} \label{sec-minimax}

We now extend our results to a class of minimax control problems.

\subsubsection{Model and Assumptions} \label{sec-minimax-model}

We first specify the minimax control model: 
\begin{itemize}[leftmargin=0.7cm,labelwidth=!] 
\item[(a)] The state space $\X$, the action space $\A$, and the state transition stochastic kernel $q(dy \,|\, x, a)$ are the same as defined in Section~\ref{sec-2.2}.
\item[(b)] The action space $\A = \A_1 \times \A_2$, where $\A_1$ and $\A_2$ are \emph{Borel spaces} and correspond to the action spaces for player 1 (the controller) and player 2 (the opponent or nature), respectively. 
In terms of its components, we write an action $a \in \A$ as $a = (a^1, a^2)$.
\item[(c)]  For $i = 1,2$, the control constraint for player $i$ is given by a set-valued mapping $A_i$ on $\X$, $A_i : x \mapsto A_i(x) \subset \A_i$, where $A_i(x)$ is a nonempty subset of admissible actions for player $i$ at state $x$. The graphs of $A_1$ and $A_2$, denoted by $\Gamma_1$ and $\Gamma_2$, respectively, are \emph{analytic}. 
\item[(d)] The two players take actions simultaneously at each stage with player 1 never observing player~2's actions. Specifically, for $n \geq 0$, let $a_n : = (a_n^1, a_n^2)$ where $a_n^i$ is the action of player $i$ at the $n$th stage. Prior to making these actions, player 2 has the knowledge of the history $h_n : = (x_0, a_0, \ldots, x_n)$ up to stage $n$, whereas player~1 has access only to $i_n : =  (x_0, a_0^1, x_1, a_1^1, \ldots, x_n)$ and never observes the actions taken by player~2. The space of $i_n$ is denoted by $I_n$. 
\end{itemize}
Let us remark on the role of assumption (d). With this restriction on the information available to the controller, the problem becomes a game against ``nature.''  As will be seen later, while we can readily characterize strategic measures induced by cooperative or noncooperative players in partial-information settings that are much more general than the setting in (d), we need this assumption in order to reformulate minimax control problems as minimax problems on strategic measures. 

We define a policy for player $i$ to be a sequence $\pi_i : = \{ \mu_n^i \}_{n \geq 0}$, where $\mu_n^i$ is a universally measurable stochastic kernel on $\A_i$ given the space of the variables observable to that player at the $n$th stage---namely, $I_n$ for player 1 and $H_n$ for player 2,
such that
$$ \mu_n^1 ( A_1(x_n) \mid i_n ) = 1 \ \ \ \forall \, i_n \in I_n, \qquad  \mu_n^2 ( A_2(x_n) \mid h_n ) = 1 \ \ \  \forall \, h_n \in H_n. $$
We denote the set of all policies for player $i$ by $\Pi_i$. Since $\Gamma_1$ and $\Gamma_2$ are analytic, $\Pi_1$ and $\Pi_2$ are nonempty by \cite[Prop.\ 7.49]{bs}.

Denote the probability measure on $\U(\Omega)$ induced by a policy pair $(\pi_1, \pi_2) \in \Pi_1 \times \Pi_2$ and initial distribution $p_0 \in \P(\X)$ by $\Pr^{\pi_1,\pi_2}_{p_0}$. The associated expectation operator is denoted by $\E^{\pi_1,\pi_2}_{p_0}$, and the restriction of  $\Pr^{\pi_1,\pi_2}_{p_0}$ to $\B(\Omega)$ is denoted by $\rho_{p_0}[\pi_1, \pi_2]$. As before, if $p_0 = \delta_x$, we write $x$ instead of $\delta_x$ in the subscripts. 

Let $J^\star:  \Pi_1 \times \Pi_2 \times \X \to \bar \R$ be some given function. We consider the minimax control problem for player 1 w.r.t.\ the criterion $J^\star$, where the optimal value function is given by
$$ g^\star(x) : = \inf_{\pi_1 \in \Pi_1} \sup_{\pi_2 \in \Pi_2} J^\star(\pi_1, \pi_2, x), \quad x \in \X.$$
A policy $\pi_1 \in \Pi_1$ is called \emph{optimal for state $x$} if $\sup_{\pi_2 \in \Pi_2} J^\star(\pi_1, \pi_2, x) = g^\star(x)$, and \emph{$\epsilon$-optimal for state $x$} (where $\epsilon > 0$) if 
\begin{equation} \label{def-minmax-pol-opt}
     \sup_{\pi_2 \in \Pi_2} J^\star(\pi_1, \pi_2, x) \leq \begin{cases} 
                g^\star(x) + \epsilon & \text{if} \ g^\star(x) > - \infty, \\
                - 1/\epsilon & \text{if} \ g^\star(x) = - \infty.
                \end{cases}   
\end{equation}                
If these relations hold for \emph{all} states $x \in \X$, $\pi_1$ is called an \emph{($\epsilon$-)optimal} policy of player 1 w.r.t.\ $J^\star$.

We now introduce two assumptions. 
The first one is an absolute continuity condition on the probability measures induced on $I_n$ when player 1 applies the same policy.
It will allow us to rewrite the minimax control problem as a minimax optimization problem on strategic measures (cf.\ Prop.~\ref{prp-minmax-strm2}(ii) and the proof of Theorem~\ref{thm-minmax-opt}).

\begin{assumption} \label{cond-minimax-abscont}
For any $\pi_1 \in \Pi_1$, $\pi_2, \hat \pi_2 \in \Pi_2$, $x \in \X$, and $n \geq 0$, the marginal distributions of $\Pr^{\pi_1,\pi_2}_{x}$  and $\Pr^{\pi_1,\hat \pi_2}_{x}$ on $I_n$ are absolutely continuous w.r.t.\ each other.
\end{assumption}

Note that if in the above we replace ``for any $x \in \X$'' by ``for any initial distribution $p_0 \in \P(\X)$,'' we obtain an equivalent assumption. We now describe a class of problems that satisfy Assumption~\ref{cond-minimax-abscont}.

\begin{example} \rm
Suppose the state transition stochastic kernel $q(dy \,|\, x,a) = f(y, x, a) \, \eta(dy \,|\, x, a^1)$, where:
\begin{itemize}[leftmargin=0.7cm,labelwidth=!]
\item[(a)] $f$ is a strictly positive Borel measurable function on $\X^2 \times \A$; and
\item[(b)] $\eta$ is a nonnegative kernel on $\X$ given $\X \times \A_1$ such that (i) for each $B \in \B(\X)$, $\eta(B \,|\, x, a^1)$ is Borel measurable in $(x, a^1)$, and (ii) for each $(x, a^1) \in \X \times \A_1$, $\eta(\cdot \,|\, x, a^1)$ is a nonnegative measure on $\B(\X)$ and furthermore, $\eta(\cdot \,|\, x, a^1)$ is $\sigma$-finite uniformly in $(x, a^1)$ in the sense that for some sequence of Borel sets $B_k$ with $\X = \cup_k B_k$, we have $\sup_{x \in \X, a^1 \in \A_1} \eta( B_k \,|\, x, a^1) < \infty$ for all $k$. 
\end{itemize}
Then Assumption~\ref{cond-minimax-abscont} holds. Indeed, let $n \geq 0, p_0 \in \P(\X)$, and let $\pi_1: = \{\mu_n^1\} \in \Pi_1$, and $\pi_2 := \{\mu_n^2\}, \hat \pi_2 : = \{{\hat \mu}_n^2\} \in \Pi_2$. Consider a set $B \in \U(I_n)$ with $\Pr^{\pi_1,\pi_2}_{p_0}\{ i_n \in B\} = 0$, and let us show that $\Pr^{\pi_1, \hat \pi_2}_{p_0}\{ i_n \in B\}  = 0$, which will imply the desired conclusion since $\pi_2, \hat \pi_2$ are arbitrary. 
To this end, let $\gamma_n$ be the unique nonnegative $\sigma$-finite measure on $\U(H_n)$ determined by the product of $p_0(dx_0)$ with the other stochastic kernels and nonnegative kernels given below:
\begin{align}
 & p_0(dx_0), \ \mu_0^1(da^1_0 \mid x_0), \ \mu_0^2(da^2_0 \mid x_0), \  \eta(dx_1 \mid x_0, a^1_0), \ \ldots, \notag \\
 & \ldots, \ \mu_{n-1}^1(da_{n-1}^1 \mid i_{n-1}), \ \mu_{n-1}^2(da_{n-1}^2 \mid h_{n-1}), \ \eta(d x_n \mid x_{n-1}, a^1_{n-1}). \label{eq-ex-kernels}
\end{align} 
Define the measure $\hat \gamma_n$ on $\U(H_n)$ similarly, with $\hat \mu^2_k$ replacing $\mu^2_k$ in the above, for all $k < n$. The existence of $\gamma_n$ and $\hat \gamma_n$ follows from a repeated application of the product measure theorem given in \cite[Thm.\ 2.6.2]{Ash72} together with the proof arguments of \cite[Prop.\ 7.45]{bs} (the latter extend the result of the product measure theorem to $\U(H_n)$, which is not a product $\sigma$-algebra, by using the properties of universally measurable stochastic kernels).
By these arguments, one can also show that
$\Pr^{\pi_1,\pi_2}_{p_0}\{ i_n \in B\}  = \int_{H_n} \ind_B(i_n) \prod_{k=1}^n f(x_k, x_{k-1}, a_{k-1}) \, \gamma_n(d h_n)$ and $\Pr^{\pi_1,\hat \pi_2}_{p_0}\{ i_n \in B\}  = \int_{H_n} \ind_B(i_n) \prod_{k=1}^n f(x_k, x_{k-1}, a_{k-1}) \, \hat \gamma_n(d h_n)$.
Since $\Pr^{\pi_1,\pi_2}_{p_0}\{ i_n \in B\}  = 0$ and the function $f$ is strictly positive, we have 
$\int_{H_n} \ind_B(i_n) \, \gamma_n(d h_n) = 0$. As can be seen from (\ref{eq-ex-kernels}), the marginal measures of $\gamma_n$ and $\hat \gamma_n$ on $I_n$ must coincide, since $\eta(dx_{k+1} \,|\, x_{k}, a_{k}^1)$ and $\mu_k^1(da_k^1 \,|\, i_k)$, $k < n$, do not depend on the actions of player~2. Therefore, $\int_{H_n} \ind_B(i_n) \, \hat \gamma_n(d h_n) = 0$, implying $\Pr^{\pi_1,\hat \pi_2}_{p_0}\{ i_n \in B\} = 0$.
\qed
\end{example}

The second assumption places a measurability requirement on the criterion $J^\star$. It is commonly satisfied in applications, as we will discuss shortly, using examples.  

\begin{assumption} \label{cond-minmax-obj}
Associated with the criterion $J^\star$, there is an upper semianalytic function $\phi:  \P(\Omega) \to \bar \R$ such that
$J^\star(\pi_1, \pi_2, x) = \phi(\rho_x[\pi_1, \pi_2])$ for all $x \in \X, \pi_1 \in \Pi_1, \pi_2 \in \Pi_2$.
\end{assumption}

\begin{example} \rm
Let $\Gamma : = \{(x, a) \in \X \times \A \mid (x, a^1) \in \A_1, \, (x, a^2) \in \A_2\}$. Note that $\Gamma$ is analytic (Borel) if $\Gamma_1$ and $\Gamma_2$ are analytic (Borel); see Section~\ref{sec-minmax-proof} for a proof. Let $c : \Gamma \mapsto \bar \R$ be universally measurable. 
Consider any average cost or risk criterion $J^\star$ discussed in Section~\ref{sec-4.1}, with $(\pi_1, \pi_2)$ in place of $\pi$ and with $\Pr^{\pi_1, \pi_2}_x, \E^{\pi_1, \pi_2}_x$ in place of $\Pr^\pi_x, \E^\pi_x$, respectively. As before, since $J^\star$ can be expressed as a function of $\Pr^{\pi_1, \pi_2}_x$, by replacing $\Pr^{\pi_1, \pi_2}_x$ with $p \in \P(\Omega)$, we can define an associated function $\phi(p)$ on $\P(\Omega)$ that satisfies the equality condition in Assumption~\ref{cond-minmax-obj}. Let us show that under suitable conditions on $c(\cdot)$ and the sets $\Gamma_1, \Gamma_2$, the function $\phi$ is upper semianalytic and hence those average cost or risk criteria discussed earlier all satisfy Assumption~\ref{cond-minmax-obj}. 

Consider first any average cost criterion $J^\star$ defined in (\ref{eq-crt1})-(\ref{eq-pt-ac2a}). 
For that criterion, assume a model condition analogous to one of the model classes given in Def.~\ref{def-ac-models} for MDPs, so that for all $x \in \X, \pi_1 \in \Pi_1$, and $\pi_2 \in \Pi_2$, $J^\star(\pi_1, \pi_2, x)$ is well-defined and does not involve the summation $\infty - \infty$ or $-\infty + \infty$.
Now suppose that \emph{$c(\cdot)$ is upper semianalytic on $\Gamma$}. Since this means the function $-c(\cdot)$ is lower semianalytic on $\Gamma$, after taking care of some technical subtlety, we can apply arguments symmetric to those given in the proof of Lemma~\ref{lem-ac-lsa} to obtain that the function $\phi(p)$ corresponding to $J^\star$ is upper semianalytic.   

The technical subtlety just mentioned is as follows. In defining $\phi(p)$ on $\P(\Omega)$, we extend $c$ from $\Gamma$ to $\X \times \A$ and also allow summations involving both $\infty$ and $-\infty$. For an upper seminalaytic $c(\cdot)$, we extend $c$ to $\X \times \A$ by letting $c(\cdot) \equiv - \infty$ outside $\Gamma$, and we take $\infty - \infty = - \infty + \infty = - \infty$ in defining $\phi(p)$. These are opposite to the conventions we use when dealing with lower semianalytic functions. (They are needed solely for technical convenience; in particular, for making the extension of $c$ and $\phi$ upper semianalytic on the entire spaces. They do not affect the values of $\phi(p)$ for strategic measures $p$ due to the model condition mentioned earlier.)

Consider now any risk criterion $J^\star$ defined in (\ref{eq-risk-1})-(\ref{eq-cvar}). Assume that \emph{$c(\cdot)$ is real-valued and upper semianalytic on $\Gamma$}. Then the function $\phi(p)$ associated with $J^\star$ is upper semianalytic. The proof of this is entirely symmetric to the proofs of Lemmas~\ref{lem-anal-fn} and~\ref{lem-risk-lsa} in the case of (\ref{eq-risk-1}) or (\ref{eq-risk-2}). For the CVaR-based criterion (\ref{eq-cvar}), the proof is slightly different. First, by an argument symmetric to that given in the proof of Lemma~\ref{lem-risk-lsa}, the function $(p, z) \mapsto \tfrac{1}{\alpha} \int (\bar c(\omega) - z)_+ \, p(d\omega)$ is upper semianalytic on $\P(\Omega) \times \R$, and hence it is upper semianalytic in $p$ for each $z$. Using this and the fact that, with $\Q$ being the countable set of rational numbers,
\begin{align} \label{eq-cvar-exp}
   \phi(p): = & \inf_{z \in \R} \Big\{ z +  \alpha^{-1} \int_\Omega (\bar c(\omega) - z)_+ \, p(d\omega) \Big\}  \notag \\
   = & \inf_{z \in \Q} \Big\{ z +  \alpha^{-1}  \int_\Omega (\bar c(\omega) - z)_+ \, p(d\omega) \Big\}, \qquad p \in \P(\Omega)
\end{align}   
(which can be verified directly), it follows from \cite[Lem.\ 7.30(2)]{bs} that $\phi$ is upper seminalytic.

Finally, for the VaR-based criterion (\ref{eq-var}), let us assume that \emph{the sets $\Gamma_1$ and $\Gamma_2$ are Borel and the function $c(\cdot)$ is real-valued and Borel measurable on $\Gamma$}. Then by \cite[Prop.\ 7.29]{bs}, the set $D : = \big\{ (p, z) \in \P(\Omega) \times \R \mid  p \{ \bar c(\omega) \leq z  \} \geq 1 - \alpha \big\}$ is Borel. 
It can be verified directly that the function $\phi(p)$ in this case satisfies
\begin{align} \label{eq-var-exp}
\phi(p) : =  & \inf \left\{ z \in \R  \,\big|\,  p \{ \bar c(\omega) \leq z \} \geq 1 - \alpha \right\} \notag \\
= & \inf \left\{ z \in \Q  \,\big|\,  p \{ \bar c(\omega) \leq z \} \geq 1 - \alpha \right\}, \qquad p \in \P(\Omega).
\end{align}  
Thus, if $\{z_1, z_2, \ldots\}$ is an enumeration of $\Q$, 
we can express $\phi$ as $\phi(p) = \inf_{i \geq 1} f_i(p)$, for the Borel measurable functions $f_i$ defined by
$f_i(p) : = z_i$ if $(p, z_i) \in D$, and $f_i(p) = + \infty$ if $(p, z_i) \not\in D$. This shows that $\phi$ is Borel measurable and hence upper semianalytic.
\qed
\end{example}

\subsubsection{Properties of Strategic Measures and Optimality Results} \label{sec-minimax-opt}

Let $\mmS$ be the set of all strategic measures: 
$$ \mmS : = \big\{ p \in \P(\Omega) \mid p = \rho_{p_0}[\pi_1, \pi_2],  \,   \pi_1 \in \Pi_1,\,  \pi_2 \in \Pi_2, \, p_0 \in \P(\X) \big\}.$$
Define
$\tmmS: = \{ (x, p) \in \X \times \P(\Omega) \mid p \in \mmS, \, p_0(p) = \delta_x \}$
and denote the $x$-section of $\tmmS$ by $\tmmS(x)$. The proofs of the next two propositions are given in Section~\ref{sec-minmax-proof}.

\begin{prop} \label{prp-minmax-strm1}
For the minimax control problem defined above:
\begin{enumerate}[leftmargin=0.7cm,labelwidth=!]
\item[\rm (i)] The sets $\mmS$ and $\tmmS$ are analytic; they are Borel if $\Gamma_1$ and $\Gamma_2$ are Borel.
\item[\rm (ii)] Let $\zeta : \X \to \P(\Omega)$ be a universally measurable mapping such that $\zeta(x) \in \tmmS(x)$ for all $x \in \X$. 
Then there exist universally measurable policies $\pi_1 \in \Pi_1, \pi_2 \in \Pi_2$ such that  $\rho_{x}[\pi_1, \pi_2] =  \zeta(x)$ for all $x \in \X$. 
\end{enumerate}
\end{prop}

For $n \geq 0$, let $\hat \nu^1_n(d a^1_n \,|\, i_n; p)$ be a Borel measurable stochastic kernel on $\A_1$ given $I_n \times \P(\Omega)$ such that for each $p \in \P(\Omega)$, it is a conditional distribution of $a^1_n$ given $i_n$ w.r.t.\ $p$. The existence of these stochastic kernels is ensured by \cite[Cor.\ 7.27.1]{bs} (cf.\ the discussion at the beginning of Section~\ref{sec-proof-0}). Let $\hat p_{I_n}(p)$ denote the marginal of $p$ on $I_n$.
Define
$$ \hmmS : = \Big\{ (x, p, p') \in \X \times \P(\Omega)^2 \,\big|\, (x, p) \in \tmmS, \, (x, p') \in \tmmS, \, p' \ \text{satisfies (\ref{eq-minmax-rel1})-(\ref{eq-minmax-rel2})} \Big\},$$
where the relations (\ref{eq-minmax-rel1}) and (\ref{eq-minmax-rel2}) are as follows:
\begin{align}
 & p' \left\{ {\hat \nu}^1_n ( \cdot \mid i_n; p') =  {\hat \nu}^1_n( \cdot \mid i_n; p) \right\} = 1, \quad \forall \, n \geq 0,  \label{eq-minmax-rel1} \\
 & \hat p_{I_n}(p') \ \text{is absolutely continuous w.r.t.} \  \hat p_{I_n}(p), \quad \forall \, n \geq 0. \label{eq-minmax-rel2}
\end{align}
We denote the $(x,p)$-section of $\hmmS$ by $\hmmS(x, p)$. 

\begin{prop} \label{prp-minmax-strm2}
For the minimax control problem defined above:
\begin{enumerate}[leftmargin=0.7cm,labelwidth=!]
\item[\rm (i)] The set $\hmmS$ is analytic; it is Borel if $\Gamma_1$ and $\Gamma_2$ are Borel.
\item[\rm (ii)] Under Assumption~\ref{cond-minimax-abscont}, if $p = \rho_x[\pi_1, \pi_2]$ for some $x \in \X$ and $(\pi_1, \pi_2) \in \Pi_1 \times \Pi_2$, then the set $\hmmS(x, p) =  \{  \rho_x[\pi_1, \tilde \pi_2] \mid \tilde \pi_2 \in \Pi_2 \}$.
\end{enumerate}
\end{prop}

We are now almost ready to present our optimality result. In order to fully settle the measurability questions involved, we still need an extension of those results listed in (d) near the end of Section~\ref{sec-2.1}. This requires an additional axiomatic assumption and is given by the following proposition under the axioms of ZFC and analytic determinacy. Its proof relies on a uniformaization theorem of Kond{\^o} \cite[Cor.\ 38.7]{Kec95} (proved under ZFC) and will be given in Section~\ref{sec-sel-proof} for completeness. While this proposition is sufficient for our purpose, its conclusions actually hold for a larger class of sets $D$ and functions $f$, as will be explained in Remark~\ref{rmk-sigma21} after the proof.

\begin{prop} \label{prp-ext-sel}
Assume the axiom of analytic determinacy. Let $X$ and $Y$ be Borel spaces, let $D \subset X \times Y$ be analytic, and let $f: D \to \bar \R$ be upper semianalytic. Define $f^*: \text{\rm proj}_X(D) \to \bar \R$ by $f^*(x) = \inf_{y \in D_x} f(x, y)$, where $D_x$ is the $x$-section of $D$.
Then the following hold:
\begin{enumerate}[leftmargin=0.7cm,labelwidth=!]
\item[\rm (i)] The function $f^*$ is universally measurable.
\item[\rm (ii)] The set $E^* : = \{ x \in \text{\rm proj}_X(D) \mid  \argmin_{y \in D_x} f(x,y) \not= \varnothing \}$ is universally measurable. 
For each $\epsilon > 0$, there exists a universally measurable function $\psi: \text{\rm proj}_X(D) \to Y$
such that $\psi(x) \in D_x$ for all $x \in \text{\rm proj}_X(D)$ and 
\begin{align} 
     f(x, \psi(x)) & = f^*(x), \qquad \forall \,  x \in E^*;   \label{eq-ext-sel1} \\
 f(x, \psi(x)) &  \leq  \begin{cases}
        f^*(x) + \epsilon & \text{if} \ f^*(x) > - \infty, \\
        -1/\epsilon & \text{if} \ f^*(x) = - \infty,
        \end{cases}  \qquad  \forall \, x \in \text{\rm proj}_X(D) \setminus E^*. \label{eq-ext-sel2}
\end{align}
\end{enumerate}
\end{prop}

\begin{rem} \label{rmk-ad} \rm
Regarding analytic determinacy and the standard ZFC axioms:\\
(a) Roughly speaking, the axiom of analytic determinacy concerns a certain class of infinite-stage two-player games, called analytic games, in which the two players take turns to choose a natural number, thereby producing an infinite sequence of such numbers, and the set of winning outcomes is a member of the algebra generated by analytic sets in the space of all such sequences. The axiom asserts that all these games are determined (i.e., one of the two players has a winning strategy). For the precise definition and discussions on the mathematical evidence for this axiom and the related large cardinal axioms in set theory, see the book \cite[Sec.\ 26.B]{Kec95} and the articles \cite[p.\ 6582]{MaS88} and \cite{Koe14}.

\noindent (b) The following assertion is known to be independent of ZFC (namely, ZFC can neither prove nor disprove it; cf.\ \cite[p.\ 546]{MPS90} and \cite[p.\ 677]{PrS16}): there exists a coanalytic set $B$ in $X \times Y := [0, 1]^2$ whose projection on $X$ is not universally measurable. If such a set $B$ exists, a counterexample to Prop.~\ref{prp-ext-sel} can be constructed by letting $D = [0,1]^2$ and $f(x, y) = \ind_{B^c}(x, y)$, and similar counterexamples to the universal measurability of the optimal value functions in MDPs and in minimax stochastic optimization problems have been discussed in \cite[Ex.\ (48)]{BFO74} and \cite[Ex.\ 1, Remark 4]{Now10}, respectively. Under ZFC and analytic determinacy, such sets $B$ do not exist \cite[Thm.\ 36.20]{Kec95}; see Section~\ref{sec-sel-proof} for more details.
\qed
\end{rem}

\begin{theorem} \label{thm-minmax-opt}
Assume the axiom of analytic determinacy. Consider the minimax control problem defined above, and let Assumptions~\ref{cond-minimax-abscont}-\ref{cond-minmax-obj} hold.
Then w.r.t.\ $J^\star$, we have the following:
\begin{enumerate}[leftmargin=0.7cm,labelwidth=!]
\item[\rm (i)] The optimal value function $g^\star$ for player 1 is universally measurable.
\item[\rm (ii)] For each $\epsilon > 0$, player 1 has a universally measurable $\epsilon$-optimal policy $\pi^*_1 \in \Pi_1$ that is, moreover, optimal for every state that admits an optimal policy.
\end{enumerate}
\end{theorem}

\begin{proof} 
By Assumptions \ref{cond-minimax-abscont}-\ref{cond-minmax-obj} and Prop.~\ref{prp-minmax-strm2}(ii), for all $\pi_1 \in \Pi_1$ and $x \in \X$, we have
\begin{equation} \label{eq-minmax-prf1}
   \sup_{\pi_2 \in \Pi_2} J^\star(\pi_1, \pi_2, x) = \sup_{p' \in \hmmS(x, p)} \phi(p'), \quad \text{where  $p = \rho_x[\pi_1, \hat \pi_2]$ for some $\hat \pi_2 \in \Pi_2$}.
\end{equation}   
In view of the definitions of $\hmmS$ and $\tmmS$, this implies that for all $x \in \X$,
$$g^\star(x) = \inf_{p \in \tmmS(x)} f(x, p),   \quad \text{where} \ f: \tmmS \to \bar \R \ \text{with} \ f (x,p) : = \sup_{p' \in \hmmS(x, p)} \phi(p').$$
Since $\phi$ is upper semianalytic by Assumption~\ref{cond-minmax-obj} and $\hmmS$ is analytic by Prop.~\ref{prp-minmax-strm2}(i), the function $f$ is upper semianalytic on $\tmmS$ by \cite[Prop.\ 7.47]{bs}. Recall that $\tmmS$ is analytic by Prop.~\ref{prp-minmax-strm1}(i). It then follows from Prop.~\ref{prp-ext-sel}(i) that $g^\star$ is universally measurable.

Let $\epsilon > 0$. By Prop.~\ref{prp-ext-sel}(ii), there is a universally measurable mapping $\zeta: \X \to \P(\Omega)$ with $\zeta(x) \in \tmmS(x)$ for all $x \in \X$ such that 
\begin{align}
     f(x, \zeta(x)) & = g^\star(x),   \qquad \forall \, x \in E^* : = \Big\{ x \in \X \,\Big|\, \argmin_{p \in  \tmmS(x)} f(x, p) \not=\varnothing \Big\}  \label{eq-minmax-prf2} \\
      f(x, \zeta(x)) & \leq \begin{cases}
         g^\star(x) + \epsilon &  \text{if} \ g^\star(x) > - \infty, \\
         - 1/\epsilon & \text{if} \ g^\star(x) = - \infty,
         \end{cases} \qquad \forall \, x \not\in E^*.  \label{eq-minmax-prf3} 
 \end{align}     
Let $(\pi^*_1, \pi^*_2)$ be the pair of policies given by Prop.~\ref{prp-minmax-strm1}(ii) for this mapping $\zeta$.
By (\ref{eq-minmax-prf1}) and Prop.~\ref{prp-minmax-strm1}(ii),
$\sup_{\pi_2 \in \Pi_2} J^\star(\pi^*_1, \pi_2, x) = f(x, \rho_x[\pi^*_1, \pi^*_2]) = f(x, \zeta(x)),$
so, in view of (\ref{eq-minmax-prf2})-(\ref{eq-minmax-prf3}), the policy $\pi^*_1$ is $\epsilon$-optimal and moreover, optimal for every state that admits an optimal policy.
\end{proof} 

\begin{rem} \rm We mention two direct extensions of the preceding theorem.\\
(a) One can also define the optimal value function $g^\star$ and the ($\epsilon$-)optimal policies w.r.t.\ the subset of nonrandomized policies, or the subset of randomized or nonrandomized, semi-Markov or semi-stationary policies. By similar arguments one can show that Theorem~\ref{thm-minmax-opt} remains valid if $\Pi_1$ is replaced by such a subset of policies of player 1.

\smallskip
\noindent (b) Another extension concerns the definition of $g^\star$. Suppose that we have a sequence of criteria $J^\star_n$, $n \geq 0$, all satisfying Assumption~\ref{cond-minmax-obj}, and we define player 1's optimal value function $g^\star$ as 
$$ g^\star(x) : =  \inf_{\pi_1 \in \Pi_1} \limsup_{n \to \infty} \sup_{\pi_2 \in \Pi_2} J^\star_n (\pi_1, \pi_2, x), \quad x \in \X. $$ 
Then, by redefining the function $f(x,p)$ and following the same reasoning in the above proof, one can show that Theorem~\ref{thm-minmax-opt} remains valid. An example of this type of $g^\star$ is the function considered in \cite{JaN14} for robust MDPs:
$g^\star(x) : =  \inf_{\pi_1 \in \Pi_1} \limsup_{n \to \infty} \sup_{\pi_2 \in \Pi_2} \E^{\pi_1,\pi_2}_x \big[  n^{-1} \textstyle{\sum_{k=0}^{n-1}} c(x_k, a_k) \big]$.
\qed
\end{rem}

\begin{rem} \rm
Suppose that $\Gamma_1$ and $\Gamma_2$ admit Borel measurable selections so that both players have Borel measurable policies. Then $\mmS$ coincides with the set of strategic measures induced by Borel measurable policy pairs and moreover, because of the absolute continuity condition in Assumption~\ref{cond-minimax-abscont}, one can show that for each initial state, player 1 has a Borel measurable $\epsilon$-optimal policy. This is similar to the MDP case (cf.\  Remark~\ref{rmk-strm}(b)). But this need not be true for general minimax problems without Assumption~\ref{cond-minimax-abscont} (such as the stochastic games studied in \cite{Now85b}). The reason, roughly speaking, is that although one can always modify an $\epsilon$-optimal universally measurable policy of player 1 into a Borel measurable policy by ``sacrificing'' its $\epsilon$-optimality in some situations, unlike in MDPs, the probability of these situations happening can be made non-negligible by player 2's adversarial responses. Thus for minimax control, Borel measurable policies, even if they exist, are in general inadequate. 
\qed
\end{rem}

\subsubsection{Some Implications on Optimality Equations/Inequalities} \label{sec-minmax-oe}

So far we have not discussed optimality equations or inequalities for the optimal value function $g^\star$.   
Basic optimality results like the ones given above can be useful in deriving such relations for some performance criteria.
We close this section with a few results and examples in this regard.

First, let us extend Prop.~\ref{prp-minmax-strm1}(ii) to parametrized policies (the proof is given in Section~\ref{sec-minmax-proof}):
 
\begin{prop} \label{prp-minmax-strm3}
For the minimax control problem defined above, suppose that $\Theta$ is a Borel space and $\zeta : \X \times \Theta \to \P(\Omega)$ is a universally measurable mapping with $\zeta(x, \theta) \in \tmmS(x)$ for all $x \in \X$ and $\theta \in \Theta$. 
Then there exist policies $\pi_1(\theta) \in \Pi_1$, $\pi_2(\theta) \in \Pi_2$, parametrized by $\theta$, such that 
\begin{enumerate}[leftmargin=0.7cm,labelwidth=!]
\item[\rm (i)] $\rho_{x}[\pi_1(\theta), \pi_2(\theta)] =  \zeta(x, \theta)$ for all $x \in \X$ and $\theta \in \Theta$; and 
\item[\rm (ii)] with $\pi_1(\theta) = \{ \mu^1_n[\theta] \}_{n \geq 0}$ and $\pi_2(\theta) = \{ \mu^2_n[\theta] \}_{n \geq 0}$, the stochastic kernels $\mu^1_n[\theta](da_n^1 \,|\, i_n)$ and $\mu^2_n[\theta](d a^2_n \,|\, h_n)$ are universally measurable in $(i_n, \theta)$ and $(h_n, \theta)$, respectively, for each $n \geq 0$.
\end{enumerate}
\end{prop}

We will use this proposition in proving the next result, which is about the existence of policies of player 2 that can respond to player 1's strategy in a nearly optimal way from time $1$ onwards. These policies will then be used to derive an inequality on $g^\star$ for some cost criteria.

Let  $\pi_1 : = \{ \mu^1_n \}_{n \geq 0} \in \Pi_1$ and $\pi_2 : = \{ \mu^2_n \}_{n \geq 0} \in \Pi_2$. 
Using their stochastic kernels for time $n \geq 1$, we define their ``shifted'' policies 
$\pi^s_1(\tilde i_1) : =  \{ \mu^{s,1}_n[\tilde i_1] \}_{n \geq 0} \in \Pi_1$ and $\pi^s_2(\tilde h_1) : =  \{ \mu^{s,2}_n[ \tilde h_1] \}_{n \geq 0} \in \Pi_2$ for each $\tilde i_1 : =(\tilde x_0, \tilde a_0^1) \in \X \times \A_1$ and $\tilde h_1 : = (\tilde x_0, \tilde a_0) \in \X \times \A$ as follows: for $n \geq 0, i_n \in I_n, h_n \in H_n$,
\begin{equation} \label{eq-shft-pol}
 \mu^{s,1}_n[\tilde i_1](\, \cdot \mid i_n) : = \mu^1_{n+1} (\, \cdot \mid  \tilde i_1, i_n), \qquad \mu^{s,2}_n[\tilde h_1](\, \cdot \mid h_n) : = \mu^2_{n+1} (\, \cdot \mid  \tilde h_1, h_n).
 \end{equation}

\begin{prop} \label{prp-minmax-strm4}
Consider the minimax control problem defined above. Under Assumptions~\ref{cond-minimax-abscont}-\ref{cond-minmax-obj}, for each $\bar \pi_1 \in \Pi_1$ and $\epsilon > 0$, player 2 has a policy $\pi_2 \in \Pi_2$ such that 
\begin{equation} \label{eq-ineq-player2}
  J^\star \big({\bar \pi}^s_1(\tilde i_1), \pi^s_2(\tilde h_1), x \big) \geq f_\epsilon\big({\bar \pi}^s_1(\tilde i_1), x\big) , \qquad \forall \, \tilde i_1 \in \X \times \A_1, \ \tilde h_1 = (\tilde i_1, \tilde a_0^2) \in \X \times \A,
\end{equation}
where $f_\epsilon : \Pi_1 \times \X \to \bar \R$ is given by
$f_\epsilon(\pi'_1, x) : = \sup_{\pi_2' \in \Pi_2} J^\star(\pi'_1, \pi'_2, x) - \epsilon$ if $\sup_{\pi_2' \in \Pi_2} J^\star(\pi'_1, \pi'_2, x) < + \infty$ and $f_\epsilon(\pi'_1, x) : = 1/\epsilon$ otherwise.
\end{prop} 

\begin{rem} \rm 
A corollary of Prop.~\ref{prp-minmax-strm4} is that for any $\pi_1 \in \Pi_1$, $\sup_{\pi_2 \in \Pi_2} J^\star (\pi_1, \pi_2, x)$ is universally measurable in $x$. To see this, let ${\bar \pi_1}^s(\cdot) \equiv \pi_1$ and for $k \geq 1$, let $\pi_{2,k}$ be a policy given by Prop.~\ref{prp-minmax-strm4} for $\epsilon = 1/k$. Then for a fixed $\tilde h_1$, observe that $J^\star \big(\pi_1, \pi^s_{2,k}(\tilde h_1), x \big)$ is a universally measurable function of $x$ and, by (\ref{eq-ineq-player2}), converges pointwise to $\sup_{\pi_2 \in \Pi_2} J^\star(\pi_1, \pi_2, x)$ as $k \to \infty$. 
\qed
\end{rem}

\begin{proof}[Proof of Prop.~\ref{prp-minmax-strm4}.]
Similarly to the proof of Theorem~\ref{thm-minmax-opt}, we first consider the partial maximization problem 
$$f(x, p) : = \sup_{p' \in \hmmS(x,p)} \phi(p'), \qquad (x, p) \in \text{proj}_{\X \times \P(\Omega)} (\hmmS) = \tmmS, $$ 
where $\phi$, the function corresponding to $J^\star$, is upper semianalytic by Assumption~\ref{cond-minmax-obj} and the set $\hmmS$ is analytic by Prop.~\ref{prp-minmax-strm2}(i). Then by \cite[Prop.\ 7.50]{bs}, for each $\epsilon > 0$, there is a universally measurable mapping 
$\xi : \tmmS \to \P(\Omega)$ with $\xi(x, p) \in \hmmS(x,p)$ for all $(x, p) \in \tmmS$ such that 
\begin{equation} \label{eq-minmax-prf4}
 \phi\big(\xi(x,p)\big) \geq \begin{cases}
            f(x, p)  - \epsilon  & \text{if} \ f(x, p) < + \infty, \\
            1/\epsilon & \text{if} \ f(x, p) = + \infty.
            \end{cases}
\end{equation}  

Let $\bar \pi_1 \in \Pi_1$ be any policy of player 1. To prove the proposition, we will construct below a policy $\pi_2$ that satisfies the requirement (\ref{eq-ineq-player2}) for $\bar \pi_1$ and a given $\epsilon > 0$. To this end, first, let $\bar \pi_2 \in \Pi_2$ be any policy of player 2 that never uses the information about $a_0^2$, so that we can simply write its ``shifted'' policies ${\bar \pi}^s_2(\tilde h_1)$ as ${\bar \pi}^s_2(\tilde i_1)$ instead, for $\tilde i_1 \in \X \times \A_1$. Note that the ``shifted'' policies ${\bar \pi}^s_1(\tilde i_1), {\bar \pi}^s_2(\tilde i_1)$ comprise stochastic kernels that are universally measurable jointly in the variable $\tilde i_1$ and their respective observable variables (cf.\ (\ref{eq-shft-pol})). Then, by using \cite[Lem.~7.28, Prop.\ 7.46]{bs} and the Dynkin system theorem~\cite[Prop.~7.24]{bs}, it is not hard to show that the mapping $\sigma$ defined as
$$ \sigma : (x, \tilde i_1)  \mapsto \rho_x \big[ {\bar \pi}^s_1(\tilde i_1),\, {\bar \pi}^s_2(\tilde i_1) \big] \in \tmmS(x)$$
is universally measurable from $\X \times ( \X \times \A_1)$ into $\mmS$. 

Since the compositions of universally measurable mappings on Borel spaces are universally measurable \cite[7.44]{bs}, it follows that the mapping
$\zeta (x, \tilde i_1) : = \xi\big(x, \sigma(x, \tilde i_1)\big)$
is universally measurable from $\X \times (\X \times \A_1)$ into $\P(\Omega)$. 
Since $\xi(x, p) \in \hmmS(x,p) \subset \tmmS(x)$ for all $(x, p) \in \tmmS$, we have $\zeta(x, \tilde i_1) \in \tmmS(x)$ for all $x \in \X$ and $\tilde i_1 \in \X \times \A_1$. By applying Prop.~\ref{prp-minmax-strm3} to this function $\zeta$ with $\theta = \tilde i_1$ and $\Theta = \X \times \A_1$, we obtain policies $\hat \pi_1(\tilde i_1), \hat \pi_2(\tilde i_1)$ that are parametrized by $\tilde i_1$ and have those properties given in Prop.~\ref{prp-minmax-strm3}(i)-(ii). Let us write $\hat \pi_2(\tilde i_1)$ as $\hat \pi_2(\tilde i_1) = \{ \hat \mu^2_{n}[\tilde i_1] \}_{n \geq 0}$ and define a policy $\pi_2 : = \{\mu_n^2 \}_{n \geq 0}$ for player 2 by letting $\mu_0^2$ be any analytically measurable selection of $\Gamma_2$ and letting 
$$ \mu_n^2( \cdot \mid h_n) : = \hat \mu^2_{n-1}[x_0, a_0^1]( \cdot \mid x_1, a_1, \ldots, x_n), \qquad h_n \in H_n, \ n \geq 1.$$
Since $\hat \pi_2(\tilde i_1)$ has the property given in Prop.~\ref{prp-minmax-strm3}(ii), $\mu_n^2$, $n \geq 1,$ are universally measurable stochastic kernels, so $\pi_2$ as defined is a valid policy in $\Pi_2$. Note that $\pi_2$ never uses the information about $a_0^2$ and denoting its ``shifted'' policies by $\pi^s_2(x_0, a_0^1)$, the above definition amounts to setting $\pi^s_2(x_0, a_0^1) = \hat \pi_2(x_0, a_0^1)$ for all $(x_0, a_0^1)\in \X \times \A_1$.

Now let us show that the policy $\pi_2$ satisfies the requirement (\ref{eq-ineq-player2}). For each $\tilde i_1 \in \X \times \A_1$,
using the property of $\hat \pi_1(\tilde i_1)$ and ${\hat \pi}_2(\tilde i_1)$ given by Prop.~\ref{prp-minmax-strm3}(i), we have
\begin{equation} \label{eq-minmax-prf5}
   \rho_x \big[ \hat \pi_1(\tilde i_1), \pi^s_2(\tilde i_1) \big]  =  \rho_x \big[ \hat \pi_1(\tilde i_1), {\hat \pi}_2(\tilde i_1) \big]  = \zeta (x, \tilde i_1)  =  \xi\big(x, \sigma(x, \tilde i_1)\big) \in \hmmS\big(x, \sigma(x, \tilde i_1)\big).
\end{equation}    
By Prop.~\ref{prp-minmax-strm2}(ii),  
\begin{equation} \label{eq-minmax-prf6}
\hmmS\big(x, \sigma(x, \tilde i_1)\big) = \big\{ \, \rho_x [ {\bar \pi}^s_1(\tilde i_1), \, \pi'_2 ] \,\big|\, \pi'_2 \in \Pi_2 \, \big\}.
\end{equation}
Thus $\rho_x [ \hat \pi_1(\tilde i_1), \pi^s_2(\tilde i_1)] \in \{  \rho_x [ {\bar \pi}^s_1(\tilde i_1),  \pi'_2 ] \mid \pi'_2 \in \Pi_2 \}$, implying $\rho_x [ \hat \pi_1(\tilde i_1),  \pi^s_2(\tilde i_1) ]  =  \rho_x [ {\bar \pi}^s_1(\tilde i_1), \pi^s_2(\tilde i_1) ]$  by the definition of $\rho_x[\cdot, \cdot]$.
Using Assumption~\ref{cond-minmax-obj} and (\ref{eq-minmax-prf4})-(\ref{eq-minmax-prf5}), we then have that
\begin{align}
 J^\star \big({\bar \pi}^s_1(\tilde i_1), \pi^s_2(\tilde i_1), x \big) & = \phi\big( \rho_x \big[ {\bar \pi}^s_1(\tilde i_1), \pi^s_2(\tilde i_1) \big] \big) \notag \\
 & =  \phi \big( \xi(x, \sigma(x, \tilde i_1)) \big)  \notag \\
 & \geq \begin{cases}
            f(x, \sigma(x, \tilde i_1))  - \epsilon  & \text{if} \ f(x, \sigma(x, \tilde i_1)) < + \infty, \\
            1/\epsilon & \text{if} \ f(x, \sigma(x, \tilde i_1)) = + \infty.
            \end{cases} \label{eq-minmax-prf7}
\end{align}
By the definition of $f$, (\ref{eq-minmax-prf6}) and Prop.~\ref{prp-minmax-strm2}(ii), 
$f(x, \sigma(x, \tilde i_1)) = \sup_{\pi'_2 \in \Pi_2} J^\star \big({\bar \pi}^s_1(\tilde i_1), \pi'_2, x \big).$
This together with (\ref{eq-minmax-prf7}) establishes the desired relation (\ref{eq-ineq-player2}) for $\pi_2$. The proof is now complete.
\end{proof} 

We now use the preceding results to derive optimality equations or inequalities for several risk-neutral criteria.
\emph{In the two examples below we assume the axiom of analytic determinacy and Assumption~\ref{cond-minimax-abscont}.} For simplicity, \emph{we also assume that $c(\cdot)$ is a bounded upper semianalytic function.}

\begin{example}[Optimality equations] \rm
First, consider the average cost criterion $\tJ^{(1)}$ given in (\ref{eq-pt-ac1}). Since $c(\cdot)$ is assumed to be bounded, for any $\pi_1 \in \Pi_1, \pi_2 \in \Pi_2, x \in \X$, we have
$\tJ^{(1)}(\pi_1, \pi_2, x) : = \E^{\pi_1, \pi_2}_x \big[ \limsup_{n \to \infty} n^{-1} \sum_{k=0}^{n-1} c(x_k, a_k) \big] = \E^{\pi_1, \pi_2}_x \big[ \limsup_{n \to \infty} n^{-1} \sum_{k=1}^{n} c(x_k, a_k) \big].$ 
Inside the last term above, take conditional expectation given $(x_0, a_0, x_1)$ to obtain
\begin{equation} \label{eq-exoe1}
 \tJ^{(1)}(\pi_1, \pi_2, x) = \E^{\pi_1, \pi_2}_x \big[ \tJ^{(1)} \big(\pi^s_1(x_0, a_0^1), \, \pi^s_2(x_0, a_0), x_1 \big) \big].
\end{equation} 
Now for $\epsilon > 0$, let $\pi^*_1 \in \Pi_1$ be an $\epsilon$-optimal policy of player 1 given by Theorem~\ref{thm-minmax-opt}.
Correspondingly, let $\pi^*_2 \in \Pi_2$ be a policy of player 2 given by Prop.~\ref{prp-minmax-strm4} for $\bar \pi_1 = \pi^*_1$, so that we have
\begin{equation} \label{eq-exoe2}
 \tJ^{(1)} \big(\pi^{*,s}_1(x_0, a_0^1), \, \pi^{*,s}_2(x_0, a_0), x_1 \big) \geq \sup_{\pi_2' \in \Pi_2} \tJ^{(1)} \big(\pi^{*,s}_1(x_0, a_0^1), \, \pi_2', x_1 \big) - \epsilon \geq g^\star(x_1) - \epsilon
\end{equation} 
for all $(x_0, a_0, x_1) \in \X \times \A \times \X$.
Then by the $\epsilon$-optimality of $\pi^*_1$ and (\ref{eq-exoe1})-(\ref{eq-exoe2}), for all $x \in \X$,
\begin{equation} 
  g^\star(x) + \epsilon \geq \tJ^{(1)}(\pi^*_1, \pi^*_2, x)  = \E^{\pi^*_1, \pi^*_2}_x \big[ \tJ^{(1)} \big(\pi^{*,s}_1(x_0, a_0^1), \, \pi^{*,s}_2(x_0, a_0), x_1 \big) \big] 
   \geq \E^{\pi^*_1, \pi^*_2}_x \big[ g^\star(x_1) \big] - \epsilon. \notag
\end{equation}
Since $\epsilon$ is arbitrary and the initial decision rule of $\pi^*_2$ can be chosen arbitrarily, this implies
\begin{equation} \label{eq-oe1}
g^\star(x) \geq  \inf_{\nu \in \P(A_1(x))} \sup_{a^2 \in A_2(x)} \int_{\A_1} \int_{\X}  g^\star(y) \, q(dy \mid x, a^1, a^2) \, \nu(d a^1 \mid x).
\end{equation}

To obtain the opposite inequality, let $\bar \pi_1$ be a policy of player 1 that applies the policy $\pi^*_1$ from time $1$ onwards; in other words, its shifted policies ${\bar \pi}^s_1(x_0, a_0^1)$ all coincide with $\pi^*_1$. Then by (\ref{eq-exoe1}) and the $\epsilon$-optimality of $\pi^*_1$, we have that 
\begin{equation} \label{eq-oe-ineq0}
g^\star(x) \leq  \sup_{\pi_2 \in \Pi_2} \tJ^{(1)}(\bar \pi_1, \pi_2, x) \leq \sup_{\pi_2 \in \Pi_2} \E^{\bar \pi_1, \pi_2}_x \big[ g^\star(x_1) \big] + \epsilon, \qquad \forall \, x \in \X.
\end{equation}
Since $\epsilon$ is arbitrary and the initial decision rule of $\bar \pi_1$ can be chosen arbitrarily, we obtain the opposite inequality to (\ref{eq-oe1}) and hence the optimality equation
\begin{equation} \label{eq-oe2}
 g^\star(x) = \inf_{\nu \in \P(A_1(x))} \sup_{a^2 \in A_2(x)} \int_{\A_1} \int_{\X}  g^\star(y) \, q(dy \mid x, a^1, a^2) \, \nu(d a^1 \mid x), \qquad \forall \, x \in \X.
\end{equation} 
Similarly, (\ref{eq-oe2}) also holds for the average cost criteria $\tJ^{(i)}$, $2 \leq i \leq 4,$ discussed in Section~\ref{sec-4.1.1}.

As another example, for $\beta \in [0,1)$, consider the $\beta$-discounted cost criterion $J_\beta(\pi_1, \pi_2, x) : = \E^{\pi_1, \pi_2}_x \big[ \sum_{k=0}^\infty \beta^k c(x_k, a_k) \big]$ (which satisfies Assumption~\ref{cond-minmax-obj} for upper semianalytic $c(\cdot)$). Denote the optimal value function in the minimax control problem by $V_\beta^*$. 
By arguments similar to the above, we have the optimality equation
\begin{equation} \label{eq-oe3}
 V_\beta^*(x) =  \inf_{\nu \in \P(A_1(x))} \sup_{a^2 \in A_2(x)} \int_{\A_1} \int_{\X} \big\{ c(x, a^1, a^2) + \beta V^*_\beta(y)\big\} \, q(dy \mid x, a^1, a^2) \, \nu(d a^1 \mid x), \quad x \in \X
\end{equation} 
and also the optimality equations for the finite-horizon setting. From these equations it follows easily that within the class of bounded universally measurable functions on $\X$, $V_\beta^*$ is the unique solution to (\ref{eq-oe3}), and that starting from the constant zero function, value iteration (with the operator defined by the right-hand side of (\ref{eq-oe3})) converges to $V_\beta^*$ uniformly over $\X$.
\qed
\end{example}

\begin{example}[Optimality inequality] \rm
Consider the standard average cost criterion $J^{(1)}$ (cf.\ (\ref{eq-crt1})). 
For bounded $c(\cdot)$, using Fatou's lemma, we have that for any $\pi_1 \in \Pi_1, \pi_2 \in \Pi_2, x \in \X$,
$$ J^{(1)}(\pi_1, \pi_2, x) : =   \limsup_{n \to \infty} n^{-1} \E^{\pi_1, \pi_2}_x \Big[ \sum_{k=0}^{n-1} c(x_k, a_k) \Big] \leq 
\E^{\pi_1, \pi_2}_x \Big[ J^{(1)} \big(\pi^s_1(x_0, a_0^1), \, \pi^s_2(x_0, a_0), x_1 \big) \Big].$$
 Applying the argument used to derive (\ref{eq-oe-ineq0}) in the previous example with the above inequality in place of (\ref{eq-exoe1}), we obtain the optimality inequality
\begin{equation} \label{eq-oe4}
 g^\star(x) \leq \inf_{\nu \in \P(A_1(x))} \sup_{a^2 \in A_2(x)} \int_{\A_1} \int_{\X}  g^\star(y) \, q(dy \mid x, a^1, a^2) \, \nu(d a^1 \mid x), \qquad \forall \, x \in \X.
\end{equation} 

If $q(dy \mid x, a^1, a^2)$ depends on $(x, a^1)$ only so that player 2 cannot affect the state transitions, (\ref{eq-oe4}) is simplified to
$g^\star(x) \leq \inf_{a^1 \in A_1(x)} \int_{\X}  g^\star(y) \, q(dy \mid x, a^1)$ for all $x \in \X$.
In this special case one can further analyze the structure of $g^\star$ by using the submartingale property of the process $\{g^\star(x_n)\}_{n \geq 0}$ under any policy pair $(\pi_1, \pi_2)$, similarly to our previous work \cite{Yu22} on average-cost MDPs.
\qed
\end{example}

\section{Proofs} \label{sec-5}

This section collects some of our proofs for Sections~\ref{sec-3}-\ref{sec-4}.
To analyze strategic measures, we will use an approach discussed in Feinberg \cite[Remark 3.3]{Fei96}, which is to work with particular versions of conditional probability distributions that are jointly measurable as functions of both the conditioning variables and the probability measures themselves. According to \cite{Fei96}, this approach was suggested by Maitra and Sudderth. 
To prepare the ground for the analysis, we first introduce some definitions and lemmas in Section~\ref{sec-proof-0}. We then give the proofs for Sections~\ref{sec-3} and \ref{sec-4} in Sections~\ref{sec3-proof} and~\ref{sec4-proof}, respectively.

\subsection{Some Definitions and Useful Lemmas} \label{sec-proof-0}

First, let us specify several stochastic kernels that will be used in our proofs to characterize the various sets of strategic measures. 
Recall that for $n \geq 0$,
$h_n : = (x_0, a_0, x_1, a_1, \ldots, x_n)$ and $h'_n : = (x_0, a_0, \ldots, x_{n}, a_{n})$; their spaces are denoted by $H_n$ and $H'_n$, respectively.
\begin{itemize}[leftmargin=0.6cm,labelwidth=!]
\item Let $\nu_n(d a_n \,|\, h_n; p)$ be a Borel measurable stochastic kernel  on $\A$ given $H_n \times \P(\Omega)$ such that for each $p \in \P(\Omega)$, it is a conditional distribution of $a_n$ given $h_n$ w.r.t.\ $p$.  
\item Let $Q_n(d x_{n+1} \,|\, h'_n; p)$ be a Borel measurable stochastic kernel on $\X$ given $H'_n \times \P(\Omega)$ such that for each $p \in \P(\Omega)$, it is a conditional distribution of $x_{n+1}$ given $h'_n$ w.r.t.\ $p$.
\item Let $\kappa (d a \,|\, x; \gamma)$ be a Borel measurable stochastic kernel on $\A$ given $\X$ and $\P(\X \times \A)$ such that for each $\gamma \in  \P(\X \times \A)$, it is a conditional distribution of $a$ given $x$ w.r.t.\ $\gamma$.
\end{itemize}
The existence of these stochastic kernels follows from \cite[Cor.\ 7.27.1]{bs}, which ``decomposes'' a Borel measurable stochastic kernel on a product space into the product of two Borel measurable stochastic kernels. Specifically, by \cite[Cor.\ 7.27.1]{bs}, if $X, Y$, and $Z$ are Borel spaces and $\psi(d(y,z) \nmid x)$ is a Borel measurable stochastic kernel on $Y \times Z$ given $X$, then there exist Borel measurable stochastic kernels $\psi_1(d y \nmid x)$ and $\psi_2(dz \nmid x, y)$ such that for all $x \in X$, 
$$\psi(B \mid x) = \int_Y \int_Z \ind_B(y, z) \, \psi_2(d z \mid x, y) \, \psi_1(dy \mid x), \qquad B \in \B(Y \times Z).$$
To obtain $\nu_n(d a_n \,|\, h_n; p)$ (resp.\ $Q_n(d x_{n+1} \,|\, h'_n; p)$), we apply \cite[Cor.\ 7.27.1]{bs} with $X = \P(\Omega)$ and with $\psi$ being the mapping that maps $p \in \P(\Omega)$ to its marginal on the space $H'_n$ (resp.\ $H_{n+1}$), 
which is a Borel measurable stochastic kernel on $H'_{n} = H_n \times \A$ (resp.\ $H_{n+1} = H'_n \times \X$) given $\P(\Omega)$. 
To obtain $\kappa (d a \,|\, x; \gamma)$, we apply \cite[Cor.\ 7.27.1]{bs} to the identity mapping $\gamma \mapsto \gamma$ on $\P(\X \times \A)$, which is a Borel measurable stochastic kernel on $\X \times \A$ given $\P(\X \times \A)$. In all these cases, the desired stochastic kernel is given by $\psi_2$.
We will use the same reasoning later to define a few other stochastic kernels.

We now give two lemmas that will be applied frequently.
The first lemma is about the association between a strategic measure and a policy. 

\begin{lemma} \label{lem-gen2}
Let $p \in \P(\Omega)$ and $\pi := \{ \mu_n\}_{n \geq 0} \in \Pi$, and let $p_0$ be the marginal distribution of $x_0$ w.r.t.\ $p$. Then 
$p = \rho_{p_0}[\pi]$ if and only if for all $n \geq 0$,
\begin{equation} \label{eq-pol-strat-meas}
  Q_n(d x_{n+1} \mid h'_n; p) = q(d x_{n+1} \mid x_n, a_n) \ \ \text{and} \ \ \  \nu_n(d a_n \mid h_n ; p) = \mu_n(d a_n \mid h_n), \ \ \  \text{$p$-almost surely}. 
\end{equation}  
\end{lemma}

Let us clarify that in (\ref{eq-pol-strat-meas}) the random variables are $h'_n$, $(x_n, a_n)$, and $h_n$. By (\ref{eq-pol-strat-meas}) we mean that for $p$-almost all $\omega \in \Omega$, these two probability measures on the space of $x_{n+1}$ are equal: 
$Q_n(\cdot \,|\, h'_n(\omega); p) = q(\cdot \,|\, x_n(\omega), a_n(\omega))$, and so are these two probability measures on the space of $a_n$: $ \nu_n(\cdot \,|\, h_n(\omega) ; p) = \mu_n(\cdot \,|\, h_n(\omega))$. Similar notations will be used for other stochastic kernels involved in the proofs.

\begin{proof}[Proof of Lemma~\ref{lem-gen2}.]
If $p = \rho_{p_0}[\pi]$, then clearly (\ref{eq-pol-strat-meas}) holds for all $n \geq 0$. Conversely, using (\ref{eq-pol-strat-meas}), it can be verified by induction on $n$ that the marginals of $p$ and $\Pr^{\pi}_{p_0}$ on $\B(H'_n)$ are the same for all $n \geq 0$. Since $\cup_{n \geq 0} \B(H'_n)$ is an algebra generating $\B(\Omega)$, it follows that $p = \rho_{p_0}[\pi]$.
\end{proof} 

The next lemma will be used in proving the measurability of the sets of strategic measures induced by policies with various structures.
We need some notation and definitions before stating the lemma. If $X$ and $Y$ are Borel spaces, let $\P(X) \otimes \P(Y)$ be the set of all product probability measures on $\B(X \times Y)$ (that is, $\P(X) \otimes \P(Y) : = \{ \nu \times \nu' \mid \nu \in \P(X), \, \nu' \in \P(Y)\}$, where $\nu \times \nu'$ is the product measure associated with $\nu$ and $\nu'$ on the product $\sigma$-algebra $\B(X) \otimes \B(Y) = \B(X \times Y)$). Recall also that $\D(X)$ is the set of all Dirac measures on $X$.
We write $\omega \in \Omega$ in terms of its components also as $\omega = (\omega^1, \omega^2, \ldots)$. If $I$ is a subset of positive integers, let $\omega^{(I)}$ consist of those $i$th coordinates of $\omega$ for $i \in I$ (i.e., $\omega^{(I)} : = \{ \omega^i \}_{i \in I}$), and let $\Omega^{(I)}$ be the space of $\omega^{(I)}$. 

\begin{lemma} \label{lem-gen1}
Let $Z = \Omega^{(I)}$, $Z' = \Omega^{(I')}$, $Z'_1 =  \Omega^{(I'_1)}$, and $Z'_2 =  \Omega^{(I'_2)}$, where $I$ and $I'$ are two disjoint subsets of positive integers, and $I'_1$ and $I'_2$ form a partition of $I'$.
Let $\Psi_1(dz' \,|\, z ; p)$ and $\Psi_2(dz' \,|\, z ; p)$ be Borel measurable stochastic kernels on $Z'$ given $Z \times \P(\Omega)$.
Then these subsets of $\P(\Omega)$ are Borel:  
\begin{enumerate}[leftmargin=0.75cm,labelwidth=!]
\item[\rm (i)] $ B_1 : = \big\{ p \in \P(\Omega) \mid p \{ \Psi_1(\cdot \, | \, \omega^{(I)} ; p) = \Psi_2(\cdot \, |\, \omega^{(I)}; p) \} = 1 \big\}$;
\item[\rm (ii)] $ B_2 : = \big\{ p \in \P(\Omega) \mid p \{ \Psi_1(\cdot \, | \, \omega^{(I)} ; p) \in \D(Z') \} = 1 \big\}$;
\item[\rm (iii)] $B_3 : = \big\{ p \in \P(\Omega) \mid p \{ \Psi_1(\cdot \, | \, \omega^{(I)} ; p) \in \P(Z'_1) \otimes \P(Z'_2) \} = 1 \big\}$.
\end{enumerate}
Suppose that $\Theta$ is a Borel space and $\Psi_1$ and $\Psi_2$ depend also on a parameter $\theta \in \Theta$, that is, $\Psi_1(dz' \,|\, z ; p, \theta)$ and $\Psi_2(dz' \,|\, z ; p, \theta)$ are two Borel measurable stochastic kernels on $Z'$ given $Z \times \P(\Omega) \times \Theta$. Then the following subset of $\P(\Omega) \times \Theta$ is Borel:
\begin{enumerate}[leftmargin=0.75cm,labelwidth=!]
\item[\rm (iv)] $B_4 : = \big\{ (p, \theta) \in \P(\Omega) \times \Theta \mid p \{ \Psi_1(\cdot \, | \, \omega^{(I)} ; p, \theta) = \Psi_2(\cdot \, |\, \omega^{(I)}; p, \theta) \} = 1 \big\}$.
\end{enumerate}
\end{lemma}

\begin{proof}
Since $B_1$ is a special case of $B_4$ with $\Theta = \varnothing$, to prove that these two sets are Borel, it suffices to prove that $B_4$ is Borel.
Let $\{f_k\}_{k \geq 0}$ be a countable family of bounded Borel measurable functions on $Z'$ that form a measure determining class for $\P(Z')$. 
The set 
$E: = \big\{ (p, \theta, z) \in \P(\Omega) \times \Theta \times Z \mid \Psi_1(\cdot \,|\, z ; p, \theta) = \Psi_2(\cdot \,|\, z; p, \theta) \big\}$ 
consists of those $(p, \theta, z)$ such that
$\int_{Z'} f_k(z') \, \Psi_1(dz' \,|\, z ; p, \theta)  = \int_{Z'} \, f_k(z')  \, \Psi_2(dz' \,|\, z ; p, \theta)$ for all $k \geq 0$. 
These integrals are Borel measurable in $(p, \theta, z)$ by \cite[Prop.~7.29]{bs} (since the functions $f_k$ and the stochastic kernels $\Psi_1, \Psi_2$ are Borel measurable). Therefore, the set $E$ is Borel. Let $E(p, \theta)$ denote the $(p, \theta)$-section of $E$. Then $\phi(p, \theta): = p \{ \omega^{(I)} \in E(p, \theta) \}$ is a Borel measurable function of $(p, \theta)$ by \cite[Cor.~7.26.1]{bs}. Since $B_4 =  \{ (p, \theta) \in \P(\Omega) \mid \phi(p, \theta) =1 \}$, $B_4$ is Borel and hence $B_1$ is also Borel, as discussed earlier.

Consider now $B_2$. Since $\D(Z')$ is the image of $Z'$ under the homeomorphism $z' \mapsto \delta_{z'}$, it is a Borel subset of $\P(Z')$ by \cite[Sec.\ I.3, Cor.\ 3.3]{Par67}.
Let $E : = \big\{ (p, z) \in \P(\Omega) \times Z \mid \Psi_1(\cdot \,| \, z ; p)  \in \D(Z') \big\}$. Then $E$ is the preimage of the Borel set $\D(Z')$ under the Borel measurable mapping $(z, p) \mapsto \Psi_1(dz' \,|\, z ; p)$ and therefore, $E$ is Borel. The rest of the proof is similar to the above proof for $B_4$: with $E(p)$ being the $p$-section of $E$ and with $\phi(p) : = p \{ \omega^{(I)} \in E(p) \}$, 
we have that $\phi$ is Borel measurable by \cite[Cor.~7.26.1]{bs}, and hence the set $B_2 = \{ p \in \P(\Omega) \mid \phi(p) =1 \}$ is Borel.

For $B_3$, note that $\P(Z'_1) \otimes \P(Z'_2)$ is the image of $\P(Z'_1) \times \P(Z'_2)$ under the homeomorphism $(\nu, \nu') \mapsto \nu \times \nu'$ from $\P(Z'_1) \times \P(Z'_2)$ into $\P(Z')$, and therefore, $\P(Z'_1) \otimes \P(Z'_2)$ is a Borel subset of $\P(Z')$ by \cite[Sec.\ I.3, Cor.\ 3.3]{Par67}. 
The rest of the proof is the same as that for $B_2$ with $\P(Z'_1) \otimes \P(Z'_2)$ in place of $\D(Z')$.
\end{proof} 

We will use Lemma~\ref{lem-gen1}(ii) to characterize strategic measures induced by nonrandomized policies. Blackwell \cite{Blk76} used this technique for a similar purpose and attributed it to Sudderth \cite{Sud69}.
We will use Lemma~\ref{lem-gen1}(iii)-(iv) in the proofs for the minimax control problems studied in Section~\ref{sec-4}.

\subsection{Proofs for Section~\ref{sec-3}} \label{sec3-proof}

\subsubsection{Proof of Theorem~\ref{thm-strat-m}} \label{sec-proof-1}

We will prove a more general result than Theorem~\ref{thm-strat-m}. The result and the proof arguments will be applicable not only to MDPs but also to the partially observable problems and the minimax control problems studied in Section~\ref{sec-4}.

Any strategic measure $p \in \S$ must satisfy two conditions: 
\begin{equation} \label{cond-strm1a}
p \big\{ (x_n, a_n) \in \Gamma \big\} = 1, \qquad \forall \, n \geq 0,
\end{equation}
\begin{equation} \label{cond-strm1b}
  Q_n(dx_{n+1} \mid h'_n; p) = q(d x_{n+1} \mid x_n, a_n), \ \ \  \forall \, n \geq 0, \qquad \text{$p$-almost surely},
\end{equation}
in accordance with the control constraint and the state transition dynamics in the MDP.

Now consider the strategic measures of policies with additional structural constraints. 
In particular, suppose that for each $n \geq 0$, a subset of the variables from $(x_0, a_0, \ldots, x_{n-1}, a_{n-1})$ is chosen and, together with $x_n$, they form the variable $\hat h_n$, whose space we denote by $\hat H_n$. 
Suppose also that we have a collection of constraints on the stochastic kernels $\{\mu_n\}_{n \geq 0}$ that form a policy. These constraints, one for each $\mu_n$, consist of either requirement (i) or both requirements given below:
\begin{itemize}[leftmargin=0.7cm,labelwidth=!]
\item[(i)] $\mu_n$ is a universally measurable stochastic kernel on $\A$ given $\hat H_n$; 
\item[(ii)] $\mu_n$ is nonrandomized; that is, $\mu_n(da_n \,|\, \hat h_n) \in \D(\A)$ for all $\hat h_n \in \hat H_n$.
\end{itemize}
If $\C$ is such a collection of constraints, we write $\Pi_{\C}$ for the set of all policies that satisfy $\C$, and we write $\S_{\C}$ for the set of strategic measures induced by those policies. Note that $\Pi, \Pi', \Pi_{sm}, \Pi'_{sm}$, as well as $\Pi_m$ and $\Pi'_m$, can all be viewed as $\Pi_{\C}$ for some $\C$.

Let $\hat \nu_n( da_n \nmid \hat h_n; p)$ be a Borel measurable stochastic kernel  on $\A$ given $\hat H_n \times \P(\Omega)$ such that for each $p \in \P(\Omega)$, it is a conditional distribution of $a_n$ given $\hat h_n$ w.r.t.\ $p$. (As discussed at the start of Section~\ref{sec-proof-0}, such a stochastic kernel exists.) 
A strategic measure $p \in \S_{\C}$ clearly satisfies that
\begin{equation} \label{cond-strm1c}
       \nu_n(da_n \mid h_n; p) = \hat \nu_n( da_n \mid \hat h_n; p)  \ \ \ \text{and}_{\,\C} \ \ \ \hat \nu_n( da_n \mid \hat h_n; p) \in \D(\A), \quad \forall \, n \geq 0, \ \  \text{$p$-almost surely},
\end{equation}
where by ``$\text{and}_{\,\C}$'' we mean that the condition ``$\hat \nu_n( da_n \nmid \hat h_n; p) \in \D(\A)$'' is present if and only if according to $\C$, $\mu_n$ is nonrandomized.

\begin{lemma} \label{lem-strm1}
The set $\S_{\C} = \big\{ p \in \P(\Omega) \mid p \ \text{satisfies (\ref{cond-strm1a})-(\ref{cond-strm1c})} \big\}$.
\end{lemma}

\begin{proof}
As discussed above, every $p \in \S_\C$ satisfies (\ref{cond-strm1a})-(\ref{cond-strm1c}). 
Consider now any $p \in \P(\Omega) $ that satisfies these constraints.
By a repeated application of \cite[Cor.\ 7.27.2]{bs} to decompose the marginals of $p$ on $H'_n, H_{n+1}$, $n \geq 0$, and taking into account (\ref{cond-strm1b}) and (\ref{cond-strm1c}), we can represent $p$ as the composition of its marginal $p_0(dx_0)$ on $H_0$ with a sequence of Borel measurable stochastic kernels: 
$$p_0(dx_0), \ \hat \nu_0(da_0 \mid \hat h_0), \ q(d x_1 \mid x_0, a_0), \  \ldots, \ \hat \nu_{n}(da_{n} \mid \hat h_{n}),  \ q(d x_{n+1} \mid x_{n}, a_n), \ \ldots$$
where $\hat \nu_n(da_n \,|\, \hat h_n) = \hat \nu_n(d a_n \,|\, \hat h_n; p)$. 
(In other words, $p$ coincides with the unique probability measure on $\B(\Omega)$ determined by the above sequence.) The stochastic kernels $\{\hat \nu_n\}$ need not satisfy the control constraint in the MDP; i.e., the sets 
$D_n : = \{ \hat h_n \in \hat H_n \mid \hat \nu_n( A(x_{n}) \,|\, \hat h_n) < 1\big\}$, $n \geq 0$, need not be empty.
They also need not fully satisfy the structural constraints specified by $\C$. 
In particular, the following sets $E_n$, $n \geq 0$, need not be empty:  
$E_n : = \{ \hat h_n \in \hat H_n \mid \hat \nu_n( \cdot \,|\, \hat h_n) \not\in \D(\A) \big\}$ if according to $\C$, $\mu_n$ must be nonrandomized; and $E_n: = \varnothing$ otherwise.

We now modify $\hat \nu_n$ on $D_n \cup E_n$ for each $n \geq 0$ to construct a policy in $\Pi_\C$.
Since the stochastic kernel $\hat \nu_n$ is Borel measurable and the graph $\Gamma$ of $A(\cdot)$ is analytic, by \cite[Prop.\ 7.46]{bs} $\hat \nu_n( A(x_{n}) \,|\, \hat h_n) = \int_\A \ind_\Gamma(x_n, a_n) \, \hat \nu_n( da_n \,|\, \hat h_n)$ is a universally measurable function of $\hat h_n$, and hence the set $D_n$ is universally measurable. The proof for Lemma~\ref{lem-gen1}(ii) shows that the set $E_n$ is Borel measurable. Therefore, $D_n \cup E_n$ is universally measurable.
Since $p$ satisfies (\ref{cond-strm1a}) and (\ref{cond-strm1c}), we must have 
\begin{equation} \label{eq-lem-strm1-prf1}
    p \{ \hat h_n \in D_n \cup E_n \} = 0, \qquad \forall \,  n \geq 0.
\end{equation}   
Now define a policy $\pi: = \{\mu_n\}_{n \geq 0} \in \Pi_\C$ as follows: for each $n \geq 0$, let 
$$ \mu_n(\cdot \mid \hat h_n) : = \begin{cases} 
  \hat \nu_n(\cdot \mid \hat h_n), & \text{if} \  \hat h_n \not \in D_n \cup E_n; \\
  \mu^o (\cdot \,|\, x_n), & \text{if} \   \hat h_n \in D_n \cup E_n,
  \end{cases}
$$
where $\mu^o$ is some fixed nonrandomized stationary policy in $\Pi'_s$.
In view of (\ref{cond-strm1c})-(\ref{eq-lem-strm1-prf1}), $p$-almost surely, $\mu_n(\cdot \,|\, \hat h_n) = \nu_n(\cdot \,|\, h_n; p)$ for all $n \geq 0$. Then by Lemma~\ref{lem-gen2} $p = \rho_{p_0}[\pi]$ and hence $p \in \S_\C$.
\end{proof} 

\begin{prop} \label{prp-Sc}
The set $\S_\C$ is analytic, and it is Borel if $\Gamma$ is Borel.
\end{prop}

\begin{proof}
For each $n \geq 0$, let $E_n : = \{p \in \P(\Omega) \mid \text{$p$ satisfies (\ref{cond-strm1a}) for the given $n$}\}$, $F_n : = \{p \in \P(\Omega) \mid \text{$p$ satisfies (\ref{cond-strm1b}) for the given $n$}\}$, and $G_n : = \{p \in \P(\Omega) \mid \text{$p$ satisfies (\ref{cond-strm1c}) for the given $n$}\}$.
The set $F_n$ is Borel by Lemma~\ref{lem-gen1}(i), which is applied with $\Psi_1$ and $\Psi_2$ corresponding to $Q_n(dx_{n+1} \nmid h'_n; p)$ and $q(d x_{n+1} \nmid x_n, a_n)$, respectively.
The set $G_n$ is Borel by Lemma~\ref{lem-gen1}(i)-(ii), which is applied with $\Psi_1$ and $\Psi_2$ corresponding to $\hat \nu_n( da_n \nmid \hat h_n; p)$ and $\nu_n(da_n \nmid h_n; p)$, respectively.
Consider now the set $E_n$. Condition (\ref{cond-strm1a}) is the same as that $p(D) = 1$ for the set $D : = (\X \times \A)^{n} \times \Gamma \times (\X \times \A)^\infty$. Since $\Gamma$ is analytic, $D$ is an analytic subset of $\Omega$ \cite[Prop.~7.38]{bs}. Then, by \cite[Prop.~7.43]{bs}, $E_n = \{ p \in \P(\Omega) \mid p(D) = 1\}$ is analytic. If $\Gamma$ is Borel, then $D$ is Borel and hence by \cite[Prop.~7.25]{bs} $E_n$ is also Borel.
Since $\S_\C = \cap_{n \geq 0} ( E_n \cap F_n  \cap G_n)$ by Lemma~\ref{lem-strm1}, the desired conclusion follows.
\end{proof} 

For $n \geq 0$, let $\gamma_n : \P(\Omega) \to \P(\X \times \A)$ with $\gamma_n(p)$ being the marginal of $p$ on the space of $(x_n, a_n)$. Recall the stochastic kernel $\kappa$ defined in Section~\ref{sec-proof-0}. In proving Lemma~\ref{lem-strm1} and Prop.~\ref{prp-Sc}, we have proved the following corollary, which establishes Theorem~\ref{thm-strat-m} for $\S_\star \in\{ \S, \S_m\}$:

\begin{cor} \label{cor-S-Sm}
The sets $\S, \S'$, $\S_m$, and $\S_m'$ can be expressed as:
\begin{enumerate}[leftmargin=0.72cm,labelwidth=!]
\item[\rm (i)] $\S = \big\{ p \in \P(\Omega) \mid p \ \text{satisfies (\ref{cond-strm1a})-(\ref{cond-strm1b})} \big\}$;
\item[\rm (ii)] $\S_m = \S \cap M$, where $M$ is the Borel set defined by
$$M : = \left\{p \in \P(\Omega) \,\big|\,   p \big\{ \nu_n( d a_n \mid h_n; p) = \kappa \big( da_n \mid x_n; \gamma_n(p) \big), \, \forall \, n \geq 0 \big\} = 1 \right\};$$
\item[\rm (iii)] $\S' = \S \cap M'$ and $\S'_m = \S_m \cap M'$, where $M'$ is the Borel set defined by  
$$M' : = \left\{ p \in \P(\Omega) \,\big|\, p \big\{  \nu_n(d a_n \mid h_n ; p) \in \D(\A), \, \forall \, n \geq 0 \big\} = 1 \right\}.$$
\end{enumerate}
These sets of strategic measures are analytic; they are Borel if $\Gamma$ is Borel.
\end{cor}

We now consider the set $\S_s$. As in Feinberg \cite[Proof of Lem.\ 3.1]{Fei96}, we consider the mapping
$\tilde \gamma :  \P(\Omega) \to \P(\X \times \A)$ defined by $\tilde \gamma(p) : = \sum_{n = 0}^\infty 2^{-n-1} \gamma_n(p)$. 
Since $\gamma_n$ is Borel measurable for all $n \geq 0$, the mapping $\tilde \gamma$ is Borel measurable. 
Consequently, 
the mapping $(x, p ) \mapsto \kappa \big( da \,|\, x ; \tilde \gamma(p) \big)$ is also Borel measurable (since the stochastic kernel $\kappa$, defined in Section~\ref{sec-proof-0}, is Borel measurable, and compositions of Borel mappings are Borel measurable); in other words, it is a Borel measurable stochastic kernel on $\A$ given $\X \times \P(\Omega)$.

\begin{lemma} \label{lem-Ss}
We have $\S_s = \S \cap \widetilde M$ and $\S_s' = \S_s \cap M'$, where the set 
$$\widetilde M : = \left\{p \in \P(\Omega) \,\big|\,   p \big\{ \nu_n( d a_n \mid h_n; p) = \kappa \big( da_n \mid x_n; \tilde \gamma(p) \big), \forall n \geq 0 \big\} = 1 \right\}$$
is Borel, and $M'$ is the Borel set given in Cor.~\ref{cor-S-Sm}. 
Both $\S_s$ and $\S_s'$ are analytic; they are Borel if $\Gamma$ is Borel.
\end{lemma}

\begin{proof}
Let $B_n: =  \left\{p \in \P(\Omega) \,\big|\,  p \big\{ \nu_n( d a_n\,|\, h_n; p) = \kappa \big( da_n \,|\, x_n; \tilde \gamma(p) \big) \big\} = 1 \right\}$ for $n \geq 0$. Then the set $\widetilde M = \cap_{n \geq 0} \, B_n$.
By applying Lemma~\ref{lem-gen1}(i) to each set $B_n$, with $\Psi_1$ and $\Psi_2$ corresponding to $\nu_n( d a_n \,|\, h_n; p)$ and $\kappa \big( da_n \,|\, x_n; \tilde \gamma(p) \big)$, respectively, we obtain that $B_n$ is Borel. Hence the set $\widetilde M$ is Borel.

Clearly $\S_s \subset \S \cap \widetilde M$ and $\S_s' \subset \S \cap \widetilde M \cap M'$. 
The proof of the opposite inclusions is similar to that of Lemma~\ref{lem-strm1}.
If $p \in \S \cap \widetilde M$, we define a stationary policy $\mu \in \Pi_s$ by
\begin{equation} \label{eq-prf-Ss-pol}
 \mu(d a \mid x ) : =  \begin{cases} 
    \kappa \big( da \mid x; \tilde \gamma(p) \big) &   \text{if} \ x \not\in D; \\
      \mu^o(d a \mid x)   &  \text{if} \ x \in D,
      \end{cases}
\end{equation}
where $\mu^o \in \Pi_s$ and the set $D$ is defined by $D : = \big\{ x \in \X \mid \kappa \big( A(x) \,|\, x; \tilde \gamma(p) \big) < 1 \big\}$. If $p \in \S \cap \widetilde M \cap M'$, then we define a nonrandomized stationary policy $\mu \in \Pi'_s$ by (\ref{eq-prf-Ss-pol}) with $\mu^o \in \Pi'_s$ and with the set $D \cup E$ in place of $D$, 
where the set $E : = \big\{ x \in \X \mid \kappa \big( da \mid x; \tilde \gamma(p) \big) \not\in \D(\A)  \big\}$. 
The set $D \in  \U(\X)$ by \cite[Prop.\ 7.46]{bs} (because the mapping $(x, p) \mapsto \kappa \big( da \,|\, x; \tilde \gamma(p) \big)$ is Borel measurable and the graph $\Gamma$ of $A(\cdot)$ is analytic), and the set $E$ is Borel by the proof of Lemma~\ref{lem-gen1}(ii). So $\mu$ as defined is a valid universally measurable policy.
In view of Cor.~\ref{cor-S-Sm}(i) and (\ref{cond-strm1a}), if $p \in \S \cap \widetilde M$, then $p \{ x_n \in D, \, \forall \, n \geq 0\} = 0$, and if $p \in \S \cap \widetilde M \cap M'$, then $p \{ x_n \in D \cup E, \, \forall \, n \geq 0\} = 0$. It follows that in either case, $p$-almost surely, $\mu(da_n \,|\,  x_n) = \nu_n( d a_n \,|\,  h_n; p)$ for all $n \geq 0$, so by Lemma~\ref{lem-gen2}, $p = \rho_{p_0}[\mu]$, where $p_0$ is the distribution of $x_0$ w.r.t.\ $p$. This proves $\S_s = \S \cap \widetilde M$ and $\S_s' = \S_s \cap M'$. The last statement of the lemma then follows from Cor.~\ref{cor-S-Sm} and the Borel measurability of $\widetilde M$ and $M'$.
\end{proof} 

Theorem~\ref{thm-strat-m} has now been proved.

\subsubsection{Proof of Proposition~\ref{prp-ummap-pol}} \label{sec-proof-2}

First, we prove a result that entails Prop.~\ref{prp-ummap-pol} except for the cases $\S_s$ and $\S'_s$.
Consider again the set $\Pi_\C$ of structured policies and the associated set $\S_\C$ of strategic measures introduced in Section~\ref{sec-proof-1}.
Let $\tilde \S_\C : = \big\{ (x, p) \in \X \times \P(\Omega) \,\big|\,  p \in \S_\C, \, p_0(p) = \delta_x \big\}$ (recall that $p_0(p)$ is the initial distribution of $x_0$ w.r.t.\ $p$). 
By Prop.~\ref{prp-Sc} and the proof of Lemma~\ref{lem-ext-set}, $\tilde \S_\C$ is analytic, and it is Borel if $\Gamma$ is Borel. 
Denote the $x$-section of $\tilde \S_\C$ by $\tilde \S_\C(x)$. 

Also, corresponding to $\C$ and the associated spaces $\hat H_n$, $n \geq 0$, 
we derive a collection $\Cs$ of constraints by replacing the condition that $\mu_n$ depends on $\hat h_n \in \hat H_n$ by the condition that $\mu_n$ depends on both $\hat h_n$ and the initial state $x_0$, if $\hat h_n$ does not already include $x_0$. 
The collection of variables in $\hat h_n$ and $\{x_0\}$ is denoted by $\hat h_n^s$, and their space is denoted by $\hat H^s_n$.    
The set $\Pi_{\Cs}$ is the set of all policies satisfying the constraints in $\Cs$.

\begin{prop} \label{prp-Sc-ummap-pol}
Let $\zeta : \X \to \P(\Omega)$ be a universally measurable mapping such that $\zeta(x) \in \tilde \S_\C(x)$ for all $x \in \X$. Then there exists a policy $\pi \in \Pi_{\Cs}$ with $\zeta(x) = \rho_x[\pi]$ for all $x \in \X$. 
\end{prop}

\begin{proof}
The proof is similar to that of Lemma~\ref{lem-strm1}, except that here the parametric dependence on the initial state $x$ needs to be taken into account in the construction of the policy $\pi$.

To simplify notation, let $\tilde \nu_n ( da_n \,|\, \hat h_n; x) : =  \hat \nu_n\big(d a_n \,|\, \hat h_n; \zeta(x) \big)$ for $n \geq 0$. 
Note that as a $\P(\A)$-valued function of $(\hat h_n, x)$, $\tilde \nu_n ( da_n \,|\, \hat h_n; x)$ is the composition of the mapping $(\hat h_n, x) \mapsto \big(\hat h_n, \zeta(x)\big)$ with the mapping $(\hat h_n, p) \mapsto \hat \nu_n(d a_n \,|\, \hat h_n; p)$.  By \cite[Prop.\ 7.44]{bs} compositions of universally measurable functions on Borel spaces are universally measurable. Since $\zeta$ is universally measurable and the stochastic kernel $\hat \nu_n(da_n \,|\, \hat h_n; p)$ is Borel measurable, it then follows from \cite[Prop.\ 7.44]{bs} that the stochastic kernel $\tilde \nu_n(da_n \,|\, \hat h_n; x)$ is universally measurable. 

For each $n \geq 0$, let $D_n$ be the set of points $(\hat h_n, x)$ at which $\tilde \nu_n$ violates the control constraint:
$$D_n : = \{ (\hat h_n, x) \in \hat H_n \times \X \mid  \tilde \nu_n( A(x_n) \mid \hat h_n ; \, x) < 1 \}. $$
If the structural constraints specified by $\C$ require $\mu_n$ to be nonrandomized, let
$$ E_n : = \{ (\hat h_n, x) \in \hat H_n \times \X \mid  \tilde \nu_n( \cdot \mid \hat h_n ; \, x) \not\in \D(\A) \};$$
otherwise, let $E_n : = \varnothing$.
As discussed earlier, $\tilde \nu_n$ is universally measurable and the set $\D(\A)$ is Borel, while the graph $\Gamma$ of $A(\cdot)$ is analytic by assumption.
Consequently, the sets $D_n$ and $E_n$ are universally measurable by \cite[Prop.\ 7.46]{bs} and by \cite[Cor.\ 7.44.1]{bs}, respectively. 
Moreover, since $\zeta(x) \in  \tilde \S_\C(x)$, $\zeta(x)\{ x_0 = x \} = 1$ and $\zeta(x) \in \S_\C$. 
Then by Lemma~\ref{lem-strm1}, (\ref{cond-strm1a}), and (\ref{cond-strm1c}), $\zeta(x)\big\{  \hat \nu_n\big( A(x_n) \nmid \hat h_n ; \, \zeta(x)\big)  = 1 \} = 1$,
and if $\mu_n$ must be nonrandomized according to $\C$, we also have $\zeta(x)\big\{\hat \nu_n \big( \cdot \,|\, \hat h_n ; \, \zeta(x)\big) \in \D(\A) \big\} = 1$. 
In view of the definition of $\tilde \nu_n$, these relations imply
\begin{equation} \label{eq-ummap-prf1}
\zeta(x)\{ (\hat h_n, x_0) \in D_n \cup E_n \} = 0,  \qquad \forall \, n \geq 0, \ x \in \X. 
\end{equation}

Now define a policy $\pi : = \{ \mu_n\}_{n \geq 0} \in \Pi_\Cs$ as follows: for $n \geq 0$ and $\hat h^s_n \in \hat H^s_n$,
\begin{equation} \label{eq-ummap-prf2}
   \mu_n(d a_n \mid \hat h^s_n) : = \begin{cases} 
        \tilde  \nu_n(d a_n \mid \hat h_n; \, x_0), & \text{if} \ (\hat h_n, x_0) \not\in D_n \cup E_n; \\
           \mu^o(d a_n \mid x_n), & \text{if} \ (\hat h_n, x_0)  \in D_n \cup E_n,
           \end{cases}        
\end{equation}
where $\mu^o$ is some fixed policy in $\Pi'_s$. Then by (\ref{cond-strm1c}), (\ref{eq-ummap-prf1}), and the definition of $\tilde \nu_n$, for each $x \in \X$, it holds $\zeta(x)$-almost surely that
$\mu_n(d a_n \,|\, \hat h^s_n) = \nu_n\big(d a_n \,|\, h_n; \zeta(x) \big)$ for all $n \geq 0$. By Lemma~\ref{lem-gen2} this implies $\zeta(x) = \rho_x[\pi]$ for all $x \in \X$. So $\pi$ is the desired policy in $\Pi_\Cs$.
\end{proof} 

Applying Prop.~\ref{prp-Sc-ummap-pol} with $\S_\C \in \{\S, \S', \S_m, \S'_m\}$, we obtain the conclusions of Prop.~\ref{prp-ummap-pol} for those four cases. 

The remaining two cases in Prop.~\ref{prp-ummap-pol} are $\S_\star = \S_s$ and $\S_\star = \S'_s$, with the graph of $\zeta$ contained in $\tilde \S_\star$. 
For these two cases, we construct the desired semi-stationary policy from the stochastic kernel $\kappa\big(da \,|\, x'; \tilde \gamma(\zeta(x)) \big)$, where the function $\tilde \gamma$ is as defined before Lemma~\ref{lem-Ss}. The construction is similar to those in the proofs of Lemma~\ref{lem-Ss} and Prop.~\ref{prp-Sc-ummap-pol}. 
In particular, let $\tilde \kappa(da \,|\, x' ; x) : = \kappa\big(da \,|\, x'; \tilde \gamma(\zeta(x)) \big)$.
Since $\zeta$ is universally measurable and the mapping $\tilde \gamma$ and the stochastic kernel $\kappa$ are Borel measurable, the stochastic kernel $\tilde \kappa$ is universally measurable by \cite[Prop.\ 7.44]{bs}.
Let 
$$ D : =  \big\{ (x', x) \in \X^2 \mid \tilde \kappa\big(A(x')  \mid x'; x \big) < 1 \big\}.$$
In the case $\S_\star = \S_s$, let $E : = \varnothing$, and in the case $\S_\star = \S'_s$, let
$$ E : = \big\{ (x', x) \in \X^2 \mid \tilde \kappa\big( d a  \mid x'; x \big)  \not\in \D(\A) \big\}.$$
Then $D, E \in \U(\X^2)$ by \cite[Prop.\ 7.46, Cor.\ 7.44.1]{bs} (since $\tilde \kappa$ is universally measurable, the graph $\Gamma$ of $A(\cdot)$ is analytic, and the set $\D(\A)$ is Borel).  
Moreover, for all $x \in \X$, since $\zeta(x) \in  {\tilde \S}_s(x)$, we have $\zeta(x) \{ x_0 = x\} = 1$ and $\zeta(x)\{ \tilde \kappa( A(x_n) \nmid x_n; x) = 1, \, \forall \, n \geq 0 \} = 1$ by (\ref{cond-strm1a}) and Lemma~\ref{lem-Ss}. 
In the case $\zeta(x) \in  {\tilde \S}'_s(x)$, we also have $\zeta(x)\{ \tilde \kappa ( d a_n \nmid x_n; x) \in \D(\A) , \, \forall \, n \geq 0\} = 1$ by Lemma~\ref{lem-Ss}.
Therefore,
\begin{equation} \label{eq-ummap-prf5}
\zeta(x)\{ (x_n, x_0) \in D \cup E \} = 0,  \qquad \forall \, n \geq 0, \ x \in \X. 
\end{equation}

Now define a semi-stationary policy $\mu$ as follows: for $(x_0, x') \in \X^2$,
\begin{equation} \label{eq-ummap-prf6}
   \mu(d a \mid x_0, x') : = \begin{cases} 
        \tilde  \kappa(d a \mid x' ; \, x_0), & \text{if} \ (x', x_0) \not\in D \cup E; \\
           \mu^o(d a \mid x'), & \text{if} \ (x', x_0)  \in D \cup E,
           \end{cases}        
\end{equation}
for some fixed $\mu^o \in \Pi'_s$. In the case $\S_\star = \S'_s$, this policy $\mu$ is nonrandomized.
By (\ref{eq-ummap-prf5}), the definition of $\tilde \kappa$, and Lemma~\ref{lem-Ss}, we have that for each $x \in \X$, it holds $\zeta(x)$-almost surely that
$\mu(d a_n \nmid x_0,  x_n) = \nu_n\big(d a_n \,|\, h_n; \zeta(x) \big)$ for all $n \geq 0$. By Lemma~\ref{lem-gen2} this implies $\zeta(x) = \rho_x[\mu]$ for all $x \in \X$, so $\mu$ is the desired semi-stationary policy in the case $\S_\star = \S_s$ or $\S'_s$.

The proof of Prop.~\ref{prp-ummap-pol} is now complete.

\subsubsection{Proof of Theorem~\ref{thm-rep-nonrand}} \label{sec-proof-3}

In what follows we prove a more general version of Theorem~\ref{thm-rep-nonrand} with $\Pi$ replaced by the set $\Pi_\C$ of structured policies introduced in Section~\ref{sec-proof-1}. As mentioned earlier, our proof is similar to that of Feinberg \cite[Thm.\ 1]{Fei82a} and makes use of an important lemma in Gikhman and Skorokhod~\cite[Lem.~1.2]{GiS79}, in addition to the properties of universally measurable sets and functions.

Let $\lambda$ denote the Lebesgue measure on $[0,1]$.

\begin{lemma} \label{lem-nonrand1}
For each $\pi : = \{ \mu_n \}_{n \geq 0} \in \Pi_\C$, there exist universally measurable functions $f_n : [0, 1] \times \hat H_n \to \A$, $n \geq 0$, such that for all $\hat h_n \in \hat H_n$, 
$$  f_n(\theta, \hat h_n) \in A(x_n) \ \ \ \forall \, \theta \in [0,1], \quad \text{and} \quad \mu_n(d a_n \mid \hat h_n) = (\lambda \circ \phi_{\hat h_n}^{-1}) (da_n),$$
where $\lambda \circ \phi_{\hat h_n}^{-1}$ is the image measure of $\lambda$ under the mapping $\phi_{\hat h_n}: [0,1] \to \A$ given by $\phi_{\hat h_n} : = f_n(\cdot, \hat h_n)$.
\end{lemma}

\begin{proof}
For each $n \geq 0$, by \cite[Lem.\ 1.2]{GiS79}, there exists a $\B([0,1]) \times \U(\hat H_n)$-measurable function 
$f_n: [0, 1] \times \hat H_n \to \A$ that meets the requirements except that the set $E: = \big\{(\theta, \hat h_n) \in [0,1] \times \hat H_n \mid f_n(\theta, \hat h_n) \not\in A(x_n) \big\}$ need not be empty. We modify $f_n$ on $E$ to define another function $\tilde f_n$ as follows. 
Let $\tilde f_n : = f_n$ on $E^c$ and $\tilde f_n(\theta, \hat h_n) : = \mu^o(x_n)$ on $E$, for some analytically measurable selection $\mu^o$ of $\Gamma$. 
Then $\tilde f_n(\theta, \hat h_n) \in A(x_n)$ for all $\theta \in [0,1], \hat h_n \in \hat H_n$.
Let us show that $\tilde f_n$ satisfies the rest of the requirements, namely, the universal measurability of $\tilde f_n$ and the relation $\mu_n(d a_n \mid \hat h_n) = (\lambda \circ \tilde \phi_{\hat h_n}^{-1}) (da_n)$ for $\tilde \phi_{\hat h_n} : = \tilde f_n(\cdot, \hat h_n)$.

The set $E^c$ can be equivalently expressed as 
$$E^c = \big\{(\theta, \hat h_n) \in [0,1] \times \hat H_n \mid (x_n, f_n(\theta, \hat h_n)\big) \in \Gamma \big\} = \psi^{-1}(\Gamma),$$
where $\psi$ denotes the mapping $(\theta, \hat h_n) \mapsto (x_n, f_n(\theta, \hat h_n))$. 
Since $\psi$ is universally measurable by the measurability of $f_n$ and the set $\Gamma$ is analytic by our model assumption, $E^c$ is universally measurable \cite[Cor.\ 7.44.1]{bs}. This, together with the measurability of $f_n$, shows that $\tilde f_n$ is universally measurable. 

For all $\hat h_n \in \hat H_n$, since $\mu_n(d a_n \mid \hat h_n) = (\lambda \circ \phi_{\hat h_n}^{-1}) (da_n)$, we have 
\begin{equation} \label{eq-prf-nonrand1}
\lambda \big(\phi_{\hat h_n}^{-1}(A(x_n))\big) = \mu_n(A(x_n) \mid \hat h_n) = 1.
\end{equation}
In the above, for the validity of the first equality, we also used the fact that the set $\phi_{\hat h_n}^{-1}(A(x_n))$ is universally measurable. To see this, note that by the measurability of $f_n$, $\phi_{\hat h_n}$ is universally measurable, and since $\Gamma$ is analytic, its $x_n$-section $A(x_n)$ is analytic by \cite[Prop.\ 7.40]{bs}. Then by \cite[Cor.\ 7.44.1]{bs} $\phi_{\hat h_n}^{-1}(A(x_n))$ is universally measurable. Now from (\ref{eq-prf-nonrand1}) it follows that
$$ \lambda(E_{\hat h_n}) = \lambda \big( \{ \theta \in [0,1] \mid f_n(\theta, \hat h_n) \not\in A(x_n) \} \big) = 0,$$
where $E_{\hat h_n}$ denotes the $\hat h_n$-section of the set $E$. Since $\tilde f_n(\cdot, \hat h_n)$ can only differ from $f_n(\cdot, \hat h_n)$ on $E_{\hat h_n}$, this implies that
$$ \lambda \big( \{ \theta \in [0,1] \mid \tilde f_n(\theta, \hat h_n) \in B \} \big) = \lambda \big( \{ \theta \in [0,1] \mid f_n(\theta, \hat h_n) \in B \} \big), \qquad \forall \,   B \in \B(\A),$$
or equivalently, expressed in terms of image measures, $\lambda \circ \tilde \phi_{\hat h_n}^{-1} = \lambda \circ \phi_{\hat h_n}^{-1}$, where $\tilde \phi_{\hat h_n} : = \tilde f_n(\cdot, \hat h_n)$. 
Since $\mu_n(d a_n \nmid \hat h_n) = (\lambda \circ \phi_{\hat h_n}^{-1}) (da_n)$, we have $\mu_n(d a_n \nmid \hat h_n) = \lambda \circ \tilde \phi_{\hat h_n}^{-1}(d a_n)$. Thus the function $\tilde f_n$ satisfies all the requirements.
\end{proof}

For each $\pi \in \Pi_\C$, consider the family of functions $\{f_n\}$ given in Lemma~\ref{lem-nonrand1}. Corresponding to each point $\bar \theta : = (\theta_0, \theta_1, \ldots) \in [0,1]^\infty$, the sequence of functions $\bar f(\bar \theta) : = \{ f_n(\theta_n, \cdot) \}_{n\geq 0}$ defines a nonrandomized universally measurable policy in $\Pi_\C$. To simplify notation, we denote this policy also by $\bar f(\bar \theta)$.

\begin{lemma} \label{lem-nonrand2}
Let $\pi \in \Pi_\C$, and consider the associated family of nonrandomized policies $\bar f(\bar \theta) \in \Pi_\C$, $\bar \theta \in [0,1]^\infty$, given by Lemma~\ref{lem-nonrand1}. The mapping $\psi :  \P(\X)  \times  [0,1]^\infty\to \P(\Omega)$ with
$\psi (p_0, \bar \theta) : = \rho_{p_0}[ \bar f(\bar \theta)]$ is universally measurable.
\end{lemma}

\begin{proof}
By \cite[Lem.\ 7.28(b)]{bs}, it suffices to show that for each $B \in \B(\Omega)$, $\rho_{p_0}[ \bar f(\bar \theta)](B)$ is universally measurable in $(p_0, \bar \theta)$. In turn, since those Borel sets of $\Omega$ with this property form a Dynkin system, by the Dynkin system theorem \cite[Prop.\ 7.24]{bs}, it suffices to consider the collection $\cup_{n \geq 0} \, \B(H'_n)$ of sets (which correspond to those cylindrical sets that generate $\B(\Omega)$) and to show that for each $B_n \in \B(H'_n)$, $\rho_{p_0}[ \bar f(\bar \theta)]\{ h'_n \in B_n\}$ is universally measurable in $(p_0, \bar \theta)$.
Now by \cite[Prop.\ 7.45]{bs} we can express $\rho_{p_0}[ \bar f(\bar \theta)]\{ h'_n \in B_n\}$ as the iterated integral
\begin{align}
 \int_\X  \int_\X \cdots  \int_\X  & \, \ind_{B_n}\big(h_n, f_n(\theta_n, \hat h_n)\big) \, q\big(dx_n \mid x_{n-1}, f_{n-1}(\theta_{n-1}, \hat h_{n-1}) \big)   \cdots 
  q\big(dx_1 \mid x_0, f_0(\theta_0, x_0) \big)  \, p_0(d x_0).  \label{eq-prf-nonrand2}
\end{align} 
The functions $f_k$, $k \leq n$, are universally measurable by Lemma~\ref{lem-nonrand1}, the stochastic kernel $q$ is Borel measurable by our model assumption, and the identity mapping on $\P(\X)$, $p_0 \mapsto p_0$, is also Borel measurable. It then follows from \cite[Cor.\ 7.44.3, Prop.\ 7.46]{bs} that the above integral is universally measurable in $(p_0, \theta_0, \theta_1, \ldots, \theta_n)$ and therefore universally measurable in $(p_0, \bar \theta)$.
\end{proof}

Recall that $\bar \lambda$ is the countable product of Lebesgue measures on $[0,1]$. 

\begin{lemma} \label{lem-nonrand3}
Let $\pi \in \Pi_\C$ and $\bar f(\bar \theta) \in \Pi_\C$, $\bar \theta \in [0,1]^\infty$, be as in Lemma~\ref{lem-nonrand2}.
Then for all $p_0 \in \P(\X)$,
\begin{equation} \label{eq-rep-nonrand2}
  \rho_{p_0}[ \pi](B) = \int_{[0,1]^\infty} \! \rho_{p_0}[ \bar f(\bar \theta)](B) \, \bar \lambda (d \bar \theta), \qquad \forall \, B \in \B(\Omega).
\end{equation}  
\end{lemma}

\begin{proof}
By Lemma~\ref{lem-nonrand2} and \cite[Lem.\ 7.28(b) or Prop.\ 7.46]{bs}, the integral on the right-hand side of (\ref{eq-rep-nonrand2}) is well-defined for each $B \in \B(\Omega)$. Thus the right-hand side of (\ref{eq-rep-nonrand2}) defines a probability measure on $\B(\Omega)$. To show that this measure coincides with $\rho_{p_0}[\pi]$, it suffices to show that
(\ref{eq-rep-nonrand2}) holds for all subsets of $\Omega$ of the form $\{ h'_n \in B_n\}$ with $B_n \in \B(H'_n)$ for some $n \geq 0$ (since these subsets form an algebra that generates $\B(\Omega)$). This can be verified directly by using the iterated-integral expression (\ref{eq-prf-nonrand2}) for $\rho_{p_0}[ \bar f(\bar \theta)]\{ h'_n \in B_n\}$, and by successively integrating out the variables $\theta_n, \theta_{n-1}, \ldots, \theta_0$ and applying the Fubini theorem together with the relations $\mu_k(d a_k \nmid \hat h_k) = (\lambda \circ \phi_{\hat h_k}^{-1}) (da_k), k \leq n$, provided by Lemma~\ref{lem-nonrand1}.
\end{proof}

Note that the relation (\ref{eq-rep-nonrand2}) is equivalent to (\ref{eq-rep-nonrand}) in Theorem~\ref{thm-rep-nonrand}(ii) (since a probability measure on $\U(\Omega)$ is uniquely determined by its restriction on $\B(\Omega)$). 
With Lemmas~\ref{lem-nonrand2}-\ref{lem-nonrand3}, we have thus proved Theorem~\ref{thm-rep-nonrand}(i)-(ii) with $\Pi_\C$ in place of $\Pi$. 
Letting $\Pi_\C = \Pi$ yields Theorem~\ref{thm-rep-nonrand}(i)-(ii), and letting $\Pi_\C = \Pi_m$ or $\Pi_{sm}$ yields the remaining assertions of Theorem~\ref{thm-rep-nonrand} regarding Markov or semi-Markov policies.

\subsection{Proofs for Section~\ref{sec-4}} \label{sec4-proof}

\subsubsection{Proofs of Proposition~\ref{prp-Si-pol} and Theorem~\ref{thm-pomdp-opt}} \label{sec-pomdp-proof}

In this subsection we consider the partially observable control problem defined in Section~\ref{sec-pomdp} and prove Prop.~\ref{prp-Si-pol} and Theorem~\ref{thm-pomdp-opt}. Because the proofs are almost identical to those given earlier for MDPs, we will only outline the main arguments. 

For the partially observable problem, let $h^o_n: = (x_0, z_0, a_0, \ldots, x_n)$, 
$h_n : = (x_0, z_0, a_0, \ldots, x_n, z_n)$,  and $h'_n : = (x_0, z_0, a_0, \ldots, x_n, z_n, a_n)$ for $n \geq 0$, and denote their spaces by $H^o_n$, $H_n$, and $H'_n$, respectively. For each $n \geq 0$, the stochastic kernels $\nu_n(d a_n  \,|\, h_n; p)$ and $Q_n(d x_{n+1} \,|\, h'_n; p)$ are as defined in Section~\ref{sec-proof-0} with the spaces $\Omega$, $H_n$, and $H'_n$ corresponding to those in the partially observable problem. In addition:
\begin{itemize}[leftmargin=0.65cm,labelwidth=!]
\item Let $O_n(d z_n \mid h^o_n; p )$ be a Borel measurable stochastic kernel on $\Z$ given $H^o_n \times \P(\Omega)$ such that for each $p \in \P(\Omega)$, it is a conditional distribution of $z_n$ given $h^o_n$ w.r.t.\ $p$.  
\item Let $\hat \nu_n(d a_n \,|\, i_n; p)$ be a Borel measurable stochastic kernel on $\A$ given $I_n \times \P(\Omega)$ such that for each $p \in \P(\Omega)$, it is a conditional distribution of $a_n$ given $i_n$ w.r.t.\ $p$.  
\end{itemize}
The existence of these stochastic kernels is ensured by \cite[Cor.\ 7.27.1]{bs}, as explained in Section~\ref{sec-proof-0}.

For $p \in \P(\Omega)$, consider these constraints:
\begin{align} 
p \{ (i_n, a_n ) \in \Gamma_n \} & = 1, \qquad \forall \, n \geq 0;  \label{cond-pomdp-strm1} \\
  Q_n(dx_{n+1} \mid h'_n; p) & = q(d x_{n+1} \mid x_n, a_n), \qquad  \forall \, n \geq 0, \ \  \text{$p$-almost surely};\label{cond-pomdp-strm2} \\
  O_n(d z_n \mid h^o_n; p ) & = \delta_{f(x_n)}(dz_n ),  \qquad  \forall \, n \geq 0, \ \  \text{$p$-almost surely};\label{cond-pomdp-strm3} \\
       \nu_n(da_n \mid h_n; p) & = \hat \nu_n( da_n \mid i_n; p),  \qquad \forall \, n \geq 0, \ \  \text{$p$-almost surely}; \label{cond-pomdp-strm4} \\
        \hat \nu_n( da_n \mid i_n; p) & \in \D(\A), \qquad \forall \, n \geq 0, \ \  \text{$p$-almost surely}.  \label{cond-pomdp-strm5}
\end{align}
Clearly, any strategic measure $p \in \imS$ satisfies (\ref{cond-pomdp-strm1})-(\ref{cond-pomdp-strm4}), and any $p \in \imSn$, being induced by a nonrandomized policy, satisfies (\ref{cond-pomdp-strm1})-(\ref{cond-pomdp-strm5}).
The lemma below follows from essentially the same line of argument for Lemma~\ref{lem-strm1}. 

\begin{lemma} \label{lem-pomdp-strm}
The set $\imS = \big\{ p \in \P(\Omega) \mid p \ \text{satisfies (\ref{cond-pomdp-strm1})-(\ref{cond-pomdp-strm4})} \big\}$. The set $\imSn = \big\{ p \in \P(\Omega) \mid p \ \text{satisfies (\ref{cond-pomdp-strm1})-(\ref{cond-pomdp-strm5})} \big\} = \imS \cap M'$, where the set $M'$ is as given in Cor.~\ref{cor-S-Sm}.
\end{lemma}

Using Lemma~\ref{lem-pomdp-strm} and the same reasoning given in the proofs of Prop.~\ref{prp-Sc}, Cor.~\ref{cor-S-Sm}, and Lemma~\ref{lem-ext-set}, we obtain that the sets $\imS, \imSn, \timS$, and $\timSn$ are analytic, and that they are Borel if $\Gamma_n, n \geq 0,$ are Borel.
This establishes Prop.~\ref{prp-Si-pol}(i). 
Proposition~\ref{prp-Si-pol}(ii) then follows from Lemma~\ref{lem-pomdp-strm} and essentially the same proof arguments for Prop.~\ref{prp-Sc-ummap-pol}. Finally, similarly to the proof of Theorem~\ref{thm-opt} or~\ref{thm-ac-basic}, Theorem~\ref{thm-pomdp-opt} is a consequence of Prop.~\ref{prp-Si-pol}, Lemmas~\ref{lem-ac-lsa},~\ref{lem-risk-lsa}, and \cite[Props.\ 7.47, 7.50]{bs}.

\subsubsection{Proofs of Propositions~\ref{prp-minmax-strm1}-\ref{prp-minmax-strm2} and~\ref{prp-minmax-strm3}} \label{sec-minmax-proof}

In order to reuse some of the results in Section~\ref{sec3-proof} and avoid repeating similar arguments, we view the strategic measures in the minimax control problem as the strategic measures in an MDP induced by policies with particular structures. Specifically, define a set-valued mapping $A(\cdot)$ on $\X$ by $A(x): = \{ a = (a^1, a^2) \in \A_1 \times \A_2 \mid a^1 \in A_1(x), \, a^2 \in A_2(x) \}, x \in \X$. Let $\Gamma$ be the graph of $A(\cdot)$; i.e., $\Gamma : = \{ (x, a) \in \X \times \A \mid (x, a^1) \in \Gamma_1, \,  (x, a^2) \in \Gamma_2 \}$. Note that $\Gamma$ is the preimage of $\Gamma_1 \times \Gamma_2$ under the homeomorphism $(x, a^1, a^2) \mapsto (x, a^1, x, a^2)$ from $\X \times \A$ into $\X \times \A_1 \times \X \times \A_2$. Therefore, $\Gamma$ is analytic (Borel) if $\Gamma_1$ and $\Gamma_2$ are analytic (Borel) by \cite[Props.\ 7.38, 7.40]{bs}.
The strategic measures in the minimax control problem can also be viewed as the strategic measures of a subset $\hat \Pi$ of policies in the MDP whose control constraints are specified by $\Gamma$, where $\hat \Pi$ consists of those policies $\pi : = \{\mu_n\}_{n \geq 0}$ such that each $\mu_n$ has the product form
\begin{equation} \label{eq-minmax-pi}
  \mu_n(da_n \mid h_n) = \mu_n^1(da_n^1 \mid i_n) \, \mu_n^2(da_n^2 \mid h_n),
\end{equation}
for some universally measurable stochastic kernel $\mu_n^1$ on $\A_1$ given $I_n$ and some universally measurable stochastic kernel $\mu_n^2$ on $\A_2$ given $H_n$. 

Recall that $\P(\A_1) \otimes \P(\A_2)$ denotes the set of all product probability measures on $\B(\A_1 \times \A_2)$.
Recall also that for $n \geq 0$, $\hat \nu^1_n(d a^1_n \,|\, i_n; p)$ is a Borel measurable stochastic kernel on $\A_1$ given $I_n \times \P(\Omega)$ such that for each $p \in \P(\Omega)$, it is a conditional distribution of $a^1_n$ given $i_n$ w.r.t.\ $p$. 
For $i = 1, 2$, let $\nu^i_n(da^i_n \,|\, h_n; p)$ be a Borel measurable stochastic kernel on $\A_i$ given $H_n \times \P(\Omega)$ such that for each $p \in \P(\Omega)$, it is a conditional distribution of $a^i_n$ given $h_n$ w.r.t.\ $p$. 
For $p \in \P(\Omega)$, consider these constraints:
\begin{align} 
  \nu_n(da_n \mid h_n; p) & \in  \P(\A_1) \otimes \P(\A_2),  \qquad \forall \, n \geq 0, \ \  \text{$p$-almost surely}; \label{cond-minmax-strm1} \\
  \nu^1_n(da^1_n \mid h_n; p) & =  {\hat \nu}^1_n( da^1_n \mid i_n; p),  \qquad \forall \, n \geq 0, \ \  \text{$p$-almost surely}. \label{cond-minmax-strm2} 
\end{align}  
 We now use them to characterize the set $\mmS$ of strategic measures in the minimax control problem.
  
\begin{lemma} \label{lem-minmax-strm}
The set $\mmS = \big\{ p \in \P(\Omega) \mid p \ \text{satisfies (\ref{cond-strm1a})-(\ref{cond-strm1b}) and (\ref{cond-minmax-strm1})-(\ref{cond-minmax-strm2})} \big\}$. 
\end{lemma}

\begin{proof}[Proof (outline)]
Clearly, every $p \in \mmS$ satisfies the constraints (\ref{cond-strm1a})-(\ref{cond-strm1b}) and (\ref{cond-minmax-strm1})-(\ref{cond-minmax-strm2}). The converse follows from the same line of proof for Lemma~\ref{lem-strm1}. Specifically, if $p$ satisfies these constraints, we define a pair of policies $\pi_i : = \{\mu_n^i\}_{n \geq 0} \in \Pi_i$, $i = 1, 2$, for the minimax control problem as follows.
For each $n \geq 0$, 
we let $\mu_n^1(da_n^1 \,|\, i_n) : =  {\hat \nu}^1_n( da^1_n  \,|\, i_n; p)$ except on a set $D^1_n \in \U(I_n)$ with $p\{ i_n \in D^1_n\} = 0$, and we  let $\mu_n^2 (da_n^2  \,|\, h_n) : = \nu^2_n(da^2_n  \,|\, h_n; p)$ except on a set $D^2_n \in \U(H_n)$ with $p\{ h_n \in D^2_n\} = 0$. These sets $D^1_n$ and $D^2_n$ consist of points where the control constraints are violated:  
$$D^1_n : =  \big\{ i_n \in I_n \mid {\hat \nu}^1_n( A_1(x_n)  \mid i_n; p) < 1 \big\}, \quad D^2_n : =  \big\{ h_n \in H_n \mid \nu^2_n( A_2(x_n)  \mid h_n; p) < 1 \big\}.$$
Then, with $f_i : \X \to \A_i$ being an analytically measurable selection of $\Gamma_i$, $i = 1, 2$, we let $\mu_n^1(da_n^1  \,|\, i_n) : = f_1(x_n)$ on $D^1_n$, and we let $\mu_n^2(da_n^2  \,|\, h_n) : = f_2(x_n)$ on $D^2_n$. 
Finally, by viewing $(\pi_1, \pi_2)$ as a single policy of the form (\ref{eq-minmax-pi}) in an MDP and noting that (\ref{cond-minmax-strm1}) implies $\nu_n (da_n \,|\, h_n; p) =  \nu^1_n(da^1_n \,|\, h_n; p) \, \nu^2_n(da^2_n \,|\, h_n; p)$ $p$-almost surely, we can apply Lemma~\ref{lem-gen2} to obtain that $p = \rho_x[\pi_1, \pi_2]$ and hence $p \in \mmS$.
\end{proof} 

\begin{proof}[Proof of Props.~\ref{prp-minmax-strm1} and \ref{prp-minmax-strm3} (outline)]
As Lemma~\ref{lem-gen1}(i) and the proof of Prop.~\ref{prp-Sc} showed, the set $\big\{ p \in \P(\Omega) \mid p \ \text{satisfies (\ref{cond-strm1a})-(\ref{cond-strm1b}) and (\ref{cond-minmax-strm2})} \big\}$ is analytic (Borel) if $\Gamma$ is analytic (Borel). For each $n \geq 0$, the set $\{  p \in \P(\Omega) \mid p \ \text{satisfies (\ref{cond-minmax-strm1})} \ \text{for the given} \ n \}$ is Borel by Lemma~\ref{lem-gen1}(iii). Combining these results with Lemma~\ref{lem-minmax-strm} yields that the set $\mmS$ is analytic, and that it is Borel if $\Gamma_1$ and $\Gamma_2$ are Borel. 
The desired conclusion for the set $\tmmS$ follows from the same proof for Lemma~\ref{lem-ext-set}. This proves Props.~\ref{prp-minmax-strm1}(i). 

Proposition~\ref{prp-minmax-strm1}(ii) follows from arguments similar to those given in the proofs of Prop.~\ref{prp-Sc-ummap-pol} and Lemma~\ref{lem-minmax-strm}. In particular, we construct the desired policy pair $\pi_i : = \{\mu_n^i\}_{n \geq 0} \in \Pi_i$, $i = 1, 2$ by modifying the stochastic kernels ${\hat \nu}^1_n( da^1_n  \,|\, i_n; \zeta(x))$ and $\nu^2_n(da^2_n  \,|\, h_n; \zeta(x))$ to satisfy the control constraints. This modification is similar to the one in the proof of Prop.~\ref{prp-Sc-ummap-pol} except that we also take into account the dependence of $\zeta(x)$ on the initial state $x$ as in the proof of Lemma~\ref{lem-minmax-strm}.

The proof of Prop.~\ref{prp-minmax-strm3} is almost the same as the above proof for Prop.~\ref{prp-minmax-strm1}(ii). The only extension we need is to replace $\zeta(x)$ by $\zeta(x, \theta)$ in the proof of Lemma~\ref{lem-minmax-strm} and in the preceding arguments, when defining the stochastic kernels $\{\mu_n^i[\theta]\}_{n \geq 0} \in \Pi_i$, $i = 1, 2$, that form the desired parametrized policies $\pi^i(\theta)$, for $\theta \in \Theta$.
\end{proof} 

\begin{proof}[Proof of Prop.~\ref{prp-minmax-strm2}.]
{\bf (i)} For $p, p' \in \P(\Omega)$, recall that condition (\ref{eq-minmax-rel1}) requires that ${\hat \nu}^1_n ( \cdot \,|\, i_n; p') =  {\hat \nu}^1_n( \cdot \mid i_n; p)$ $p'$-almost surely and condition (\ref{eq-minmax-rel2}) requires $\hat p_{I_n}(p')$ to be absolutely continuous w.r.t.\ $\hat p_{I_n}(p)$, for all $n \geq 0$. Let 
$E_n : = \big\{ (p, p') \in \P(\Omega)^2 \mid (p, p') \ \text{satisfies (\ref{eq-minmax-rel1}) for the given $n$} \big\}$ 
and
$F_n : =  \big\{ (p, p') \in \P(\Omega)^2 \mid (p, p') \ \text{satisfies (\ref{eq-minmax-rel2}) for the given $n$} \big\}$.  
We first prove that these sets are Borel. For the set $E_n$, this follows from applying Lemma~\ref{lem-gen1}(iv) with $(\theta, p)$ in the lemma corresponding to $(p,p')$ here and with $\Psi_1$ and $\Psi_2$ corresponding to $\hat \nu_n^1(da_n^1 \mid i_n; p')$ and $ \hat \nu_n^1(da_n^1 \mid i_n; p)$, respectively.  

Consider the set $F_n$. Note that the marginal $\hat p_{I_n}(p)$ of $p$ on $I_n$ is a Borel measurable function from $\P(\Omega)$ into $\P(I_n)$. For $\eta_1, \eta_2 \in \P(I_n)$, we write $\eta_1 \ll \eta_2$ if $\eta_1$ is absolutely continuous w.r.t.\ $\eta_2$, and we denote by $s(\eta_1, \eta_2)$ the singular part of $\eta_1$ w.r.t.\ $\eta_2$. By \cite[Thm.\ 2.10]{DuF64}, $s(\cdot, \cdot)$ is a Borel measurable function from $\P(I_n)^2$ into the space of nonnegative finite measures on $I_n$ (equipped with the topology of weak convergence), and therefore, the set 
$$B : = \big\{ (\eta_1, \eta_2) \in \P(I_n)^2 \mid \eta_1 \ll \eta_2 \big\} = \big\{ (\eta_1, \eta_2) \in \P(I_n)^2 \mid s(\eta_1, \eta_2)(I_n) = 0 \big\}$$
is Borel by \cite[Thms.\ 2.1, 2.10]{DuF64}. Since $F_n$ is the preimage of $B$ under the Borel measurable mapping $(p, p') \mapsto \big(\hat p_{I_n}(p'), \hat p_{I_n}(p) \big)$, it follows that $F_n$ is Borel.

Now $\hmmS = D \cap  \big(\X \times \cap_{n \geq 0} (E_n \cap F_n) \big)$, where 
$D : = \big\{ (x, p, p') \in \X \times \P(\Omega)^2 \mid (x, p) \in \tmmS, \, (x, p') \in \tmmS \big\}$.
Since $D$ is the preimage of the set $\tmmS \times \tmmS$ under the homeomorphism $(x, p, p') \mapsto (x, p, x, p')$ from $\X \times \P(\Omega)^2$ into $\X \times \P(\Omega) \times \X \times \P(\Omega)$, $D$ is analytic (Borel) if $\tmmS$ is analytic (Borel) by \cite[Props.\ 7.38, 7.40]{bs}.
Combining this with Prop.~\ref{prp-minmax-strm1}(i) and the Borel measurability of $E_n$ and $F_n$, it follows that $\hmmS$ is analytic (Borel) if $\Gamma_1$ and $\Gamma_2$ are analytic (Borel). This proves part (i) of the proposition.

\smallskip
\noindent{\bf (ii)} Let $p = \rho_x[\pi_1, \pi_2]$ for some $x \in \X$, $\pi_1 : = \{ \mu_n^1\}_{n \geq 0} \in \Pi_1$, and $\pi_2 \in \Pi_2$.  
It is clear that under Assumption~\ref{cond-minimax-abscont}, $\rho_x[\pi_1,  \tilde \pi_2] \in \hmmS(x,p)$ for all $\tilde \pi_2 \in \Pi_2$. Now consider an arbitrary $p' \in  \hmmS(x,p)$. Then $p' = \rho_x[\tilde \pi_1, \tilde \pi_2]$ for some $\tilde \pi_i = \{\tilde \mu_n^i\}_{n \geq 0} \in \Pi_i$, $i = 1, 2$, and furthermore, in view of (\ref{eq-minmax-rel1}), for all $n \geq 0$,
$$p' \big\{ \hat \nu_n^1(da_n^1 \mid i_n; p') = \tilde \mu_n^1(da_n^1 \mid i_n) \big\} = 1, \quad  p' \big\{ \hat \nu_n^1(da_n^1 \mid i_n; p') =  \hat \nu_n^1(da_n^1 \mid i_n; p)  \big\} = 1,$$
which implies $p' \big\{ \tilde \mu_n^1(da_n^1 \mid i_n) = \hat \nu_n^1(da_n^1 \mid i_n; p)  \big\} = 1$. Combining this with the relations
$$\hat p_{I_n}(p') \ll \hat p_{I_n}(p) \ \ \ \text{and} \ \ \ p \big\{ \hat \nu_n^1(da_n^1 \mid i_n; p) = \mu_n^1(da_n^1 \mid i_n) \big\} = 1, \qquad \forall \, n \geq 0,  $$ 
(where the first relation is from (\ref{eq-minmax-rel2})), we obtain that 
$p'\big\{ \tilde \mu_n^1(da_n^1 \mid i_n) = \mu_n^1(da_n^1 \mid i_n) \big\} = 1$ for all $n \geq 0$.
Since $\nu_n(da_n \,|\, h_n; p') = \tilde \mu_n^1(da_n^1 \,|\, i_n) \, \tilde \mu_n^2(d a_n^2 \,|\, h_n)$ $p'$-almost surely for all $n \geq 0$, this implies 
$$ \nu_n(da_n \mid h_n; p') = \mu_n^1(da_n^1 \mid i_n) \, \tilde \mu_n^2(d a_n^2 \mid h_n),\quad  \forall \, n \geq 0, \ \ \text{$p'$-almost surely}.$$
Then by viewing $(\pi_1, \tilde \pi_2)$ as a single policy of the form (\ref{eq-minmax-pi}) in an MDP and by applying Lemma~\ref{lem-gen2}, we have that $p'= \rho_x[\pi_1, \tilde \pi_2]$. This proves part (ii) of the proposition.
\end{proof} 

\subsubsection{Projective Class $\Sigma_2^1$ and Proof of Proposition~\ref{prp-ext-sel}} \label{sec-sel-proof}

First, we briefly explain some concepts and results from descriptive set theory \cite{Kec95} that will be needed in proving Prop.~\ref{prp-ext-sel}.
In descriptive set theory, the projective class $\Sigma_2^1$ consists of all sets $B$ such that for some Polish spaces $X$ and $Y$, 
$B \subset X$ and $B$ is the image of some coanalytic set in $Y$ under a continuous mapping from $Y$ into $X$ \cite[Sec.\ 37.A]{Kec95}. The class $\Sigma_2^1$ is closed under countable intersections and unions, and continuous preimages and images (in particular, projections) \cite[Prop.\ 37.1]{Kec95}. It is also closed under the Souslin operation and contains all analytically measurable sets, as well as all limit-measurable sets (cf.\ \cite[Sec.\ 37.A]{Kec95} and \cite[Prop.\ B.10]{bs}).

We call a set in the class $\Sigma_2^1$ a $\Sigma_2^1$ set. Besides the aforementioned basic facts, the following results about $\Sigma_2^1$ sets are important for our purpose.
Let $X$ and $Y$ be Polish spaces. For a function $f$, let $\grh f$ denote its graph.
\begin{enumerate}[leftmargin=0.7cm,labelwidth=!]
\item[\rm (a)]  A uniformization theorem of Kond{\^o} \cite[Cor.\ 38.7]{Kec95} asserts that
if $B$ is a $\Sigma_2^1$ set in $X \times Y$, 
then there exists a function $f : \text{proj}_X(B) \to Y$ with $\grh f \subset B$ such that $\grh f$ is a $\Sigma_2^1$ set.
\item[\rm (b)] Under the axiom of analytic determinacy, $\Sigma_2^1$ sets are universally measurable \cite[Thm.\ 36.20]{Kec95}.
\item[\rm (c)] Under the axiom of analytic determinacy, the function $f$ in (a) is universally measurable.
(This follows from (a) and (b), since for $E \in \B(Y)$, $f^{-1}(E) = \text{proj}_X\big(\grh f \cap X \times E\big)$ is a $\Sigma_2^1$ set.)
\end{enumerate}

We are now ready to prove Prop.~\ref{prp-ext-sel}. The line of argument is almost the same as that of \cite[Props.\ 7.47, 7.50]{bs} and similar to that of \cite[Cor.\ 1, Thm.\ 2]{BrP73}, except that we use the above results for $\Sigma_2^1$ sets instead of the properties of analytic or Borel sets used in those cited references.

\begin{proof}[Proof of Prop.~\ref{prp-ext-sel}]
{\bf (i)} For each $r \in \R$, the set $\{ (x, y) \in D \mid f(x, y) < r \}$ is a $\Sigma_2^1$ set. To see this, write it as $D \setminus F_r$, where $F_r : = \{ (x, y) \in D \mid f(x, y) \geq r \}$. Since $f$ is upper semianalytic, $F_r$ is analytic \cite[Lem.\ 7.30(1)]{bs}; by assumption $D$ is analytic. Therefore, $D \setminus F_r$ is a $\Sigma_2^1$ set.
Then the set 
$$E_r : = \{ x \in \text{proj}_X(D) \mid f^*(x) < r \} = \text{proj}_X \big( \{ (x, y) \in D \mid f(x, y) < r \} \big) 
$$ 
is also a $\Sigma_2^1$ set \cite[Prop.\ 37.1]{Kec95}, so by \cite[Thm.\ 36.20]{Kec95}, under the axiom of analytic determinacy, $E_r$ is universally measurable for all $r \in \R$.  By \cite[Thm.\ 4.1.6]{Dud02}, this implies that $f^*$ is universally measurable.

\smallskip
\noindent {\bf (ii)}
To prove the universal measurability of $E^*$, consider first the set
\begin{align}
   G : =  & \, \big\{ (x, y, b) \in X \times Y \times \bar \R \mid (x, y) \in D, \, f(x, y) \leq  b \big\} \notag \\
 =  & \, \cap_{k \geq 1} \cup_{r \in \bar \Q} \, \big\{ (x, y, b) \in X \times Y \times \bar \R \,\big|\, (x, y) \in D, \, f(x, y) \leq r, \, r \leq b + k^{-1} \big\}  \label{eq-sel-prf1}
\end{align} 
where $\bar \Q : = \Q \cup \{ - \infty, + \infty\}$ (recall that $\Q$ is the set of rational numbers). The preceding proof for part (i) showed also that $\{ (x, y) \in D \mid f(x, y) \leq r \}$ is a $\Sigma_2^1$ set for $r \in \bar \R$ (since it is $D$ for $r = +\infty$, and for $r < +\infty$, it can be expressed as a countable intersection of sets of the form $\{ (x, y) \in D \mid f(x, y) < r' \}$, $r' \in \R$). This implies that for each $k$ and $r$, the set in the right-hand side of (\ref{eq-sel-prf1}) is a $\Sigma_2^1$ set. Consequently, $G$ and $\text{proj}_{X \times \bar \R} (G)$ are $\Sigma_2^1$ sets and therefore, universally measurable under the axiom of analytic determinacy \cite[Thm.\ 36.20]{Kec95}. 

Next, define a mapping $T:  \text{proj}_X(D) \to X \times \bar \R$ by $T(x): = (x, f^*(x))$. 
The set $\text{proj}_X(D)$ is analytic since $D$ is analytic \cite[Prop.\ 7.39]{bs}. 
As proved in part (i), $f^*$ is universally measurable. Therefore, $T$ is universally measurable. 
Since the set $E^*$ can be expressed as
$$E^* = \big\{ x \in \text{proj}_X(D) \mid \exists \, y \in D_x \ \text{s.t.} \ f(x, y) = f^*(x) \big\}  = T^{-1} \big(\text{proj}_{X \times \bar \R} (G) \big),$$
in view of the universal measurability of $\text{proj}_{X \times \bar \R} (G)$ proved earlier, we have, by \cite[Cor.\ 7.44.1]{bs}, that $E^*$ is universally measurable.

We now prove the rest of the proposition, by constructing a function $\psi$ that satisfies the requirements for an arbitrary $\epsilon > 0$. To this end, we first define two other functions $\bar \psi$ and $\psi^*$ that will be needed in the construction. 

Since $D$ is analytic, by the Jankov-von Neumann selection theorem \cite[Prop.\ 7.49]{bs}, there exists an analytically measurable function $\bar \psi: \text{proj}_X(D) \to Y$ with $\grh \bar \psi \subset D$. 

As proved above, $G$ is a $\Sigma_2^1$ set, so by a uniformization theorem of Kond{\^o} \cite[Cor.\ 38.7]{Kec95}, there exists a function $\phi: \text{proj}_{X \times \bar \R} (G) \to Y$ with $\grh \phi \subset G$ such that $\grh \phi$ is a $\Sigma_2^1$ set. Then $\phi$ is universally measurable under the axiom of analytic determinacy, as discussed before the proof. 
Consequently, if $E^*$ is nonempty, the function $\psi^*: E^* \to Y$ with $\psi^*(x) : = \phi(x, f^*(x)) = (\phi  \circ  T)(x)$ is also universally measurable \cite[Prop.\ 7.44]{bs}. Note that by the definition of $\psi^*$, 
\begin{equation}  \label{eq-sel-prf2}
  \psi^*(x) \in D_x, \ \ \  f(x, \psi^*(x))= f^*(x), \qquad \forall \, x \in E^*.
\end{equation}

We now proceed to construct the desired function $\psi$.
For integers $k$, define sets
$$ F(k) : = \{ (x, y) \in D \mid f(x, y) < k \epsilon \}, \qquad B(k) : = \{ x \in \text{proj}_X(D) \mid (k-1) \epsilon \leq f^*(x) < k \epsilon \},$$
$$B(-\infty) : =  \{ x \in \text{proj}_X(D) \mid f^*(x) = - \infty \}, \qquad B(+\infty) : =  \{ x \in \text{proj}_X(D) \mid f^*(x) = + \infty \}.$$ 
As proved earlier in part (i), $F(k)$ is a $\Sigma^1_2$ set, and $f^*$ is universally measurable. So, under the axiom of analytic determinacy, all these sets are universally measurable. 

For each $k$, 
by a uniformization theorem of Kond{\^o} \cite[Cor.\ 38.7]{Kec95}, there exists a function
$\psi_k : \text{proj}_X\big(F(k)\big) \to Y$ with $\grh \psi_k \subset F(k)$ such that $\grh \psi_k$ is a $\Sigma_2^1$ set. As discussed before the proof, under the axiom of analytic determinacy, $\psi_k$ is universally measurable.
Let $k^*$ be an integer such that $k^* \leq - \epsilon^{-2}$. Now define $\psi :  \text{proj}_X(D) \to Y$ as follows.
Let $\psi(x) : = \psi^*(x)$ for $x \in E^*$. For $x \in \text{proj}_X(D) \setminus E^*$, let
$$
\psi(x) : = \begin{cases}
  \psi_k(x) & \text{if} \ x \in B(k), \\
  \bar \psi(x) & \text{if} \ x \in B(+\infty), \\
  \psi_{k^*}(x) & \text{if} \ x \in B(-\infty).  
  \end{cases}
$$
Then $\psi$ is universally measurable and $\grh \psi \subset D$, and moreover, in view of (\ref{eq-sel-prf2}) and the definitions of $B(k), B(+\infty)$, and $B(-\infty)$, $\psi$ satisfies (\ref{eq-ext-sel1})-(\ref{eq-ext-sel2}), as required. This completes the proof.
\end{proof} 

\begin{rem} \rm \label{rmk-sigma21}
As can be seen from the preceding proof, for the conclusions of Prop.~\ref{prp-ext-sel} to hold, it suffices that $D$ is a $\Sigma_2^1$ set and $f$ is such that $\{ (x, y) \in D \mid f(x, y) < r \}$ is a $\Sigma_2^1$ set for all $r \in \R$. (Essentially we only need to make a minor change in the preceding proof, which is to apply Kond{\^o}'s uniformization theorem to $D$ in the definition of the function $\bar \psi$.) One can draw a parallel between this generalization of Prop.~\ref{prp-ext-sel} and the result in \cite[Props.\ 7.47, 7.50]{bs}, in which the above two sets are assumed to be analytic. Although this is beyond our scope, we note that the similarity between these two results reflects the parallelism between $\Sigma_2^1$ sets and analytic sets (see \cite[Sec.\ 37]{Kec95} on the hierarchy of projective classes for more details).
\qed
\end{rem}

\section*{Acknowledgements}
The author is grateful to Professor Eugene Feinberg for pointing to several important references on strategic measures and stochastic games and valuable feedback on earlier versions of this work. The author also thanks Professor William Sudderth, for mentioning the early work \cite{MPS90} on Borel gambling problems, which also used Kond{\^o}'s uniformization theorem in its analysis; Professor Serdar Y{\"u}ksel, for helpful discussion on minimax control problems; and an anonymous reviewer, whose critical comments helped improve the paper. 
This research was supported, during different periods of time, by grants from DeepMind, Alberta Machine Intelligence Institute (AMII), and Alberta Innovates---Technology Futures (AITF).

\addcontentsline{toc}{section}{References} 
\bibliographystyle{apa} 
\let\oldbibliography\thebibliography
\renewcommand{\thebibliography}[1]{%
  \oldbibliography{#1}%
  \setlength{\itemsep}{0pt}%
}
{\fontsize{9}{11} \selectfont
\bibliography{strat_meas_BorelDP_bib}}

\begin{thebibliography}{}

\bibitem[\protect\astroncite{Arapostathis et~al.}{1993}]{ArB93}
Arapostathis, A., Borkar, V.~S., Fern{\'a}ndez-Gaucherand, E., Ghosh, M.~K.,
  and Marcus, S.~I. (1993).
\newblock Discrete-time controlled {Markov} processes with average cost
  criterion: A survey.
\newblock {\em SIAM J. Control Optim.}, 31(2):282--344.

\bibitem[\protect\astroncite{Ash}{1972}]{Ash72}
Ash, R. (1972).
\newblock {\em Real Analysis and Probability}.
\newblock Academic Press, New York.

\bibitem[\protect\astroncite{Balder}{1989}]{Bal89}
Balder, E.~J. (1989).
\newblock On compactness of the space of policies in stochastic dynamic
  programming.
\newblock {\em Stoch. Proc. Appl.}, 32:141--150.

\bibitem[\protect\astroncite{B{\"a}uerle and Rieder}{2014}]{BaR14}
B{\"a}uerle, N. and Rieder, U. (2014).
\newblock More risk-sensitive {Markov} decision processes.
\newblock {\em Math. Oper. Res.}, 39:105--120.

\bibitem[\protect\astroncite{Bertsekas and Shreve}{1978}]{bs}
Bertsekas, D.~P. and Shreve, S.~E. (1978).
\newblock {\em Stochastic Optimal Control: The Discrete Time Case}.
\newblock Academic Press, New York.

\bibitem[\protect\astroncite{Bierth}{1987}]{Bie87}
Bierth, K.~J. (1987).
\newblock An expected average reward criterion.
\newblock {\em Stoch. Proc. Appl.}, 26:123--140.

\bibitem[\protect\astroncite{Blackwell}{1968}]{Blk-borel}
Blackwell, D. (1968).
\newblock A {Borel} set not containing a graph.
\newblock {\em Ann. Math. Statist.}, 39:1345--1347.

\bibitem[\protect\astroncite{Blackwell}{1976}]{Blk76}
Blackwell, D. (1976).
\newblock The stochastic processes of {Borel} gambling and dynamic programming.
\newblock {\em Ann. Statist.}, 4:370--374.

\bibitem[\protect\astroncite{Blackwell et~al.}{1974}]{BFO74}
Blackwell, D., Freedman, D., and Orkin, M. (1974).
\newblock The optimal reward operator in dynamic programming.
\newblock {\em Ann. Probability}, 2(5):926--941.

\bibitem[\protect\astroncite{Blackwell and Ryll-Nardzewski}{1963}]{BlR63}
Blackwell, D. and Ryll-Nardzewski, C. (1963).
\newblock Non-existence of everywhere proper conditional distributions.
\newblock {\em Ann. Math. Statist.}, 34:223--225.

\bibitem[\protect\astroncite{Borkar}{2002}]{Bor02}
Borkar, V.~S. (2002).
\newblock Convex analytic methods in {Markov} decision processes.
\newblock In Feinberg, E.~A. and Shwartz, A., editors, {\em Handbook of
  {Markov} Decision Processes: Methods and Applications}, pages 347--375.
  Springer Science+Business Media, New York.

\bibitem[\protect\astroncite{Brown and Purves}{1973}]{BrP73}
Brown, L.~D. and Purves, R. (1973).
\newblock Measurable selections of extrema.
\newblock {\em Ann. Statist.}, 1:902--912.

\bibitem[\protect\astroncite{Cavazos-Cadena and Salem-Siva}{2010}]{CaS10}
Cavazos-Cadena, R. and Salem-Siva, F. (2010).
\newblock The discounted method and equivalence of average criteria for
  risk-sensitive {Markov} decision processes on {Borel} spaces.
\newblock {\em Appl. Math. Optim.}, 61:167--190.

\bibitem[\protect\astroncite{Dubins and Freedman}{1964}]{DuF64}
Dubins, L.~E. and Freedman, D. (1964).
\newblock Measurable sets of measures.
\newblock {\em Pacific J. Math.}, 14:1211--1222.

\bibitem[\protect\astroncite{Dudley}{2002}]{Dud02}
Dudley, R.~M. (2002).
\newblock {\em Real Analysis and Probability}.
\newblock Cambridge University Press, Cambridge.

\bibitem[\protect\astroncite{Dynkin and Yushkevich}{1979}]{DyY79}
Dynkin, E.~B. and Yushkevich, A.~A. (1979).
\newblock {\em Controlled {Markov} Processes}.
\newblock Springer, New York.

\bibitem[\protect\astroncite{Feinberg}{1980}]{Fei80}
Feinberg, E.~A. (1980).
\newblock An $\epsilon$-optimality control of a finite {Markov} chain with an
  average reward criterion.
\newblock {\em Theory Probab. Appl.}, 25(1):70--81.

\bibitem[\protect\astroncite{Feinberg}{1982a}]{Fei82}
Feinberg, E.~A. (1982a).
\newblock Controlled {Markov} processes with arbitrary numerical criteria.
\newblock {\em Theory Probab. Appl.}, 27(3):486--503.

\bibitem[\protect\astroncite{Feinberg}{1982b}]{Fei82a}
Feinberg, E.~A. (1982b).
\newblock Non-randomized {Markov} and semi-{Markov} strategies in dynamic
  programming.
\newblock {\em Theory Probab. Appl.}, 27(1):116--126.

\bibitem[\protect\astroncite{Feinberg}{1991}]{Fei91}
Feinberg, E.~A. (1991).
\newblock Non-randomized strategies in stochastic decision processes.
\newblock {\em Ann. Oper. Res.}, 29:315--332.

\bibitem[\protect\astroncite{Feinberg}{1996}]{Fei96}
Feinberg, E.~A. (1996).
\newblock On measurability and representation of strategic measures in {Markov}
  decision processes.
\newblock In {\em Statistics, Probability and Game Theory}, volume~30 of {\em
  Lecture Notes -- Monograph Series}, pages 29--43. IMS.

\bibitem[\protect\astroncite{Feinberg and Kasyanov}{2021}]{FK21}
Feinberg, E.~A. and Kasyanov, P.~O. (2021).
\newblock {MDPs} with setwise continuous transition probabilities.
\newblock {\em Oper. Res. Lett.}, 49:734--740.

\bibitem[\protect\astroncite{Feinberg et~al.}{2020}]{FKL20}
Feinberg, E.~A., Kasyanov, P.~O., and Liang, Y. (2020).
\newblock Fatou's lemma in its classic form and {Lebesgue's} convergence
  theorems for varying measures with applications to {MDPs}.
\newblock {\em Theory Probab. Appl.}, 65:270--291.

\bibitem[\protect\astroncite{Gikhman and Skorokhod}{1979}]{GiS79}
Gikhman, I.~I. and Skorokhod, A.~V. (1979).
\newblock {\em Controlled Stochastic Processes}.
\newblock Springer, New York.

\bibitem[\protect\astroncite{G{\"o}del}{1938}]{God38}
G{\"o}del, K. (1938).
\newblock The consistency of the axiom of choice and of the generalized
  continuum hypothesis.
\newblock {\em Proc. Natl. Acad. Sci. USA}, 24:556--557.

\bibitem[\protect\astroncite{Gonz\'{a}lez-Trejo et~al.}{2002}]{GHH02}
Gonz\'{a}lez-Trejo, J.~I., Hern\'{a}ndez-Lerma, O., and Hoyos-Reyes, L.~F.
  (2002).
\newblock Minimax control of discrete-time stochastic systems.
\newblock {\em SIAM J. Control Optim.}, 41(5):1626--1659.

\bibitem[\protect\astroncite{Hern\'{a}ndez-Lerma and Lasserre}{1996}]{HL96}
Hern\'{a}ndez-Lerma, O. and Lasserre, J.~B. (1996).
\newblock {\em Discrete-Time {Markov} Control Processes: Basic Optimality
  Criteria}.
\newblock Springer, New York.

\bibitem[\protect\astroncite{Hern\'{a}ndez-Lerma and Lasserre}{1999}]{HL99}
Hern\'{a}ndez-Lerma, O. and Lasserre, J.~B. (1999).
\newblock {\em Further Topics on Discrete-Time {Markov} Control Processes}.
\newblock Springer, New York.

\bibitem[\protect\astroncite{Hern\'{a}ndez-Lerma and Lasserre}{2002}]{HL02}
Hern\'{a}ndez-Lerma, O. and Lasserre, J.~B. (2002).
\newblock The linear programming approach.
\newblock In Feinberg, E.~A. and Shwartz, A., editors, {\em Handbook of
  {Markov} Decision Processes: Methods and Applications}, chapter~12. Springer
  Science+Business Media, New York.

\bibitem[\protect\astroncite{Ja\'{s}kiewicz}{2007}]{Jas07}
Ja\'{s}kiewicz, A. (2007).
\newblock Average optimality for risk-sensitive control with general state
  space.
\newblock {\em Ann. Appl. Probab.}, 17:654--675.

\bibitem[\protect\astroncite{Ja\'{s}kiewicz and Nowak}{2014}]{JaN14}
Ja\'{s}kiewicz, A. and Nowak, A.~S. (2014).
\newblock Robust {Markov} control processes.
\newblock {\em J. Math. Anal. Appl.}, 420:1337--1353.

\bibitem[\protect\astroncite{Ja\'{s}kiewicz and Nowak}{2018}]{JaN18}
Ja\'{s}kiewicz, A. and Nowak, A.~S. (2018).
\newblock Zero-sum stochastic games.
\newblock In Ba\c{s}ar, T. and Zaccour, G., editors, {\em Handbook of Dynamic
  Game Theory}, pages 215--279. Springer International Publishing AG,
  Switzerland.

\bibitem[\protect\astroncite{Kechris}{1995}]{Kec95}
Kechris, A.~S. (1995).
\newblock {\em Classical Descriptive Set Theory}.
\newblock Springer-Verlag, New York.

\bibitem[\protect\astroncite{Koellner}{2014}]{Koe14}
Koellner, P. (2014).
\newblock {Large Cardinals and Determinacy}.
\newblock In Zalta, E.~N., editor, {\em The {Stanford} Encyclopedia of
  Philosophy}. Metaphysics Research Lab, Stanford University, {S}pring 2014
  edition.
\newblock
  \url{https://plato.stanford.edu/archives/spr2014/entries/large-cardinals-determinacy/}.

\bibitem[\protect\astroncite{Krengel}{1985}]{Kre85}
Krengel, U. (1985).
\newblock {\em Ergodic Theorems}.
\newblock Walter de Gruyter, Berlin.

\bibitem[\protect\astroncite{Maitra et~al.}{1990}]{MPS90}
Maitra, A., Purves, R., and Sudderth, W. (1990).
\newblock Leavable gambling problems with unbounded utilities.
\newblock {\em Trans. Amer. Math. Soc.}, 320:543--567.

\bibitem[\protect\astroncite{Maitra and Sudderth}{1993}]{MS93}
Maitra, A. and Sudderth, W. (1993).
\newblock Borel stochastic games with limsup payoffs.
\newblock {\em Ann. Probab.}, 21:861--885.

\bibitem[\protect\astroncite{Maitra and Sudderth}{1998}]{MS98}
Maitra, A. and Sudderth, W. (1998).
\newblock Finitely additive stochastic games with {Borel} measurable payoffs.
\newblock {\em Int. J. Game Theory}, 27:257--267.

\bibitem[\protect\astroncite{Martin and Steel}{1988}]{MaS88}
Martin, D.~A. and Steel, J.~R. (1988).
\newblock Projective determinacy.
\newblock {\em Proc. Natl. Acad. Sci.}, 85:6582--6586.

\bibitem[\protect\astroncite{Nowak}{1985}]{Now85b}
Nowak, A.~S. (1985).
\newblock Universally measurable strategies in zero-sum stochastic games.
\newblock {\em Ann. Probab.}, 13(1):269--287.

\bibitem[\protect\astroncite{Nowak}{2010}]{Now10}
Nowak, A.~S. (2010).
\newblock On measurable minimax selectors.
\newblock {\em J. Math. Anal. Appl.}, 366:385--388.

\bibitem[\protect\astroncite{Parthasarathy}{1967}]{Par67}
Parthasarathy, K.~R. (1967).
\newblock {\em Probability Measures on Metric Spaces}.
\newblock Academic Press, New York.

\bibitem[\protect\astroncite{Prikry and Sudderth}{2016}]{PrS16}
Prikry, K. and Sudderth, W.~D. (2016).
\newblock Measurability of the value of a parametrized game.
\newblock {\em Int. J. Game Theory}, 45:675--683.

\bibitem[\protect\astroncite{Rieder}{1991}]{Rie91}
Rieder, U. (1991).
\newblock Non-cooperative dynamic games with general utility functions.
\newblock In Raghavan, T. E.~S., Ferguson, T.~S., Parthasarathy, T., and
  Vrieze, O.~J., editors, {\em Stochastic Games and Related Topics}, pages
  161--174. Kluwer, Dordrecht, The Netherlands.

\bibitem[\protect\astroncite{Sch{\"a}l}{1975}]{Sch75b}
Sch{\"a}l, M. (1975).
\newblock On dynamic programming: Compactness of the space of policies.
\newblock {\em Stoch. Proc. Appl.}, 3:345--364.

\bibitem[\protect\astroncite{Shapiro et~al.}{2021}]{SDR21}
Shapiro, A., Dentcheva, D., and Ruszczy\'{n}ski, A. (2021).
\newblock {\em Lectures on Stochastic Programming: Modeling and Theory}.
\newblock Society for Industrial and Applied Mathematics and Mathematical
  Optimization Society, Philadelphia, 3nd edition.

\bibitem[\protect\astroncite{Shreve and Bertsekas}{1978}]{ShrB78}
Shreve, S.~E. and Bertsekas, D.~P. (1978).
\newblock Alternative theoretical frameworks for finite horizon discrete-time
  stochastic optimal control.
\newblock {\em SIAM J. Control Optim.}, 16(6):953--978.

\bibitem[\protect\astroncite{Shreve and Bertsekas}{1979}]{ShrB79}
Shreve, S.~E. and Bertsekas, D.~P. (1979).
\newblock Universally measurable policies in dynamic programming.
\newblock {\em Math. Oper. Res.}, 4(1):15--30.

\bibitem[\protect\astroncite{Strauch}{1966}]{Str-negative}
Strauch, R.~E. (1966).
\newblock Negative dynamic programming.
\newblock {\em Ann. Math. Statist.}, 37:871--890.

\bibitem[\protect\astroncite{Sudderth}{1969}]{Sud69}
Sudderth, W. (1969).
\newblock On the existence of good stationary strategies.
\newblock {\em Trans. Amer. Math. Soc.}, 135:399--414.

\bibitem[\protect\astroncite{Vega-Amaya}{2018}]{VAm18}
Vega-Amaya, O. (2018).
\newblock Solutions of the average cost optimality equation for {Markov}
  decision processes with weakly continuous kernel: The fixed-point approach
  revisited.
\newblock {\em J. Math. Anal. Appl.}, 464:152--163.

\bibitem[\protect\astroncite{Wei et~al.}{2022}]{WFC22}
Wei, C., Fau{\ss}, M., and Chapman, M.~P. (2022).
\newblock {CVaR}-based safety analysis in the infinite time horizon setting.
\newblock In {\em Proc. American Control Conference}, pages 2863--2870,
  Piscataway, NJ. IEEE.

\bibitem[\protect\astroncite{Yu}{2015}]{Yu-tc15}
Yu, H. (2015).
\newblock On convergence of value iteration for a class of total cost {Markov}
  decision processes.
\newblock {\em SIAM J. Control Optim.}, 53(4):1982--2016.

\bibitem[\protect\astroncite{Yu}{2020}]{Yu20}
Yu, H. (2020).
\newblock Average cost optimality inequality for {Markov} decision processes
  with {Borel} spaces and universally measurable policies.
\newblock {\em SIAM J. Control Optim.}, 58(4):2469--2502.

\bibitem[\protect\astroncite{Yu}{2022}]{Yu22}
Yu, H. (2022).
\newblock On structural properties of optimal average cost functions in
  {Markov} decision processes with {Borel} spaces and universally measurable
  policies.
\newblock {\em J. Math. Anal. Appl.}, 509:125954.

\bibitem[\protect\astroncite{Y{\"u}ksel}{2020}]{Yuk20}
Y{\"u}ksel, S. (2020).
\newblock A universal dynamic program and refined existence results for
  decentralized stochastic control.
\newblock {\em SIAM J. Control Optim.}, 58(5):2711--2739.

\bibitem[\protect\astroncite{Y{\"u}ksel and Saldi}{2017}]{YuS17}
Y{\"u}ksel, S. and Saldi, N. (2017).
\newblock Convex analysis in decentralized stochastic control, strategic
  measures, and optimal solutions.
\newblock {\em SIAM J. Control Optim.}, 55(1):1--28.

\end{thebibliography}

\end{document}